\documentclass{amsart}

\usepackage{color}
\usepackage[dvipsnames]{xcolor}
\usepackage{pstricks,pst-node,graphicx}
\usepackage{amsmath,amsfonts,amssymb,relsize,ytableau}

\usepackage{fullpage}
\newtheorem{theorem}{Theorem}[section]
\newtheorem{lemma}[theorem]{Lemma}
\newtheorem{proposition}[theorem]{Proposition}
\newtheorem{corollary}[theorem]{Corollary}

\theoremstyle{definition}
\newtheorem{definition}[theorem]{Definition}

\theoremstyle{remark}

\numberwithin{equation}{section}

\newcommand{\st}[1]{\ensuremath\overline{#1}}
\newcommand{\texplus}{\raisebox{.4mm}{$\scriptscriptstyle{+}\!$}}
\newcommand{\texminus}{\raisebox{.4mm}{$\scriptscriptstyle{-}\!$}}

\newcommand{\ial}{\ensuremath\stackrel{i}{\longleftarrow}}
\newcommand{\iar}{\ensuremath\stackrel{i}{\longrightarrow}}

\newcommand{\B}{\ensuremath\mathcal{B}}
\newcommand{\e}{\ensuremath\mathbf{e}}
\newcommand{\Q}{\ensuremath\mathcal{Q}}

\newcommand{\wt}{\ensuremath\mathrm{wt}}

\newcommand{\SSYT}{\ensuremath\mathrm{SSYT}}
\newcommand{\SSHT}{\ensuremath\mathrm{SSHT}}
\newcommand{\SSST}{\ensuremath\mathrm{SSST}}

\newcommand{\Yam}{\ensuremath\mathrm{Yam}}

\newcommand{\sym}{\ensuremath\mathrm{sym}}
\newcommand{\rect}{\ensuremath\mathrm{rect}}

\newcommand{\upe}{\ensuremath \varepsilon}
\newcommand{\downf}{\ensuremath \varphi}

\newlength\cellsize 
\savebox2{%
\begin{picture}(12,12)
\put(0,0){\line(1,0){12}}
\put(0,0){\line(0,1){12}}
\put(12,0){\line(0,1){12}}
\put(0,12){\line(1,0){12}}
\end{picture}}

\newcommand\cellify[1]{\def\thearg{#1}\def\nothing{}%
\ifx\thearg\nothing
\vrule width0pt height\cellsize depth0pt\else
\hbox to 0pt{\usebox2\hss}\fi%
\vbox to 12\unitlength{
\vss
\hbox to 12\unitlength{\hss$#1$\hss}
\vss}}

\newcommand\tableau[1]{\vtop{\let\\=\cr
\setlength\baselineskip{-12000pt}
\setlength\lineskiplimit{12000pt}
\setlength\lineskip{0pt}
\halign{&\cellify{##}\cr#1\crcr}}}

\savebox4{% CIRCLE
\begin{picture}(12,12)
\put(6,6){\circle{12}}
\end{picture}}

\newcommand{\cir}[1]{\def\thearg{#1}\def\nothing{}%
\ifx\thearg\nothing\vrule width0pt height\cellsize depth0pt%
  \else\hbox to 0pt{\usebox4\hss}\fi%
  \vbox to 12\unitlength{\vss\hbox to 12\unitlength{\hss$#1$\hss}\vss}}

\definecolor{boxgray}{gray}{.8}
\newcommand{\cb}{\color{boxgray}\rule{1\cellsize}{1\cellsize}\hspace{-\cellsize}\usebox2}

\newcommand{\tb}{%
  \psset{unit=\cellsize}
  \begin{pspicture}(1,1)
    \psset{linewidth=0.1ex}
    \pscustom[fillstyle=solid,fillcolor=boxgray]{
      \pspolygon(0,0)(1,1)(0,1)}
  \end{pspicture}}

\newcommand{\xb}{%
  \psset{unit=\cellsize}
  \begin{pspicture}(1,1)
    \psset{linewidth=0.1ex}
    \pscustom[fillstyle=solid,fillcolor=boxgray]{
      \pspolygon(0,0)(0,1)(1,1)(1,0)}
    \psline(0,0)(1,1)
    \psline(1,0)(0,1)
  \end{pspicture}}

\usepackage[colorinlistoftodos]{todonotes}

\newcommand{\Sami}[1]{\todo[size=\tiny,inline,color=blue!30]{#1 \\ \hfill --- Sami}}

\definecolor{lblue}{rgb}{0.647, 0.749, 0.972}
\begin{document}

%%%%%%%%%%%%%%%%%%%%%%%%%%%%%%%%%%%%%%%%%%%%%%%%%%%%%%%%%%%%
%  TITLE PAGE information
%%%%%%%%%%%%%%%%%%%%%%%%%%%%%%%%%%%%%%%%%%%%%%%%%%%%%%%%%%%%

%     [Short Title]{Full Title}
\title[Characterization of queer crystals]{A local characterization of crystals for the \\ quantum queer superalgebra}  

\author[Assaf]{Sami Assaf}
\address{Department of Mathematics, University of Southern California, 3620 South Vermont Avenue, Los Angeles, CA 90089-2532, U.S.A.}
\email{shassaf@usc.edu}
\thanks{Work supported in part by a Simons Foundation Collaboration Grant for Mathematicians (Award 524477, S.A.).}

\author[Oguz]{Ezgi Kantarci Oguz}
\address{Department of Mathematics, University of Southern California, 3620 South Vermont Avenue, Los Angeles, CA 90089-2532, U.S.A.}
\email{kantarci@usc.edu}
%\thanks{}

%    General info

%\date{\today}

%\dedicatory{}


\begin{abstract}
  We define operators on semistandard shifted tableaux and use Stembridge's local characterization for regular graphs to prove they define a crystal structure. This gives a new proof that Schur $P$-polynomials are Schur positive. We define queer crystal operators (also called odd Kashiwara operators) to construct a connected queer crystal on semistandard shifted tableaux of a given shape. Using the tensor rule for queer crystals, this provides a new proof that products of Schur $P$-polynomials are Schur $P$-positive. Finally, to facilitate applications of queer crystals in the context of Schur $P$-positivity, we give local axioms for queer regular graphs, generalizing Stembridge's axioms, that partially characterize queer crystals.
\end{abstract}

\maketitle
%\tableofcontents

%%%%%%%%%%%%%%%%%%%%%%%%%%%%%%%%%%%%%%%%%%%%%%%%%%%%%%%%%%%%%%%%
%
\section{Introduction}
%
%%%%%%%%%%%%%%%%%%%%%%%%%%%%%%%%%%%%%%%%%%%%%%%%%%%%%%%%%%%%%%%%
\label{sec:introduction}

Kashiwara \cite{Kas91} introduced crystal bases to study representations of quantized universal enveloping Lie algebras. Canonical bases, developed independently by Lusztig \cite{Lus90}, study the same problem from a geometric viewpoint. A \emph{crystal graph} is a directed, colored graph with vertex set given by the crystal basis and directed edges given by deformations of the Chevalley generators. Crystal graphs encode important information for the corresponding representations. For example, the character of the crystal coincides with the character of the representation, and branching rules and tensor decompositions can be computed combinatorially on the crystals. The crystal basis for the general linear group is naturally indexed by semistandard Young tableaux, and there is an explicit combinatorial construction of the crystal graph on tableaux developed independently by Kashiwara and Nakashima \cite{KN94} and Littelmann \cite{Lit95}. Furthermore, Stembridge \cite{Ste03} gave a local characterization of the crystal graph for the general linear group (and, more generally, for any simply laced Lie type) that allows one to determine whether a given crystal arises as a crystal for a highest weight representation.

The characters of connected highest weight crystals for the general linear group are given by Schur polynomials \cite{Jac41,Sch01}: the generating polynomials for semistandard Young tableaux. Schur polynomials are ubiquitous throughout mathematics, arising as irreducible characters for polynomial representations of the general linear group, Frobenius characters for irreducible representations of the symmetric group, and polynomial representatives for the cohomology classes of Schubert cycles in Grassmannians. Given the representation theoretic and geometric context, a quintessential question that arises is whether a given polynomial is Schur positive, meaning that it can be realized as a character for a representation of the general linear group. One approach for such problems that simultaneously sheds light on the underlying representation theory is to define a crystal structure on the objects that generate the given polynomial, and then  use Stembridge's local characterization to prove Schur positivity. 

To illustrate this approach and to motivate our generalization of this theory, we consider Schur's $P$-polynomials \cite{Sch11}. Schur $P$-polynomials arise as characters of tensor representations of the queer Lie superalgebra \cite{Ser84}, characters of projective representations of the symmetric group \cite{Ste89}, and representatives for cohomology classes dual to Schubert cycles in isotropic Grassmannians \cite{Pra91}. They enjoy many properties parallel to Schur polynomials. Stanley conjectured that Schur $P$-polynomials are Schur positive, and this follows from Sagan's shifted insertion \cite{Sag87} independently developed by Worley \cite{Wor84}. Assaf \cite{Ass18} gave another proof using the machinery of dual equivalence graphs \cite{Ass15}. In this paper, we present a new crystal-theoretic proof of the Schur positivity of Schur $P$-polynomials.

Hawkes, Paramonov, and Schilling \cite{HPS17} constructed a crystal on semistandard shifted tableaux, the generating objects for Schur $P$-polynomials. They use Haiman's mixed insertion \cite{Hai89} to associate to each shifted tableau a Young tableau and, using this correspondence, transport the usual crystal structure on Young tableaux to the shifted setting. Therefore, while this construction gives a desirable crystal-theoretic interpretation of the Schur positivity of Schur $P$-polynomials, it does not give a new proof of positivity in the sense that it relies on the positivity that follows from shifted insertion \cite{Wor84,Sag87}. In this work, we present a direct approach using Stembridge's local characterization that can also be seen as giving a new, direct proof that the crystal on shifted tableaux commutes with shifted insertion.

As noted above, Schur $P$-polynomials arise as characters of tensor representations of the queer Lie superalgebra. Lie superalgebras are algebras with a $\mathbb{Z}/2\mathbb{Z}$ grading, allowing for two families of variables (one commuting and one not) to interact. Originally arising from mathematical physics in connection with supersymmetry, Lie superalgebras were formalized mathematically and classified by Kac \cite{Kac77}. One well-studied superalgebra generalization of the general linear Lie algebra is the queer superalgebra. Quantized universal enveloping algebras were developed for the queer superalgebra by Sergeev \cite{Ser84}, with the corresponding crystal theory developed by Grantcharov, Jung, Kang, Kashiwara, and Kim \cite{GJKK10,GJKKK10}.

Grantcharov, Jung, Kang, Kashiwara, and Kim \cite{GJKKK10} gave an explicit construction of the queer crystal on semistandard decomposition tableaux \cite{Ser10}, an alternative combinatorial model for Schur $P$-polynomials developed by Serrano. They raised the question of whether the set of shifted semistandard Young tableaux of a fixed shape has a natural crystal structure. To answer this in the affirmative, we extend our crystal operators to include queer operators, also called odd Kashiwara operators, that augment the crystal with additional edges in such a way that the graph on semistandard shifted tableaux of fixed shape is connected, with a unique highest weight corresponding to the highest weight tensor representation. 

Many polynomials that arise in representation theoretic or geometric contexts can be expressed as a non-negative sum of Schur $P$-polynomials. For example, Hiroshima \cite{Hir} recently gave a queer crystal for type C Stanley symmetric functions, a generalization of Stanley symmetric functions \cite{Sta84} introduced for types B and C by Billey and Haiman \cite{BH95}, and Assaf and Marberg recently developed a queer crystal for involution Stanley symmetric functions introduced by Hamaker, Marberg, and Pawlowski \cite{HMP17}. In order to give a universal approach to proving Schur $P$-positivity of these and other functions, parallel to the axiomatization of shifted dual equivalence graphs \cite{Ass18,BHRY14}, we present local axioms for queer regular graphs. We prove that our explicit construction on semistandard shifted tableaux, and hence all normal queer crystals, are queer regular, and we show that in many cases the converse holds as well. In particular, we believe that our axioms can be tightened give a local characterization for normal queer crystals, thus providing a powerful new tool in symmetric function theory.

Our paper is organized as follows. In \S\ref{sec:schur}, we review the crystal theory for the general linear group from the combinatorial perspective. We review crystals and normal crystals that arise for highest weight representation, we review the explicit crystal on semistandard Young tableaux, and we present Stembridge's local characterization for such crystals. In \S\ref{sec:schur-P}, we apply Stembridge's axioms to prove that our explicit operators on semistandard shifted tableaux define a normal crystal. We begin by reviewing the combinatorics of shifted tableaux, then define our operators on shifted tableaux, and prove that Stembridge's axioms are satisfied. We also present corollaries demonstrating the utility of the resulting Schur expansion for Schur $P$-polynomials, including a short proof that a Schur $P$-polynomial is a single Schur polynomial if and only if the indexing shape is a staircase. In \S\ref{sec:graphs-Q}, we extend our construction on shifted tableaux to the queer setting. We review queer crystals and normal queer crystals that arise for tensor representations of queer Lie superalgebras, we define an explicit normal queer crystal on shifted tableaux, and we present our local axioms for such crystals. Our results and constructions were first announced in \cite{AO18}, where we conjectured that our axioms were sufficient. While that conjecture is false, we expect that a refinement might be true.

%%%%%%%%%%%%%%%%%%%%%%%%%%%%%%%%%%%%%%%%%%%%%%%%%%%%%%%%%%%%%%%%
%
\section{Crystals for the quantum general linear Lie algebra}
%
%%%%%%%%%%%%%%%%%%%%%%%%%%%%%%%%%%%%%%%%%%%%%%%%%%%%%%%%%%%%%%%%
\label{sec:schur}

Kashiwara \cite{Kas90,Kas91} introduced crystal bases in his study of the representation theory of quantized universal enveloping algebra $U_q(\mathfrak{g})$ for Lie algebra $\mathfrak{g}$. In this section, we review the theory of crystal bases and crystal graphs, focusing solely on the case of $U_q(\mathfrak{gl}(r))$ to simplify the exposition and keep it self-contained. In \S~\ref{sec:crystal-A}, we review crystal bases in the language of root systems, restricting to type $A_{r+1}$. In \S~\ref{sec:tableaux-A}, we review tableaux combinatorics and present the explicit combinatorial realization of crystal graphs on tableaux due to Kashiwara and Nakashima \cite{KN94} and Littelman \cite{Lit95}. In \S~\ref{sec:local-A}, we review an alternative local axiomatization of tableaux crystals due to Stembridge \cite{Ste03} that will be central to our proofs.

%%%%%%%%%%%%%%%%%%%%%%%%%%%%%%%%%%%%%%%%%%%%%%%%%%%%%%%%%%%%%%%%
\subsection{Crystal bases and crystal graphs}
%%%%%%%%%%%%%%%%%%%%%%%%%%%%%%%%%%%%%%%%%%%%%%%%%%%%%%%%%%%%%%%%
\label{sec:crystal-A}

We use the language of root systems to define crystal bases of type $A_{r+1}$, though the exposition is self-contained and no familiarity with Lie theory is assumed; see \cite{BS17} for further details. Let $\e_1, \e_2, \ldots, \e_{r+1}$ be the standard basis for $V = \mathbb{R}^{r+1}$ with the usual inner product. Consider the \emph{root system} $\Phi = \{\e_i - \e_j \mid i \neq j\}$. We refer to the \emph{positive roots} as the subset $\Phi^{+} = \{\e_i - \e_j \mid i<j\}$. Let $\alpha_i = \e_i - \e_{i+1}$ for $i=1,\ldots,r$ denote the \emph{simple roots}. The \emph{weight lattice} is $\Lambda = \mathbb{Z}^{r+1}$, and the \emph{dominant weights} $\Lambda^{+} \subset \Lambda$ are those $\lambda \in \Lambda$ such that $\lambda_1 \geq \lambda_2 \geq \cdots \geq \lambda_{r+1} \geq 0$. 

\begin{definition}
  A \emph{crystal} of dimension $r+1$ is a nonempty set $\B$ not containing $0$ together with \emph{crystal operators} $e_i, f_i  :  \B \rightarrow \B \cup \{0\}$ for $i=1,2,\ldots,r$ and a \emph{weight map} $\wt : \B \rightarrow \Lambda$ satisfying the conditions
  \begin{enumerate}
  \item for $b,b^{\prime}\in\B$, $e_i(b)=b^{\prime}$ if and only if $f_i(b^{\prime}) = b$, and in this case we have $\wt(b^{\prime}) = \wt(b) + \alpha_i$;
  \item for $b\in\B$ and $i=1,\ldots,r$, we have $\downf_i(b) = (\wt(b)_i - \wt(b)_{i+1}) + \upe_i(b)$, where $\upe_i, \downf_i : \B \rightarrow \mathbb{Z}$ are
    \begin{eqnarray}
      \upe_i(b) & = & \max\{k \in \mathbb{Z}_{\geq 0} \mid e_i^k(b) \neq 0 \}, \\
      \downf_i(b) & = & \max\{k \in \mathbb{Z}_{\geq 0} \mid f_i^k(b) \neq 0 \}.
    \end{eqnarray}
  \end{enumerate}
  We call $\upe_i(b), \downf_i(b)$ the \emph{string lengths} through $b$, with $\upe_i(b)$ the $i$-tail and $\downf_i(b)$ the $i$-head. 
  \label{def:base-A}
\end{definition}

Note that it is enough to define the $f_i$'s and the weight map, and, abusing notation, we denote the crystal $(\B,\{e_i,f_i\}_{1\leq i \leq r},\wt)$ simply by $\B$.

As a first example, the \emph{standard crystal} $\B(n)$, for $n \in \mathbb{Z}_{>0}$, has basis $\left\{\raisebox{-0.3\cellsize}{$\tableau{i}$} \mid i=1,\ldots,n\right\}$, crystal operators $f_j$ that act by incrementing the entry if $i=j$, and taking it to $0$ otherwise. The weight map is $\wt\left(\,\raisebox{-0.3\cellsize}{$\tableau{i}$}\,\right) = \e_i$.

\begin{definition}
  The \emph{character} of a crystal $\B$ is the polynomial
  \begin{equation}
    \mathrm{ch}(\B) = \sum_{b \in B} x_1^{\wt(b)_1} x_2^{\wt(b)_2} \cdots x_{r+1}^{\wt(b)_{r+1}}.
    \label{e:char}
  \end{equation}
  \label{def:character}
\end{definition}

For example, the character of the standard crystal $\B(n)$ is $\mathrm{ch}(\B(n)) = x_1 + x_2 + \cdots + x_n$. Condition (1) of Definition~\ref{def:base-A} ensures that the characters of connected crystals are homogeneous of a given degree.

A \emph{crystal graph} is a directed, colored graph with vertex set given by the crystal basis $\B$ and directed edges given by  the crystal operators $e_i$ and $f_i$, where if $b^{\prime}=e_i(b)$ (resp. $b^{\prime\prime} = f_i (b)$), then we write $b {\blue {\blue \ial}} b^{\prime}$ (resp. $b {\blue \iar} b^{\prime\prime}$) and all edges to $0$ are omitted. The {\em $i$-string through $b$} is the maximal path
$$
e_i^{\upe_i} x {\blue \iar} \cdots {\blue \iar} e_i x {\blue \iar} x {\blue \iar} f_i x {\blue \iar} \cdots {\blue \iar} f_i^{\downf_i} x .
$$
For example, the standard crystal graph $\B(n)$ is shown in Figure~\ref{fig:standard}. 

\begin{figure}[ht]
  \begin{displaymath}
    \begin{array}{c@{\hskip 2\cellsize}c@{\hskip 2\cellsize}c@{\hskip 2\cellsize}c@{\hskip 2\cellsize}c}
      \rnode{a1}{\tableau{1}} & \rnode{a2}{\tableau{2}} & \rnode{a3}{\tableau{3}} & \rnode{c}{\raisebox{0.5\cellsize}{$\cdots$}} & \rnode{a4}{\tableau{n}} 
    \end{array}
    \psset{nodesep=2pt,linewidth=.2ex}
    \ncline[linecolor=red]{->} {a1}{a2} \naput{1}
    \ncline[linecolor=blue]{->}  {a2}{a3} \naput{2}
    \ncline[linecolor=purple]{->}  {a3}{c} \naput{3}
    \ncline[linecolor=orange]{->}  {c}{a4} \naput{n\!-\!1}
  \end{displaymath}
  \caption{\label{fig:standard}The standard crystal $\B(n)$ for $U_{q}(\mathfrak{gl}(n))$.}
\end{figure}

\begin{definition}
  Given two crystals $\B_1$ and $\B_2$, the \emph{tensor product} $\B_1 \otimes \B_2$ is the set $\B_1 \otimes \B_2$ together with crystal operators $e_i, f_i$ defined on the tensor product $\B_1 \otimes \B_2$ by
  \begin{equation}
    f_i(b_1 \otimes b_2) = \left\{ \begin{array}{rl}
      f_i(b_1) \otimes b_2 & \mbox{if } \upe_i(b_2) < \downf_i(b_1), \\
      b_1 \otimes f_i(b_2) & \mbox{if } \upe_i(b_2) \geq \downf_i(b_1),
    \end{array} \right.
  \end{equation}
  and weight function $\wt(b_1 \otimes b_2) = \wt(b_1) + \wt(b_2)$, where addition is taken coordinate-wise.
\label{def:tensor-A}
\end{definition}

\begin{figure}[ht]
  \begin{displaymath}
    \begin{array}{c@{\hskip 2\cellsize}l@{\hskip 2\cellsize}l@{\hskip 2\cellsize}l}
                              & \rnode{a1}{\tableau{1}}               & \rnode{a2}{\tableau{2}}               & \rnode{a3}{\tableau{3}} \\[2\cellsize]
      \rnode{b1}{\tableau{1}} & \rnode{c11}{\tableau{1}\otimes\tableau{1}} & \rnode{c21}{\tableau{2}\otimes\tableau{1}} & \rnode{c31}{\tableau{3}\otimes\tableau{1}} \\[2\cellsize]
      \rnode{b2}{\tableau{2}} & \rnode{c12}{\tableau{1}\otimes\tableau{2}} & \rnode{c22}{\tableau{2}\otimes\tableau{2}} & \rnode{c32}{\tableau{3}\otimes\tableau{2}} \\[2\cellsize]
      \rnode{b3}{\tableau{3}} & \rnode{c13}{\tableau{1}\otimes\tableau{3}} & \rnode{c23}{\tableau{2}\otimes\tableau{3}} & \rnode{c33}{\tableau{3}\otimes\tableau{3}}
    \end{array}
    \psset{nodesep=2pt,linewidth=.2ex}
    \ncline[linewidth=.2ex,linecolor=red]{->} {a1}{a2} \naput{1}
    \ncline[linecolor=blue]{->}  {a2}{a3} \naput{2}
    \ncline[linewidth=.2ex,linecolor=red]{->} {c11}{c21} \naput{1}
    \ncline[linecolor=blue]{->}  {c21}{c31} \naput{2}
    \ncline[linewidth=.2ex,linecolor=red]{->} {b1}{b2} \nbput{1}
    \ncline[linewidth=.2ex,linecolor=red]{->} {c21}{c22} \nbput{1}
    \ncline[linewidth=.2ex,linecolor=red]{->} {c31}{c32} \nbput{1}
    \ncline[linecolor=blue]{->}  {c22}{c32} \naput{2}
    \ncline[linecolor=blue]{->}  {b2}{b3} \nbput{2}
    \ncline[linecolor=blue]{->}  {c12}{c13} \nbput{2}
    \ncline[linecolor=blue]{->}  {c32}{c33} \nbput{2}
    \ncline[linewidth=.2ex,linecolor=red]{->} {c13}{c23} \naput{1}
  \end{displaymath}
  \caption{\label{fig:tensor}The tensor product of two standard crystals $\B(3)$ for $U_{q}(\mathfrak{gl}(3))$.}
\end{figure}
We use the combinatorial structure of a crystals to encode essential information for studying the corresponding representations of the quantum group, with the crystal operators corresponding to deformations of the Chevalley generators. To do so, we must restrict our attention to \emph{normal crystals}. A direct definition of normal crystals is quite involved, but the following lemma gives a complete characterization.
\begin{proposition}
  Given two normal crystals $\B_1$ and $\B_2$, the tensor product $\B_1 \otimes \B_2$ is a normal crystal. Further, every connected normal crystal of dimension $r+1$ and degree $k$ arises as a connected component in $\B(r+1)^{\otimes k}$, the $k$-fold tensor product of the standard crystal $\B(r+1)$.
\label{prop:normal-A}
\end{proposition}

 Connected normal crystals are in one-to-one correspondence with dominant weights, which in turn index irreducible representations. Given a dominant weight $\lambda \in \Lambda^{+}$, let $\B(\lambda)$ denote the connected normal crystal with highest weight $\lambda$. Then $\mathrm{ch}(\B(\lambda))$ is precisely the character of the irreducible representation indexed by $\lambda$, which corresponds to the Schur polynomial $s_{\lambda}(x_1,\ldots,x_{r+1})$ defined in \S\ref{sec:tableaux-A}. Even more compelling is the remarkable fact that the following combinatorial procedure on crystals corresponds to the tensor product of the corresponding representations. For example, Figure~\ref{fig:tensor} computes that the tensor product of two copies of the standard crystal $\B(3)=\B((1,0,0))$ is given by $\B((2,0,0))$ and $\B((1,1,0))$.

\begin{definition}
  An element $b \in \B$ of a crystal is a \emph{highest weight element} if $e_i(b) = 0$ for all $i=1,2,\ldots,r$. 
  \label{def:hw-A}
\end{definition}

For example, the highest weight element of $\B(n)$ is $\raisebox{-0.3\cellsize}{$\tableau{1}$}$, which has weight $\e_1 \in \Lambda^{+}$. Highest weights give an efficient way to categorize normal crystals.

\begin{proposition}
  A connected, normal crystal $\B$ has a unique highest weight $b$, and we call $\wt(b)\in\Lambda^{+}$ the \emph{highest weight of $\B$}. Moreover, two connected normal crystals are isomorphic as colored directed graphs if and only if they have the same highest weight.
\end{proposition}

A consequence of this proposition is that normal crystals are determined completely by the underlying graph structure, and we do not need to be concerned with the weight maps.
Even so, it is not a straight-forward process to understand which crystals are normal, but as we shall see, there is a local characterization of that will prove quite useful. Before presenting this, we give an explicit combinatorial realization of normal crystals using tableaux.

%%%%%%%%%%%%%%%%%%%%%%%%%%%%%%%%%%%%%%%%%%%%%%%%%%%%%%%%%%%%%%%%
\subsection{Crystals on Young tableaux}
%%%%%%%%%%%%%%%%%%%%%%%%%%%%%%%%%%%%%%%%%%%%%%%%%%%%%%%%%%%%%%%%
\label{sec:tableaux-A}

A \emph{partition} $\lambda$ is a weakly decreasing sequence of positive integers, $\lambda = (\lambda_1,\lambda_2, \ldots, \lambda_{\ell})$, where $\lambda_1 \geq \lambda_2 \geq \cdots \geq \lambda_{\ell} > 0$. By extending $\lambda$ with trailing $0$s until it has length $r+1$, we may identify partitions with dominant weights $\Lambda^{+}$. The \emph{size} of a partition is the sum of its parts, i.e. $\lambda_1  +  \lambda_2 + \cdots + \lambda_{\ell}$. We identify a partition $\lambda$ with its Young diagram drawn in French (coordinate) notation, i.e. the collection of left-justified cells with $\lambda_i$ cells in row $i$ indexed from the bottom. 

A \emph{semistandard Young tableau} of shape $\lambda$ is a filling of the shifted Young diagram for $\lambda$ with positive integers such that entries weakly increase along rows and columns and each column has at most one entry $i$ for each $i$. For example, see Figure~\ref{fig:S31}. 

\begin{figure}[ht]
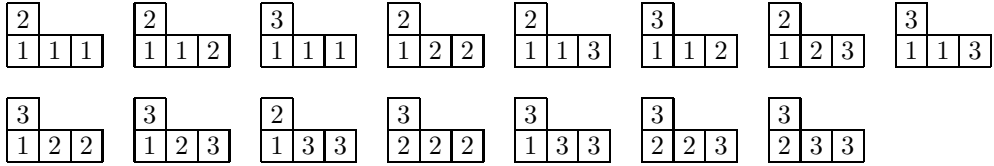

  \begin{center}
    \begin{displaymath}
      \begin{array}{c@{\hskip\cellsize}c@{\hskip\cellsize}c@{\hskip\cellsize}c@{\hskip\cellsize}c@{\hskip\cellsize}c@{\hskip\cellsize}c@{\hskip\cellsize}c}
        \tableau{2 \\ 1 & 1 & 1} & \tableau{2 \\ 1 & 1 & 2} & \tableau{3 \\ 1 & 1 & 1} & \tableau{2 \\ 1 & 2 & 2} & 
        \tableau{2 \\ 1 & 1 & 3} & \tableau{3 \\ 1 & 1 & 2} & \tableau{2 \\ 1 & 2 & 3} & \tableau{3 \\ 1 & 1 & 3} \\ \\
        \tableau{3 \\ 1 & 2 & 2} & \tableau{3 \\ 1 & 2 & 3} & \tableau{2 \\ 1 & 3 & 3} & \tableau{3 \\ 2 & 2 & 2} & 
        \tableau{3 \\ 1 & 3 & 3} & \tableau{3 \\ 2 & 2 & 3} & \tableau{3 \\ 2 & 3 & 3} & 
      \end{array}
    \end{displaymath}
    \caption{\label{fig:S31}The semistandard Young tableaux of shape $(3,1)$ with entries $\{1,2,3\}$.}
  \end{center}
\end{figure}

\begin{definition}
  The \emph{Schur polynomial} indexed by the partition $\lambda$ is given by
  \begin{equation}
    s_{\lambda} (x_1,\ldots,x_n) = \sum_{T \in \SSYT_n(\lambda)} x_1^{\wt(T)_1} \cdots x_n^{\wt(T)_n},
    \label{e:schur}
  \end{equation}
  where $\SSYT_n(\lambda)$ denotes the set of all semistandard Young tableaux of shape $\lambda$ with largest entry at most $n$, and $\wt(T)$ is the composition whose $i$th part is the multiplicity with which $i$ occurs in $T$. 
  \label{def:schur}
\end{definition}
  
For example, from Figure~\ref{fig:S31}, we compute
\begin{eqnarray*}
  s_{(3,1)}(x_1,x_2,x_3) & = & x_1^3 x_2 + x_1^3 x_3 + x_1^2 x_2^2 + 2 x_1^2 x_2 x_3 + x_1^2 x_3^2 + x_1 x_2^3 + 2 x_1 x_2^2 x_3 \\
  & & + 2 x_1 x_2 x_3^2 + x_1 x_3^3 + x_2^3 x_3 + x_2^2 x_3^2 + x_2 x_3^3.
\end{eqnarray*}

Schur polynomials are the irreducible characters for polynomial representations of the general linear group. Moreover, semistandard Young tableaux naturally index a basis for the irreducible representations. Therefore it is natural to seek a crystal structure with semistandard Young tableaux as the underlying basis. This was done by Kashiwara and Nakashima \cite{KN94} and Littelman \cite{Lit95}, though our presentation is again simplified to the case of the general linear group.

For a word $w$ of length $k$, a positive integer $r \leqslant k$, and a positive integer $i$, define
\begin{equation}
  m_i(w,r) = \wt(w_{1} w_{2} \cdots w_{r})_{i} - \wt(w_{1} w_{2} \cdots w_{r})_{i+1},
  \label{e:max}
\end{equation}
where $\wt(w)$ is the weak composition whose $j$th part is the number of $j$'s in $w$. Set $m_i(w) = \max_r(m_i(w,r),0)$.  Observe that if $m_i(w) > 0$ and $w_p$ is the leftmost occurrence of this maximum, then $w_p = i$, and if $q$ is the rightmost occurrence of this maximum, then either $q=k$ or $w_{q+1} = i+1$.

For $T$ a Young tableau, the \emph{row reading word of $T$}, denoted by $w(T)$, is the word obtained by reading the entries of $T$ left to right along rows, from the top row down. 

\begin{definition}
  Given a positive integer $i$, define \emph{lowering operators}, denoted by $f_i$, on semistandard Young tableaux as follows: if $m_i(w(T)) = 0$, then $f_i(T)=0$; otherwise, let $p$ be the smallest index such that $m_i(w(T),p) = m_i(w(T))$, and set $f_i(T)$ to change the entry in $T$ corresponding to $w_p$ to $i+1$.
  \label{def:young-lower}
\end{definition}

Another way to state Definition~\ref{def:young-lower} is in terms of blocked and free entries in the row reading word of $T$. For this, we say that a pair of letters $w_a, w_b$ are \emph{$i$-paired} if $w_a = i+1$, $w_b=i$, and $a<b$ such that for any $a<c<b$, either $w_c \neq i, i+1$ or $w_c$ is $i$-paired with some $w_d$ between $w_a$ and $w_b$. The entries that are $i$-paired are called $i$-blocked, and any letter $i$ or $i+1$ that is not $i$-blocked is called \emph{$i$-free}. With this terminology, $f_i$ changes the rightmost $i$-free entry $i$ to become $i+1$.

For example, Figure~\ref{fig:raise-A} shows the lowering operators applied to the left and middle tableaux, where the changed cells are indicated with a circle. This is, in fact, the complete $4$-string through these tableaux.

\begin{figure}[ht]
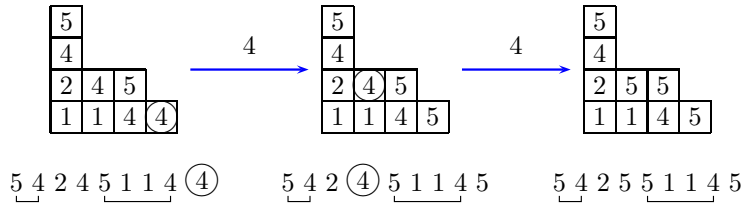

  \begin{center}
    \begin{displaymath}
      \begin{array}{c@{\hskip 2\cellsize}c@{\hskip 2\cellsize}c}
        \rnode{A}{\tableau{5 \\ 4 \\ 2 & 4 & 5 \\ 1 & 1 & 4 & \cir{4}}} &
        \rnode{B}{\tableau{5 \\ 4 \\ 2 & \cir{4} & 5 \\ 1 & 1 & 4 & 5}} &
        \rnode{C}{\tableau{5 \\ 4 \\ 2 & 5 & 5 \\ 1 & 1 & 4 & 5}} \\ \\
        \rnode{a1}{5} \ \rnode{a2}{4} \ 2 \ 4 \ \rnode{a3}{5} \ 1 \ 1 \ \rnode{a4}{4} \ \raisebox{-.5ex}{$\cir{4}$} &
        \rnode{b1}{5} \ \rnode{b2}{4} \ 2 \ \raisebox{-.5ex}{$\cir{4}$} \ \rnode{b3}{5} \ 1 \ 1 \ \rnode{b4}{4} \ 5 &
        \rnode{c1}{5} \ \rnode{c2}{4} \ 2 \ 5 \ \rnode{c3}{5} \ 1 \ 1 \ \rnode{c4}{4} \ 5
      \end{array}
      \psset{nodesep=5pt,linewidth=.2ex}
      \ncline[linecolor=blue]{->} {A}{B} \naput{4}
      \ncline[linecolor=blue]{->} {B}{C} \naput{4}
      \ncdiag[linewidth=.1ex,nodesep=2pt,angleA=-90,angleB=-90,arm=0.5ex] {a1}{a2} 
      \ncdiag[linewidth=.1ex,nodesep=2pt,angleA=-90,angleB=-90,arm=0.5ex] {a3}{a4} 
      \ncdiag[linewidth=.1ex,nodesep=2pt,angleA=-90,angleB=-90,arm=0.5ex] {b1}{b2} 
      \ncdiag[linewidth=.1ex,nodesep=2pt,angleA=-90,angleB=-90,arm=0.5ex] {b3}{b4} 
      \ncdiag[linewidth=.1ex,nodesep=2pt,angleA=-90,angleB=-90,arm=0.5ex] {c1}{c2} 
      \ncdiag[linewidth=.1ex,nodesep=2pt,angleA=-90,angleB=-90,arm=0.5ex] {c3}{c4} 
    \end{displaymath}
    \caption{\label{fig:raise-A}A complete $4$-string on semistandard Young tableaux of shape $(4,3,1,1)$, with the row reading word indicated below and $4$-blocked pairs bracketed.}
  \end{center}
\end{figure}

\begin{definition}
  Given a positive integer $i$, define \emph{raising operators}, denoted by $e_i$, on semistandard Young tableaux as follows: let $q$ be the largest index such that $m_i(w(T),q) = max_r(m_i(w(T),r))$. If $q=k$, then $e_i(T)=0$; otherwise, set $e_i(T)$ to change the entry in $T$ corresponding to $w_{q+1}$ to $i$.
  \label{def:young-raise}
\end{definition}

Note that $m_i(T)$ is precisely the string length $\downf_i(T)$, which also coincides with the number of $i$-free entries equal to $i$. The string length $\upe_i(T)$ is given by
\begin{equation}
  \max\{ r \mid \wt(w_{r} w_{r+1} \cdots w_{n})_{i+1} - \wt(w_{r} w_{r+1} \cdots w_{n})_{i} \},
\label{e:lmax}
\end{equation}
which also coincides with the number of $i$-free entries equal to $i+1$. Moreover, the largest index to attain the maximum in \eqref{e:lmax}, if positive, is the entry on which $e_i$ acts, which is the leftmost $i$-free entry $i+1$. 

\begin{theorem}[\cite{KN94,Lit95}]
  The raising and lowering operators $e_i, f_i$ for $i=1,\ldots,n-1$ are well-defined maps from $\SSYT_n(\lambda)$ to $\SSYT_n(\lambda) \cup \{0\}$ that determine a normal crystal $(\SSYT_n(\lambda),\{e_i,f_i\}_{1 \leq i < n},\wt)$ that is isomorphic to $\B(\lambda)$.
  \label{thm:young-crystal}
\end{theorem}

For example, Figure~\ref{fig:crystal-A} depicts the crystal structures on semistandard Young tableaux of shapes $(3,1)$, $(2,2)$, and $(2,1,1)$ with entries $\{1,2,3\}$. These crystals are isomorphic to $\B((3,1,0))$, $\B((2,2,0))$, and $\B((2,1,1))$, respectively.

\begin{figure}[ht]
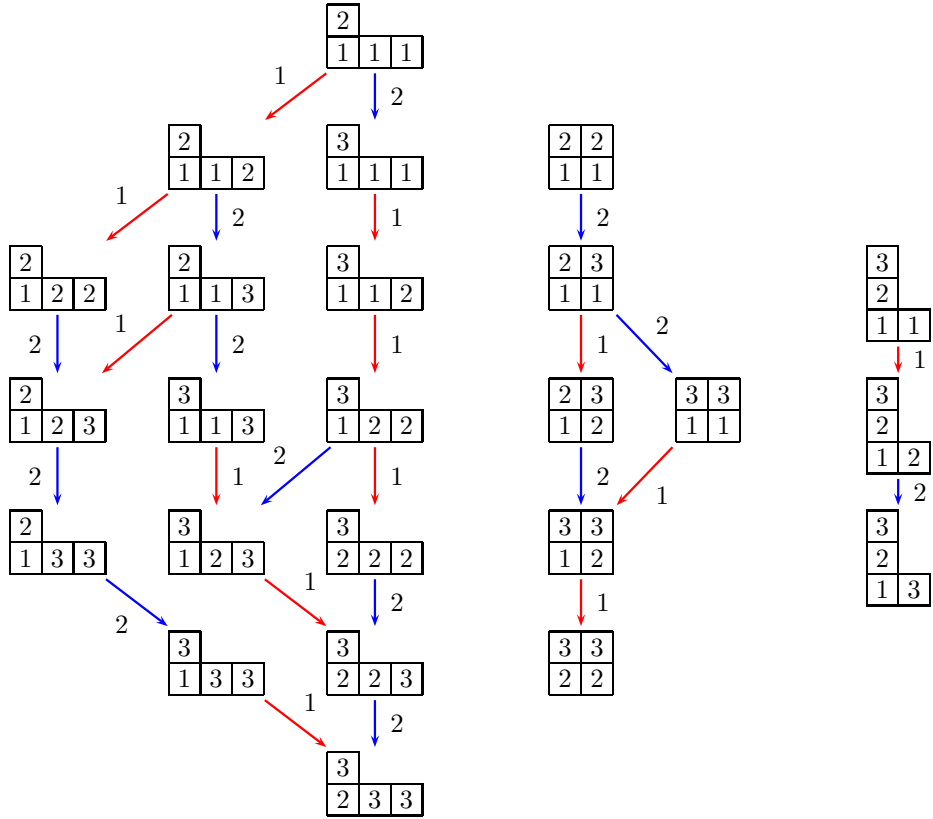

  \begin{center}
    \begin{displaymath}
      \begin{array}{c@{\hskip 2\cellsize}c@{\hskip 2\cellsize}c@{\hskip 4\cellsize}c@{\hskip 2\cellsize}c@{\hskip 4\cellsize}c}
        & & \rnode{c1}{\tableau{2 \\ 1 & 1 & 1}} & & & \\[7ex]
        & \rnode{b2}{\tableau{2 \\ 1 & 1 & 2}} & \rnode{c2}{\tableau{3 \\ 1 & 1 & 1}} & \rnode{d2}{\tableau{2 & 2 \\ 1 & 1}} & & \\[7ex]
        \rnode{a3}{\tableau{2 \\ 1 & 2 & 2}} & \rnode{b3}{\tableau{2 \\ 1 & 1 & 3}} & \rnode{c3}{\tableau{3 \\ 1 & 1 & 2}} & \rnode{d3}{\tableau{2 & 3 \\ 1 & 1}} & & \rnode{f3}{\tableau{3 \\ 2 \\ 1 & 1 }} \\[8ex]
        \rnode{a4}{\tableau{2 \\ 1 & 2 & 3}} & \rnode{b4}{\tableau{3 \\ 1 & 1 & 3}} & \rnode{c4}{\tableau{3 \\ 1 & 2 & 2}} & \rnode{d4}{\tableau{2 & 3 \\ 1 & 2}} & \rnode{e4}{\tableau{3 & 3 \\ 1 & 1}} & \rnode{f4}{\tableau{3 \\ 2 \\ 1 & 2}} \\[8ex]
        \rnode{a5}{\tableau{2 \\ 1 & 3 & 3}} & \rnode{b5}{\tableau{3 \\ 1 & 2 & 3}} & \rnode{c5}{\tableau{3 \\ 2 & 2 & 2}} & \rnode{d5}{\tableau{3 & 3 \\ 1 & 2}} & & \rnode{f5}{\tableau{3 \\ 2 \\ 1 & 3}} \\[7ex]
        & \rnode{b6}{\tableau{3 \\ 1 & 3 & 3}} & \rnode{c6}{\tableau{3 \\ 2 & 2 & 3}} & \rnode{d6}{\tableau{3 & 3 \\ 2 & 2}} & & \\[7ex]
        & & \rnode{c7}{\tableau{3 \\ 2 & 3 & 3}}
      \end{array}
      \psset{nodesep=2pt,linewidth=.2ex}
      % RANK 1 TO 2 
      \ncline[linewidth=.2ex,linecolor=red]{->} {c1}{b2} \nbput{1}
      \ncline[linecolor=blue]{->}  {c1}{c2} \naput{2}
      % RANK 2 TO 3
      \ncline[linewidth=.2ex,linecolor=red]{->}   {b2}{a3} \nbput{1}
      \ncline[linecolor=blue]{->}  {b2}{b3} \naput{2}
      \ncline[linewidth=.2ex,linecolor=red]{->} {c2}{c3} \naput{1}
      \ncline[linecolor=blue]{->}  {d2}{d3} \naput{2}
      % RANK 3 TO 4
      \ncline[linecolor=blue]{->}  {a3}{a4} \nbput{2}
      \ncline[linewidth=.2ex,linecolor=red]{->}   {b3}{a4} \nbput{1}
      \ncline[linecolor=blue]{->}  {b3}{b4} \naput{2}
      \ncline[linewidth=.2ex,linecolor=red]{->} {c3}{c4} \naput{1}
      \ncline[linewidth=.2ex,linecolor=red]{->} {d3}{d4} \naput{1}
      \ncline[linecolor=blue]{->}  {d3}{e4} \naput{2}
      \ncline[linewidth=.2ex,linecolor=red]{<-} {f4}{f3} \nbput{1} 
      % RANK 4 TO 5
      \ncline[linecolor=blue]{->}  {a4}{a5} \nbput{2}
      \ncline[linewidth=.2ex,linecolor=red]{->} {b4}{b5} \naput{1}
      \ncline[linecolor=blue]{->}  {c4}{b5} \nbput{2}
      \ncline[linewidth=.2ex,linecolor=red]{->} {c4}{c5} \naput{1}
      \ncline[linecolor=blue]{->}  {d4}{d5} \naput{2}
      \ncline[linewidth=.2ex,linecolor=red]{->} {e4}{d5} \naput{1}
      \ncline[linecolor=blue]{->}  {f4}{f5} \naput{2}
      % RANK 5 TO 6
      \ncline[linecolor=blue]{->}  {a5}{b6} \nbput{2}
      \ncline[linewidth=.2ex,linecolor=red]{<-}  {c6}{b5} \nbput{1} 
      \ncline[linecolor=blue]{->}  {c5}{c6} \naput{2}
      \ncline[linewidth=.2ex,linecolor=red]{<-} {d6}{d5} \nbput{1}
      % RANK 6 TO 7
      \ncline[linewidth=.2ex,linecolor=red]{<-}  {c7}{b6} \nbput{1} 
      \ncline[linecolor=blue]{->}  {c6}{c7} \naput{2}
    \end{displaymath}
    \caption{\label{fig:crystal-A}The crystal structures on semistandard Young tableaux of shapes $(3,1)$, $(2,2)$, $(2,1,1)$ with entries $\{1,2,3\}$, which are isomorphic to $\B((3,1,0))$, $\B((2,2,0))$, $\B((2,1,1))$.}
  \end{center}
\end{figure}

In particular, on the level of polynomials we have
\[ s_{\lambda}(x_1,\ldots,x_n) = \mathrm{ch}(\B(\lambda)). \]
Using crystal theory, this means that whenever we have a normal crystal structure on a set of objects, the objects are in weight-preserving bijection with highest weight crystals $\B(\lambda)$, therefore the generating polynomial is Schur positive. Specifically, we have the following.

\begin{corollary}
  For a normal crystal $\B$, we have
  \begin{equation}
    \mathrm{ch}(\B) = \sum_{\substack{b \in \B \\ e_i(b)=0 \forall i}} s_{\wt(b)}.
  \end{equation}
  In particular, the character of $\B$ is symmetric and Schur positive.
  \label{cor:highest-weights}
\end{corollary}

%For example, using the tensor rule for crystals, we have the following crystal theoretic proof of the classical Littlewood--Richardson rule for Schur polynomials.
%
%\begin{corollary}
%  Let $\lambda,\mu$ be partitions of length at most $n$. Then
%  \begin{equation}
%    s_{\lambda} s_{\mu} = \mathrm{ch}(\B(\lambda)) \mathrm{ch}(\B(\mu)) = \mathrm{ch}(\B(\lambda)\otimes\B(\mu)) = \sum_{\substack{b \in \B(\lambda)\otimes\B(\mu) \\ e_i(b)=0 \forall i}} s_{\wt(b)}.
%  \end{equation}
%\end{corollary}

%%%%%%%%%%%%%%%%%%%%%%%%%%%%%%%%%%%%%%%%%%%%%%%%%%%%%%%%%%%%%%%%
\subsection{Local characterization of crystals}
%%%%%%%%%%%%%%%%%%%%%%%%%%%%%%%%%%%%%%%%%%%%%%%%%%%%%%%%%%%%%%%%
\label{sec:local-A}

Establishing that a given directed graph structure on a combinatorial set is a normal crystal is difficult. To simplify this greatly, Stembridge \cite{Ste03} gave a local characterization of crystals that arise from representations for simply-laced types.

In order to define Stembridge's axioms for type $A_{r+1}$, we first introduce notation associated with a directed, colored graph. A \emph{graph of dimension $r+1$} will mean a directed, colored graph $\mathcal{X}$ with directed edges $e_i(x) {\blue \iar} x {\blue \iar} f_i(x)$ for $i=1,2,\ldots,r$. Extending the string lengths $\upe_i,\downf_i$ to this case, we also have the following differences whenever $e_i, f_i$ is defined at $x$,
\begin{displaymath}
  \begin{array}{rclcrcl}
    \Delta_i \upe_j(x) & = & \upe_j( x) - \upe_j(e_i(x)), & &
    \nabla_i \upe_j(x) & = & \upe_j(f_i(x)) - \upe_j( x), \\
    \Delta_i \downf_j(x) & = & \downf_j(e_i x) - \downf_j(x), & &
    \nabla_i \downf_j(x) & = & \downf_j(x) - \downf_j(f_i x).
  \end{array}
\end{displaymath}

\begin{definition}[\cite{Ste03}]
  A directed, colored graph $\mathcal{X}$ is {\em regular} if the following hold:
  \begin{itemize}
  \item[(A1)] all monochromatic directed paths have finite length;

  \item[(A2)] for every vertex $x$, there is at most one edge $x {\blue \ial} y$ and at most one edge $x {\blue \iar} z$;

  \item[(A3)] assuming $e_i x$ is defined, $\Delta_i \upe_j(x) + \Delta_i \downf_j(x) = \left\{
    \begin{array}{rl}
        2 & \;\mbox{if}\;\; j=i \\
        -1 & \;\mbox{if}\;\; j=i\pm 1 \\
        0 & \;\mbox{if}\;\; |i-j|\geq 2
    \end{array}
    \right.$;
    
  \item[(A4)] assuming $e_i x$ is defined, $\Delta_i \upe_j(x), \Delta_i \downf_j(x) \leq 0$ for $j \neq i$;

  \item[(A5)] $\Delta_i \upe_j(x) = 0$ $\Rightarrow$ $e_ie_j x = e_je_i x = y$ and $\nabla_j \downf_i(y) = 0$; \\
    $\nabla_i \downf_j(x) = 0$ $\Rightarrow$ $f_if_j x = f_jf_i x = y$ and $\Delta_j \upe_i(y) = 0$;

  \item[(A6)] $\Delta_i \upe_j(x) = \Delta_j \upe_i(x) = -1$ $\Rightarrow$ $E_iE_{j}^{2}E_i x = E_jE_{i}^{2}E_j x = y$ and $\nabla_i \downf_j(y) = \nabla_j \downf_i(y)=-1$; \\
    $\nabla_i \downf_j(x) = \nabla_j \downf_i(x) = -1$ $\Rightarrow$ $F_iF_{j}^{2}F_i x = F_jF_{i}^{2}F_j x = y$ and $\Delta_i \upe_j(y) = \Delta_j \upe_i(y)=-1$.
  \end{itemize}
  \label{def:regular}
\end{definition}

Axiom A1 ensures that the crystal operators have finite order when applied to a given basis element, though it does not force the graph itself to be finite. Axiom A2 ensures that the edges of the graph correspond to well-defined operators, which we assume map a basis element to $0$ in the absence of an edge.

Axioms A3 and A4 dictate how the length of a $j$-string differs between $x$ and $e_i x$; see Figure~\ref{fig:P3P4}. When $|i-j| \geq 2$, there is no change (left case), but when $j = i\pm 1$, the length changes by $1$, either by decreasing the tail when $\Delta_i \upe(x,j) = -1$ (middle case) or by increasing the head when $\Delta_i \downf(x,j) = -1$ (right case).

\begin{figure}[ht]
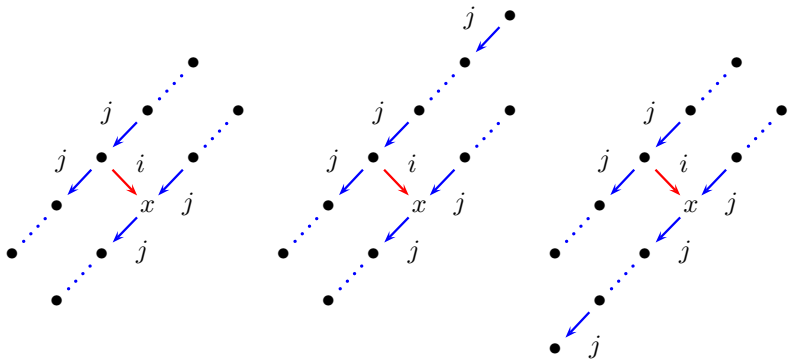

  \begin{center}
    \begin{displaymath}
      \begin{array}{c@{\hskip\cellsize}c@{\hskip\cellsize}c@{\hskip\cellsize}c@{\hskip\cellsize}c@{\hskip\cellsize}c@{\hskip\cellsize} c@{\hskip\cellsize}c@{\hskip\cellsize}c@{\hskip\cellsize}c@{\hskip\cellsize}c@{\hskip\cellsize}c@{\hskip\cellsize} c@{\hskip\cellsize}c@{\hskip\cellsize}c@{\hskip\cellsize}c@{\hskip\cellsize}c@{\hskip\cellsize}c@{\hskip\cellsize}}
        & & & & & & & & & & & \rnode{a12}{\bullet} & & & & & & \\[0.5\cellsize]
        & & & &  \rnode{b5}{\bullet} & &
        & & & & \rnode{b11}{\bullet} & & 
        & & & & \rnode{a17}{\bullet} & \\[0.5\cellsize]
        & & & \rnode{c4}{\bullet}  & & \rnode{c6}{\bullet}  &
        & & & \rnode{c10}{\bullet} & & \rnode{c12}{\bullet} &
        & & & \rnode{b16}{\bullet} & & \rnode{b18}{\bullet} \\[0.5\cellsize]
        & & \rnode{d3}{\bullet}  & & \rnode{d5}{\bullet}  & &
        & & \rnode{d9}{\bullet}  & & \rnode{d11}{\bullet} & &
        & & \rnode{c15}{\bullet} & & \rnode{c17}{\bullet} & \\[0.5\cellsize]
        & \rnode{e2}{\bullet}  & & \rnode{e4}{x}  & & &
        & \rnode{e8}{\bullet}  & & \rnode{e10}{x} & & &
        & \rnode{d14}{\bullet} & & \rnode{d16}{x} & & \\[0.5\cellsize]
        \rnode{f1}{\bullet}  & & \rnode{f3}{\bullet}  & & & &
        \rnode{f7}{\bullet}  & & \rnode{f9}{\bullet}  & & & &
        \rnode{e13}{\bullet} & & \rnode{e15}{\bullet} & & & \\[0.5\cellsize]
        & \rnode{g2}{\bullet}  & & & & &
        & \rnode{g8}{\bullet}  & & & & &        
        & \rnode{f14}{\bullet} & & & & \\[0.5\cellsize]
        & & & & & & & & & & & & \rnode{g13}{\bullet} & & & & & 
      \end{array}
      \psset{linewidth=.2ex,nodesep=2pt}
      % |i-j| => 2
      \ncline[linecolor=blue,linestyle=dotted,linewidth=.3ex] {b5}{c4}
      \ncline[linecolor=blue]{->}              {c4}{d3} \nbput{j}
      \ncline[linecolor=blue]{->}              {d3}{e2} \nbput{j}
      \ncline[linecolor=blue,linestyle=dotted,linewidth=.3ex] {e2}{f1}
      \ncline[linewidth=.2ex,linecolor=red]{->}  {d3}{e4}  \naput{i}
      \ncline[linecolor=blue,linestyle=dotted,linewidth=.3ex] {c6}{d5}
      \ncline[linecolor=blue]{->}              {d5}{e4} \naput{j}
      \ncline[linecolor=blue]{->}              {e4}{f3} \naput{j}
      \ncline[linecolor=blue,linestyle=dotted,linewidth=.3ex] {f3}{g2}
      % delta = -1
      \ncline[linecolor=blue]{->}              {a12}{b11} \nbput{j}
      \ncline[linecolor=blue,linestyle=dotted,linewidth=.3ex] {b11}{c10}
      \ncline[linecolor=blue]{->}              {c10}{d9}  \nbput{j}
      \ncline[linecolor=blue]{->}              {d9} {e8}  \nbput{j}
      \ncline[linecolor=blue,linestyle=dotted,linewidth=.3ex] {e8} {f7}
      \ncline[linewidth=.2ex,linecolor=red]{->}  {d9}{e10}  \naput{i}
      \ncline[linecolor=blue,linestyle=dotted,linewidth=.3ex] {c12}{d11}
      \ncline[linecolor=blue]{->}              {d11}{e10} \naput{j}
      \ncline[linecolor=blue]{->}              {e10}{f9}  \naput{j}
      \ncline[linecolor=blue,linestyle=dotted,linewidth=.3ex] {f9} {g8}
      % epsilon = -1
      \ncline[linecolor=blue,linestyle=dotted,linewidth=.3ex] {a17}{b16}
      \ncline[linecolor=blue]{->}              {b16}{c15} \nbput{j}
      \ncline[linecolor=blue]{->}              {c15}{d14} \nbput{j}
      \ncline[linecolor=blue,linestyle=dotted,linewidth=.3ex] {d14}{e13}
      \ncline[linewidth=.2ex,linecolor=red]{->}  {c15}{d16}  \naput{i}
      \ncline[linecolor=blue,linestyle=dotted,linewidth=.3ex] {b18}{c17}
      \ncline[linecolor=blue]{->}              {c17}{d16} \naput{j}
      \ncline[linecolor=blue]{->}              {d16}{e15} \naput{j}
      \ncline[linecolor=blue,linestyle=dotted,linewidth=.3ex] {e15}{f14}
      \ncline[linecolor=blue]{->}              {f14}{g13} \naput{j}
    \end{displaymath}
    \caption{\label{fig:P3P4}An illustration of axioms A3 and A4, where $f_j {\blue \swarrow}$, $f_i {\red \searrow}$.}
  \end{center}
\end{figure}  

Axioms A5 and A6 give information about how edges with different labels interact locally; see Figure~\ref{fig:P5P6}. When $|i-j| \geq 2$, the conditions of A5 will always be satisfied, though when $j = i\pm 1$, either A5 or A6 could hold, the former when $\Delta_i \upe(x,j) = 0$ and the latter when $\Delta_i \upe(x,j) = \Delta_j \upe(x,i) = -1$.

\begin{figure}[ht]
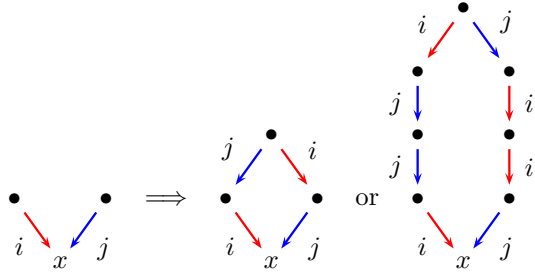

  \begin{center}
    \begin{displaymath}
      \begin{array}{c@{\hskip\cellsize}c@{\hskip\cellsize}c@{\hskip\cellsize}c@{\hskip\cellsize}c@{\hskip\cellsize}c@{\hskip\cellsize}c@{\hskip\cellsize}c@{\hskip\cellsize}c@{\hskip\cellsize}c@{\hskip\cellsize}c}  
        & & & & & & & & & \rnode{a11}{\bullet} & \\[\cellsize]
        & & & & & & & & \rnode{b10}{\bullet} & & \rnode{b12}{\bullet} \\[\cellsize]
        & & & & & \rnode{c6}{\bullet} & & & \rnode{c11l}{\bullet} & & \rnode{c11r}{\bullet} \\[\cellsize] 
        \rnode{d1}{\bullet} & & \rnode{d3}{\bullet} & \rnode{d4}{\Longrightarrow} & \rnode{d5}{\bullet} & & \rnode{d7}{\bullet} & \rnode{d8}{\mbox{or}} & \rnode{d9}{\bullet} & & \rnode{d13}{\bullet} \\[\cellsize]
        & \rnode{e2}{x} & & & & \rnode{e6}{x} & & & & \rnode{e11}{x} & 
      \end{array}
      \psset{linewidth=.2ex,nodesep=3pt}
      % fragment
      \ncline[linecolor=red,linewidth=.2ex]{->} {d1}{e2}  \nbput{i}
      \ncline[linecolor=blue]{->} {d3}{e2}  \naput{j}
      % commuting
      \ncline[linecolor=blue]{->} {c6}{d5}  \nbput{j}
      \ncline[linecolor=red,linewidth=.2ex]{->} {c6}{d7}  \naput{i}
      \ncline[linecolor=red,linewidth=.2ex]{->} {d5}{e6}  \nbput{i}
      \ncline[linecolor=blue]{->} {d7}{e6}  \naput{j}
      % blip
      \ncline[linecolor=blue]{->} {a11}{b12}  \naput{j}
      \ncline[linecolor=red,linewidth=.2ex]{->} {a11}{b10}  \nbput{i}
      \ncline[linecolor=red,linewidth=.2ex]{->} {b12}{c11r} \naput{i}
      \ncline[linecolor=blue]{->} {b10}{c11l} \nbput{j}
      \ncline[linecolor=blue]{->} {c11l}{d9}  \nbput{j}
      \ncline[linecolor=red,linewidth=.2ex]{->} {c11r}{d13} \naput{i}
      \ncline[linecolor=red,linewidth=.2ex]{->} {d9}{e11}   \nbput{i}
      \ncline[linecolor=blue]{->} {d13}{e11}  \naput{j}
    \end{displaymath}
    \caption{\label{fig:P5P6}An illustration of axioms A5 and A6, where $f_j {\blue \swarrow}$, $f_i {\red \searrow}$.}
  \end{center}
\end{figure}

The following theorem shows that regular graphs correspond precisely to the crystal graphs of representations for the general linear group. This result allows us to combinatorialize the problem of studying representations of the general linear group by instead studying regular graphs of degree $n$.

\begin{theorem}[\cite{Ste03}]
  Every normal crystal is a regular graph and every regular graph is a normal crystal.
\label{thm:structure-crystal}
\end{theorem}

The theorem is proved by showing that Littelmann's path operators \cite{Lit94} generate regular graphs, since the Path Model is known to generate $\B(\lambda)$. This then gives a straightforward machinery for proving that a given structure is a crystal graph by showing that the graph is regular.

%%%%%%%%%%%%%%%%%%%%%%%%%%%%%%%%%%%%%%%%%%%%%%%%%%%%%%%%%%%%%%%%
%
\section{A crystal for shifted tableaux}
%
%%%%%%%%%%%%%%%%%%%%%%%%%%%%%%%%%%%%%%%%%%%%%%%%%%%%%%%%%%%%%%%%
\label{sec:schur-P}

Schur $P$-polynomials arise as characters of tensor representations of the queer Lie superalgebra \cite{Ser84}, characters of projective representations of the symmetric group \cite{Ste89}, and representatives for cohomology classes dual to Schubert cycles in isotropic Grassmannians \cite{Pra91}. Stanley conjectured that Schur $P$-polynomials are Schur positive, and this follows from Sagan's shifted insertion \cite{Sag87} independently developed by Worley \cite{Wor84}. Assaf \cite{Ass18} gave another proof using the machinery of dual equivalence graphs \cite{Ass15}. In this section, we present a new proof using the machinery of crystal graphs. In \S\ref{sec:tableaux-B}, we review the combinatorial definition of Schur $P$-polynomials as the generating polynomials for semistandard shifted tableaux. In \S\ref{sec:crystal-B}, we define crystal operators on semistandard shifted tableaux analogous to those for semistandard Young tableaux. In \S\ref{sec:local-B}, we use Stembridge's axioms to prove that our crystal is normal, thus giving a new proof of the Schur positivity of Schur $P$-polynomials.

%%%%%%%%%%%%%%%%%%%%%%%%%%%%%%%%%%%%%%%%%%%%%%%%%%%%%%%%%%%%%%%%
\subsection{Shifted tableaux}
%%%%%%%%%%%%%%%%%%%%%%%%%%%%%%%%%%%%%%%%%%%%%%%%%%%%%%%%%%%%%%%%
\label{sec:tableaux-B}

A partition $\gamma$ is \emph{strict} if  $\gamma_1 > \gamma_2 > \cdots > \gamma_{\ell} > 0$. We identify a strict partition $\gamma$ with its shifted Young diagram, the collection of cells with $\gamma_i$ cells in row $i$ shifted $\ell(\gamma)-i$ cells to the left. 

A \emph{semistandard shifted tableau} of shape $\gamma$ is a filling of the shifted Young diagram for $\gamma$ with marked or unmarked positive integers such that entries weakly increase along rows and columns according to the ordering $\st{1} < 1 < \st{2} < 2 < \cdots$, each row at has most one marked entry $\st{i}$ for each $i$, each column has at most one unmarked entry $i$ for each $i$, and the diagonal has no markings. For example, see Figure~\ref{fig:P31}. 

\begin{figure}[ht]
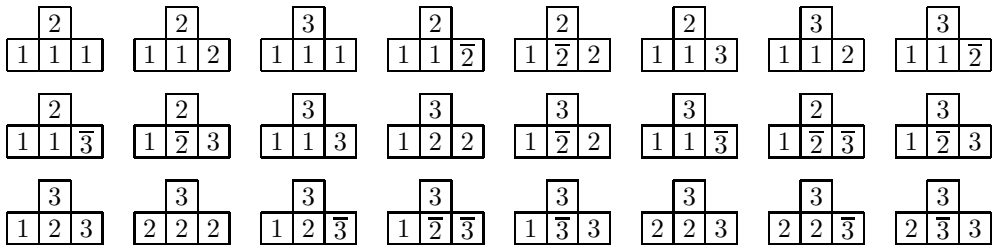

  \begin{center}
    \begin{displaymath}
      \begin{array}{c@{\hskip\cellsize}c@{\hskip\cellsize}c@{\hskip\cellsize}c@{\hskip\cellsize}c@{\hskip\cellsize}c@{\hskip\cellsize}c@{\hskip\cellsize}c}
        \tableau{ & 2 \\ 1 & 1 & 1} & \tableau{ & 2 \\ 1 & 1 & 2} & \tableau{ & 3 \\ 1 & 1 & 1} & \tableau{ & 2 \\ 1 & 1 & \st{2}} & \tableau{ & 2 \\ 1 & \st{2} & 2} & \tableau{ & 2 \\ 1 & 1 & 3} & \tableau{ & 3 \\ 1 & 1 & 2} & \tableau{ & 3 \\ 1 & 1 & \st{2}} \\[4ex]
        \tableau{ & 2 \\ 1 & 1 & \st{3}} & \tableau{ & 2 \\ 1 & \st{2} & 3} & \tableau{ & 3 \\ 1 & 1 & 3} & \tableau{ & 3 \\ 1 & 2 & 2} & \tableau{ & 3 \\ 1 & \st{2} & 2} & \tableau{ & 3 \\ 1 & 1 & \st{3}} & \tableau{ & 2 \\ 1 & \st{2} & \st{3}} & \tableau{ & 3 \\ 1 & \st{2} & 3} \\[4ex]
        \tableau{ & 3 \\ 1 & 2 & 3} & \tableau{ & 3 \\ 2 & 2 & 2} & \tableau{ & 3 \\ 1 & 2 & \st{3}} & \tableau{ & 3 \\ 1 & \st{2} & \st{3}} & \tableau{ & 3 \\ 1 & \st{3} & 3} & \tableau{ & 3 \\ 2 & 2 & 3} & \tableau{ & 3 \\ 2 & 2 & \st{3}} & \tableau{ & 3 \\ 2 & \st{3} & 3}
      \end{array}
    \end{displaymath}
    \caption{\label{fig:P31}The semistandard shifted tableaux of shape $(3,1)$ with entries $\{\st{1},1,\st{2},2,\st{3},3\}$ and no marks on the main diagonal.}
  \end{center}
\end{figure}

Schur $P$-polynomials enjoy nice properties parallel to Schur polynomials. 

\begin{definition}
  The Schur $P$-polynomial indexed by the strict partition $\gamma$ is given by
  \begin{equation}
    P_{\gamma} (x_1,\ldots,x_{n}) = \sum_{S \in \SSHT_n(\gamma)} x_1^{\wt(S)_1} \cdots x_n^{\wt(S)_n},
    \label{e:schur-P}
  \end{equation}
  where $\SSHT_n(\gamma)$ denotes the set of all semistandard shifted tableaux of shifted shape $\gamma$ with largest entry at most $n$, and $\wt(S)$ is composition whose $i$th part is the total multiplicity with which $i$ and $\st{i}$ occur in $S$. 
  \label{def:schur-P}
\end{definition}
  
For example, from Figure~\ref{fig:P31}, we compute
\begin{eqnarray*}
  P_{(3,1)}(x_1,x_2,x_3) & = & x_1^3 x_2 + x_1^3 x_3 + 2 x_1^2 x_2^2 + 4 x_1^2 x_2 x_3 + 2 x_1^2 x_3^2 + x_1 x_2^3 \\
  & & + 4 x_1 x_2^2 x_3 + 4 x_1 x_2 x_3^2 + x_1 x_3^3 + x_2^3 x_3 + 2 x_2^2 x_3^2 + x_2 x_3^3.
\end{eqnarray*}

Stanley conjectured that Schur $P$ polynomials expand positively into the Schur basis, and this follows as a corollary to Sagan's shifted insertion \cite{Sag87} independently developed by Worley \cite{Wor84}. Another combinatorial proof appears in \cite{Ass18} using the machinery of combinatorial graphs.

\begin{theorem}[\cite{Sag87,Wor84}]
  For $\gamma$ a strict partition, the coefficients $g_{\gamma,\lambda}$ defined by
    \begin{equation}
      P_{\gamma} (x_1,\ldots,x_{n}) = \sum_{\lambda} g_{\gamma,\lambda} s_{\lambda}(x_1,\ldots,x_{n})
      \label{e:P_schur}
    \end{equation}
    are non-negative integers. That is, Schur $P$-polynomials are Schur positive.
\label{thm:P_Schur}
\end{theorem}

For example, our previous computation can we written as 
\begin{eqnarray*}
  P_{(3,1)}(x_1,x_2,x_3) & = & s_{(3,1)}(x_1,x_2,x_3) + s_{(2,2)}(x_1,x_2,x_3) + s_{(2,1,1)}(x_1,x_2,x_3).
\end{eqnarray*}

Stembridge \cite{Ste89} expanded on the idea of shifted insertion in his study of projective representations of the symmetric group, and ultimately established that the product of Schur $P$-polynomials expands positively in the Schur $P$ basis. More recently, Cho \cite{Cho13} built on work of Serrano \cite{Ser10} to give another proof of positivity, and Assaf \cite{Ass18} gave another proof using the machinery of combinatorial graphs.

\begin{theorem}[\cite{Ste89}]
  For $\gamma,\delta$ strict partitions, the coefficients $f_{\gamma,\delta}^{\epsilon}$ defined by
  \begin{equation}
    P_{\gamma} (x_1,\ldots,x_{n}) P_{\delta} (x_1,\ldots,x_{n}) = \sum_{\epsilon} f_{\gamma,\delta}^{\epsilon} P_{\epsilon}(x_1,\ldots,x_{n}),
    \label{e:P_product}
  \end{equation}
  are non-negative integers. That is, products of Schur $P$-polynomials are Schur $P$ positive.
\label{thm:P_product}
\end{theorem}

In the present paper, we give new combinatorial proofs of Theorems~\ref{thm:P_Schur} and \ref{thm:P_product} using crystal graphs, both of which give rise to new combinatorial formulas for the expansions.

%%%%%%%%%%%%%%%%%%%%%%%%%%%%%%%%%%%%%%%%%%%%%%%%%%%%%%%%%%%%%%%%
\subsection{Crystals on shifted tableaux}
%%%%%%%%%%%%%%%%%%%%%%%%%%%%%%%%%%%%%%%%%%%%%%%%%%%%%%%%%%%%%%%%
\label{sec:crystal-B}

We give a new proof of the Schur positivity of Schur $P$-polynomials by constructing a crystal graph on semistandard shifted tableaux. We note that Hawkes, Paramonov, and Schilling \cite{HPS17} recently constructed a crystal on semistandard shifted tableaux in the context of type B/C Stanley symmetric functions. Their construction uses Haiman's mixed insertion \cite{Hai89} to associate a reduced word for a signed permutation with a pair of shifted tableaux, the left semistandard and the right standard. Using shifted insertion developed independently by Worley \cite{Wor84} and Sagan \cite{Sag87} to associate the left semistandard shifted tableau with a semistandard Young tableauthey are able to put the usual crystal structure reviewed in \S\ref{sec:tableaux-A} on the semistandard Young tableaux. Therefore, while the Hawkes-Paramonov-Schilling construction gives a desirable crystal interpretation of the Schur positivity of Schur $P$-polynomials, it does not give a new proof of positivity in the sense that it relies on the positivity that follows from shifted insertion \cite{Wor84,Sag87}.

We developed our crystal independently, unaware of \cite{HPS17}, and our presentation below is direct and differs from the derived description in \cite{HPS17}. However, in Proposition~\ref{prop:HPS} below, we prove that the two constructions are indeed equivalent. Nevertheless, we proceed with our direct description below and give a direct proof, using Stembridge's axioms, that this defines a normal crystal, thus giving a new proof of the Schur positivity of Schur $P$-polynomials independent of shifted insertion.

To begin, we define a reading word for shifted tableaux as follows.

\begin{definition}
  For $T$ a shifted tableau, the \emph{hook reading word of $T$}, denoted by $w(T)$, is the word obtained by reading the marked entries of $T$ up the $i$th column, then the unmarked entries of $T$ along the $i$th row, left to right, for $i$ from $\max(\gamma_1,\ell(\gamma))$ to $1$.
  \label{def:hook-word}
\end{definition}

For example, the hook reading words for the three tableaux in Figure~\ref{fig:raise-B} are listed below the tableaux. By disregarding the marks, we can calculate $m_i(w(T))$ just as in the unshifted case. Note that our reading word is not the reading word used to define crystal operators in \cite{HPS17}. 

The row and column conditions for shifted tableaux ensure that the cells with entries $\st{i},i$ must form a ribbon, i.e. contain no $2\times 2$ block. This observation helps to make the following well-defined.

\begin{definition}
  The \emph{shifted lowering operators}, denoted by $\st{f}_i$, act on semistandard shifted tableaux by: $\st{f}_i(T)=0$ if $m_i(w(T)) \leqslant 0$; otherwise, letting $p$ be the smallest index such that $m_i(w(T),p) = m_i(w(T))$, letting $x$ denote the entry of $T$ corresponding to $w_p$, and letting $y$ be the entry north of $x$ and $z$ the entry east of $x$, we have
    \begin{itemize}
    \item[(L1)] \begin{itemize}
      \item[(a)] if $x=i$ and $z=\st{i+1}$, then $\st{f}_i$ changes $x$ to $\st{i+1}$ and changes $z$ to $i+1$;
      \item[(b)] else if $x=i$ and $y$ does not exist or $y > i+1$, then $\st{f}_i$ changes $x$ to $i+1$;
      \item[(c)] else if $x=i$ and the northeastern-most cell on the $(i+1)$-ribbon containing $y$  has a marking, then $\st{f}_i$ removes that marking and changes $x$ to $\st{i+1}$;
      \item[(d)] else if $x=i$, then $\st{f}_i$ changes $x$ to $\st{i+1}$;
      \end{itemize}

      \begin{figure}[ht]
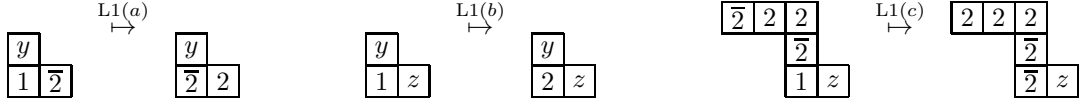

        \begin{displaymath}
          \begin{array}{ccc@{\hskip 4\cellsize}ccc@{\hskip 4\cellsize}ccc}
            \tableau{\\ y \\ 1 & \st{2} } & \stackrel{\text{L}1(a)}{\mapsto} & \tableau{\\ y \\ \st{2} & 2 } &
            \tableau{\\ y \\ 1 & z }      & \stackrel{\text{L}1(b)}{\mapsto} & \tableau{\\ y \\ 2 & z } &
            \tableau{\st{2} & 2 & 2 \\ & & \st{2} \\ & & 1 & z } & \stackrel{\text{L}1(c)}{\mapsto} &
            \tableau{ 2 & 2 & 2 \\ & & \st{2} \\ & & \st{2} & z }  
          \end{array}
        \end{displaymath}
        \caption{\label{fig:shifted-lower-1}An illustration of the shifted lowering operators with $x=i=1$.}
      \end{figure}

    \item[(L2)] \begin{itemize}
      \item[(a)] if $x=\st{i}$ and $y = i$, then $\st{f}_i$ changes $x$ to $i$ and changes $y$ to $\st{i+1}$;
      \item[(b)] else if $x=\st{i}$ and $z$ does not exist or $z > \st{i+1}$, then $\st{f}_i$ changes $x$ to $\st{i+1}$;
      \item[(c)] else if $x=\st{i}$, then $\st{f}_i$ changes $x$ to $i$ and changes the first entry $i$ southwest along the $i$-ribbon containing $x$ that is not followed by $i$ or $\st{i+1}$ to $\st{i+1}$.
      \end{itemize}

      \begin{figure}[ht]
        \begin{displaymath}
          \begin{array}{ccc@{\hskip 4\cellsize}ccc@{\hskip 4\cellsize}ccc}
            \tableau{ 1 \\ \st{1} & z } & \stackrel{\text{L}2(a)}{\mapsto} & \tableau{ \st{2} \\ 1 & z } &
            \tableau{ y \\ \st{1} & z } & \stackrel{\text{L}2(b)}{\mapsto} & \tableau{ y \\ \st{2} & z } &
            \tableau{ y \\ \st{1} & 1 & 1 & \st{2} \\ & & \st{1} & 1 } & \stackrel{\text{L}2(c)}{\mapsto} &
            \tableau{ y \\ 1 & 1 & 1 & \st{2} \\ & & \st{1} & \st{2} } 
          \end{array}
        \end{displaymath}
        \caption{\label{fig:shifted-lower-2}An illustration of the shifted lowering operators with $x=\st{i}=\st{1}$.}
      \end{figure}

    \end{itemize}
  \label{def:shifted-lower}
\end{definition}
      
The rules for the shifted lowering operators are illustrated case by case in Figures~\ref{fig:shifted-lower-1} and \ref{fig:shifted-lower-2}. Figure~\ref{fig:raise-B} shows the lowering operators applied to the left and middle tableaux, where the changed cells are indicated with a circle. This is, in fact, the complete $4$-string through these tableaux. A complete example on $\SSHT_3((3,1,0))$ is shown in Figure~\ref{fig:P31-A}.

\begin{figure}[ht]
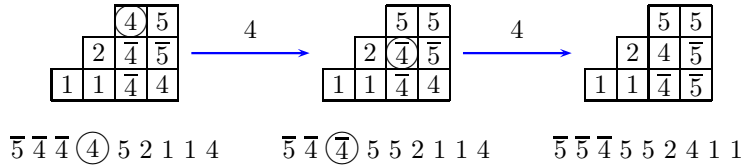

  \begin{center}
    \begin{displaymath}
      \begin{array}{c@{\hskip 2\cellsize}c@{\hskip 2\cellsize}c}
        \rnode{A}{\tableau{& & \cir{4} & 5 \\ & 2 & \st{4} & \st{5} \\ 1 & 1 & \st{4} & 4}} &
        \rnode{B}{\tableau{& & 5 & 5 \\ & 2 & \cir{\st{4}} & \st{5} \\ 1 & 1 & \st{4} & 4}} &
        \rnode{C}{\tableau{& & 5 & 5 \\ & 2 & 4 & \st{5} \\ 1 & 1 & \st{4} & \st{5}}} \\ \\
        \st{5} \ \st{4} \ \st{4} \ \raisebox{-.5ex}{$\cir{4}$} \ 5 \ 2 \ 1 \ 1 \ 4 &
        \st{5} \ \st{4} \ \raisebox{-.5ex}{$\cir{\st{4}}$} \ 5 \ 5 \ 2 \ 1 \ 1 \ 4 &
        \st{5} \ \st{5} \ \st{4} \ 5 \ 5 \ 2 \ 4 \ 1 \ 1 
      \end{array}
      \psset{nodesep=5pt,linewidth=.2ex}
      \ncline[linecolor=blue]{->} {A}{B} \naput{4}
      \ncline[linecolor=blue]{->} {B}{C} \naput{4}
    \end{displaymath}
    \caption{\label{fig:raise-B}A complete $4$-string on semistandard shifted tableaux of shape $(4,3,2)$, with the hook reading word indicated below.}
  \end{center}
\end{figure}

Note that the total number of marked cells stays the same except for Case (1)(d), where it increases by one. From the definition, it is not obvious that these operators are well-defined nor that the result is again a semistandard shifted tableau. However, both are indeed the case, as will be shown in Theorem~\ref{thm:shifted-well}.

\begin{figure}[ht]
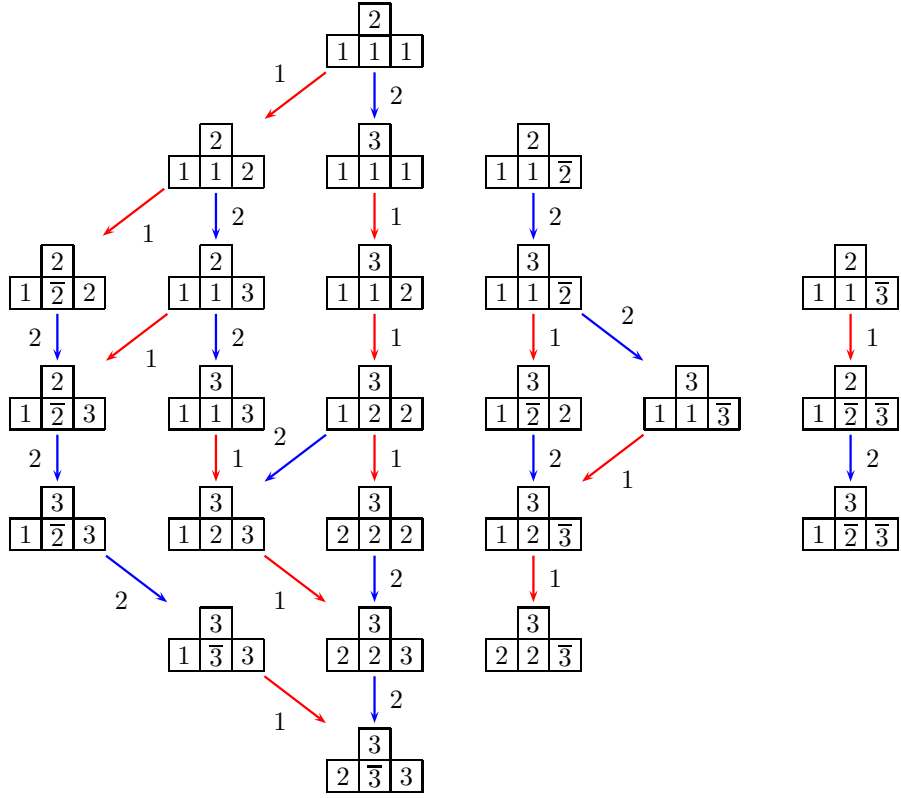

  \begin{center}
    \begin{displaymath}
      \begin{array}{c@{\hskip 2\cellsize}c@{\hskip 2\cellsize}c@{\hskip 2\cellsize}c@{\hskip 2\cellsize}c@{\hskip 2\cellsize}c}
        & & \rnode{c1}{\tableau{ & 2 \\ 1 & 1 & 1}} & & & \\[7ex]
        & \rnode{b2}{\tableau{ & 2 \\ 1 & 1 & 2}} & \rnode{c2}{\tableau{ & 3 \\ 1 & 1 & 1}} & \rnode{d2}{\tableau{ & 2 \\ 1 & 1 & \st{2}}} & & \\[7ex]
        \rnode{a3}{\tableau{ & 2 \\ 1 & \st{2} & 2}} & \rnode{b3}{\tableau{ & 2 \\ 1 & 1 & 3}} & \rnode{c3}{\tableau{ & 3 \\ 1 & 1 & 2}} & \rnode{d3}{\tableau{ & 3 \\ 1 & 1 & \st{2}}} & & \rnode{f3}{\tableau{ & 2 \\ 1 & 1 & \st{3}}}\\[7ex]
        \rnode{a4}{\tableau{ & 2 \\ 1 & \st{2} & 3}} & \rnode{b4}{\tableau{ & 3 \\ 1 & 1 & 3}} & \rnode{c4}{\tableau{ & 3 \\ 1 & 2 & 2}} & \rnode{d4}{\tableau{ & 3 \\ 1 & \st{2} & 2}} & \rnode{e4}{\tableau{ & 3 \\ 1 & 1 & \st{3}}} & \rnode{f4}{\tableau{ & 2 \\ 1 & \st{2} & \st{3}}}\\[7ex]
        \rnode{a5}{\tableau{ & 3 \\ 1 & \st{2} & 3}} & \rnode{b5}{\tableau{ & 3 \\ 1 & 2 & 3}} & \rnode{c5}{\tableau{ & 3 \\ 2 & 2 & 2}} & \rnode{d5}{\tableau{ & 3 \\ 1 & 2 & \st{3}}} & & \rnode{f5}{\tableau{ & 3 \\ 1 & \st{2} & \st{3}}}\\[7ex]
        & \rnode{b6}{\tableau{ & 3 \\ 1 & \st{3} & 3}} & \rnode{c6}{\tableau{ & 3 \\ 2 & 2 & 3}} & \rnode{d6}{\tableau{ & 3 \\ 2 & 2 & \st{3}}} & & \\[7ex]
        & & \rnode{c7}{\tableau{ & 3 \\ 2 & \st{3} & 3}} & & & 
      \end{array}
      \psset{nodesep=2pt,linewidth=.2ex}
      % RANK 1 TO 2
      \ncline[linewidth=.2ex,linecolor=red]{->} {c1}{b2} \nbput{1}
      \ncline[linecolor=blue]{->}  {c1}{c2} \naput{2}
      % RANK 2 TO 3
      \ncline[offset=2pt,linewidth=.2ex,linecolor=red]{<-} {a3}{b2} \nbput{1}
      \ncline[linecolor=blue]{->}  {b2}{b3} \naput{2}
      \ncline[linewidth=.2ex,linecolor=red]{->} {c2}{c3} \naput{1}
      \ncline[linecolor=blue]{->}  {d2}{d3} \naput{2}
      % RANK 3 TO 4
      \ncline[linecolor=blue]{->}  {a3}{a4} \nbput{2}
      \ncline[linewidth=.2ex,linecolor=red]{<-} {a4}{b3} \nbput{1} 
      \ncline[linecolor=blue]{->}  {b3}{b4} \naput{2}
      \ncline[linewidth=.2ex,linecolor=red]{->} {c3}{c4} \naput{1}
      \ncline[linewidth=.2ex,linecolor=red]{->} {d3}{d4} \naput{1}
      \ncline[linecolor=blue]{->}  {d3}{e4} \naput{2}
      \ncline[linewidth=.2ex,linecolor=red]{<-} {f4}{f3} \nbput{1} 
      % RANK 4 TO 5
      \ncline[linecolor=blue]{->}  {a4}{a5} \nbput{2}
      \ncline[linewidth=.2ex,linecolor=red]{->} {b4}{b5} \naput{1}
      \ncline[linecolor=blue]{->}  {c4}{b5} \nbput{2}
      \ncline[linewidth=.2ex,linecolor=red]{<-} {c5}{c4} \nbput{1}
      \ncline[linecolor=blue]{->}  {d4}{d5} \naput{2}
      \ncline[linewidth=.2ex,linecolor=red]{->} {e4}{d5} \naput{1}
      \ncline[linecolor=blue]{->}  {f4}{f5} \naput{2}
      % RANK 5 TO 6
      \ncline[linecolor=blue]{->}  {a5}{b6} \nbput{2}
      \ncline[linewidth=.2ex,linecolor=red]{->}{b5}{c6} \nbput{1}
      \ncline[linecolor=blue]{->}  {c5}{c6} \naput{2}
      \ncline[linewidth=.2ex,linecolor=red]{<-} {d6}{d5} \nbput{1}
      % RANK 6 TO 7
      \ncline[linewidth=.2ex,linecolor=red]{->}{b6}{c7} \nbput{1}
      \ncline[linecolor=blue]{->}  {c6}{c7} \naput{2}
    \end{displaymath}
    \caption{\label{fig:P31-A}The crystal structure on semistandard shifted tableaux of shape $(3,1)$ with entries $\{\st{1},1,\st{2},2,\st{3},3\}$ and no marks on the main diagonal.}
  \end{center}
\end{figure}

In order to analyze how the shifted crystal operators change the lengths of monochromatic paths, we introduce the notion of \emph{blocked} entries of a semistandard shifted tableau $T$, determined via the hook reading word $w(T)$. Blocking will depend upon a parameter $i$, for $1 \leq i < n$, and for $i$-blocking, only entries $\st{i},i,\st{i+1},i+1$ will be considered.

\begin{definition}
  Let $i \geq 1$ be an index and $T$ a semistandard shifted tableau.  A pair of entries $y,x$ of $T$, with $y \in \{\st{i+1},i+1\}$ and $x \in \{\st{i},i\}$ are \emph{$i$-paired} if $y$ occurs before $x$ in the hook reading word of $T$ and every entry $\st{i}, i, \st{i+1}, i+1$ that lies between them in the hook reading word is $i$-blocked. The $i$-paired entries become $i$-blocked, and we continue the operation recursively till there are no more possible $i$-pairs. An entry $\st{i},i,\st{i+1},i+1$ of $T$ is \emph{$i$-free} if it is not part of an $i$-pair.
  \label{def:block}
\end{definition}

Figure~\ref{fig:free-B} shows the $4$-blocked entries of the hook reading word paired with under brackets and the $4$-free entries are indicated in red. Notice that for each tableau, there are exactly two $4$-free entries.

\begin{figure}[ht]
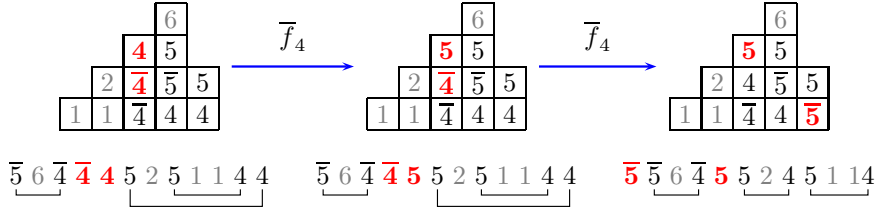

  \begin{center}
    \begin{displaymath}
      \begin{array}{c@{\hskip 1.5\cellsize}c@{\hskip 1.5\cellsize}c}
        \rnode{A}{\tableau{& & & {\gray 6} \\ & & \mathbf{\red 4} & 5 \\ & {\gray 2} & \mathbf{\red\st{4}} & \st{5} & 5 \\ {\gray 1} & {\gray 1} & \st{4} & 4 & 4}} &
        \rnode{B}{\tableau{& & & {\gray 6} \\ & & \mathbf{\red 5} & 5 \\ & {\gray 2} & \mathbf{\red\st{4}} & \st{5} & 5 \\ {\gray 1} & {\gray 1} & \st{4} & 4 & 4}} &
        \rnode{C}{\tableau{& & & {\gray 6} \\ & & \mathbf{\red 5} & 5 \\ & {\gray 2} & 4 & \st{5} & 5 \\ {\gray 1} & {\gray 1} & \st{4} & 4 & \mathbf{\red\st{5}}}} \\ \\
        \rnode{a1}{\st{5}} \ {\gray 6} \ \rnode{b1}{\st{4}} \ \mathbf{\red\st{4}} \ \mathbf{\red 4} \ \rnode{c1}{5} \ {\gray 2} \ \rnode{d1}{5} \ {\gray 1 \ 1} \ \rnode{e1}{4} \ \rnode{f1}{4} &
        \rnode{a2}{\st{5}} \ {\gray 6} \ \rnode{b2}{\st{4}} \ \mathbf{\red\st{4}} \ \mathbf{\red 5} \ \rnode{c2}{5} \ {\gray 2} \ \rnode{d2}{5} \ {\gray 1 \ 1} \ \rnode{e2}{4} \ \rnode{f2}{4} &
        \mathbf{\red\st{5}} \ \rnode{a3}{\st{5}} \ {\gray 6} \ \rnode{b3}{\st{4}} \ \mathbf{\red 5} \ \rnode{c3}{5} \ {\gray 2} \ \rnode{d3}{4} \ \rnode{e3}{5} \ {\gray 1 \ 1 } \rnode{f3}{4} 
      \end{array}
      \psset{nodesep=5pt,linewidth=.2ex}
      \ncline[linecolor=blue]{->} {A}{B} \naput{\st{f}_4}
      \ncline[linecolor=blue]{->} {B}{C} \naput{\st{f}_4}
      \ncdiag[linewidth=.1ex,nodesep=2pt,angleA=-90,angleB=-90,arm=0.5ex] {a1}{b1} 
      \ncdiag[linewidth=.1ex,nodesep=2pt,angleA=-90,angleB=-90,arm=1.5ex] {c1}{f1} 
      \ncdiag[linewidth=.1ex,nodesep=2pt,angleA=-90,angleB=-90,arm=0.5ex] {d1}{e1} 
      \ncdiag[linewidth=.1ex,nodesep=2pt,angleA=-90,angleB=-90,arm=0.5ex] {a2}{b2} 
      \ncdiag[linewidth=.1ex,nodesep=2pt,angleA=-90,angleB=-90,arm=1.5ex] {c2}{f2} 
      \ncdiag[linewidth=.1ex,nodesep=2pt,angleA=-90,angleB=-90,arm=0.5ex] {d2}{e2} 
      \ncdiag[linewidth=.1ex,nodesep=2pt,angleA=-90,angleB=-90,arm=0.5ex] {a3}{b3} 
      \ncdiag[linewidth=.1ex,nodesep=2pt,angleA=-90,angleB=-90,arm=0.5ex] {c3}{d3} 
      \ncdiag[linewidth=.1ex,nodesep=2pt,angleA=-90,angleB=-90,arm=0.5ex] {e3}{f3} 
    \end{displaymath}
    \caption{\label{fig:free-B}A $4$-string on $\SSHT_6(5,4,2,1)$, with $4$-blocked entries bracketed and $4$-free entries indicated in red.}
  \end{center}
\end{figure}

We show below that the $i$-free entries of $T$ occur with all entries $\st{i},i$ preceding all entries $\st{i+1},i+1$ in the hook reading word $w(T)$. This implies the following alternative characterization for the entry of $T$ on which $\st{f}_i$ acts.

\begin{lemma}
  For $T$ a semistandard shifted tableau and $i \geq 1$ an index, $m_i(w(T))$ is number of $i$-free entries $\st{i}, i$ of $T$, and if $m_i(w(T)) > 0$, then the smallest index $p$ for which $m_i(w(T),p) = m_i(w(T))$ occurs at the rightmost $i$-free entry $\st{i}, i$.
  \label{lem:free}
\end{lemma}

\begin{proof}
  We claim all $i$-free entries $\st{i},i$ precede all $i$-free entries $\st{i+1},i+1$ in the hook reading word $w(T)$. To see this, note that $i$-paired entries are nested, and no $i$-free entry may occur between an $i$-pair. Therefore, if $y,x$ are $i$-free entries with $y \in \{\st{i+1},i+1\}$, $x\in \{\st{i},i\}$ and $y$ preceding $x$ in $w(T)$ such that no intermediate entries are $i$-free, then all intermediate entries $\st{i},i,\st{i+1},i+1$ are paired with another entry between $y$ and $x$, therefore $y$ and $x$ have the same number of entries equal to $\st{i+1},i+1$ as equal to $\st{i},i$ between them in $w(T)$, contradicting that $y,x$ are free. Thus no $i$-free entry $\st{i+1},i+1$ may precede an $i$-free entry $\st{i},i$ in $w(T)$, establishing the claim.

  Recall that $m_i(w(T)) = \max_r(m_i(w(T),r))$ is positive if and only if $\st{f}_i(T) \neq 0$, and in this case the smallest index $p$ for which $m_i(w(T),p) = m_i(w(T))$ occurs at an entry $\st{i}$ or $i$. By the previous claim, if $r<s$ are the indices in $w(T)$ of an $i$-blocked pair, then $m_i(w(T),r-1) = m_i(w(T),s)$. Therefore the smallest index $p$ for which $m_i(w(T),p) = m_i(w(T))$ must occur at the rightmost $i$-free entry $\st{i}$ or $i$. Furthermore, since $w(T)_p$ cannot lie between any $i$-blocked pair, we have $m_i(w(T),p)$ is the number of $i$-free entries $\st{i}$ or $i$ weakly preceding $w(T)_p$, which by the previous claim, is simply the number of $i$-free entries $\st{i}$ or $i$ of $T$.
\end{proof}

Using blocked and free entries, we now prove the shifted lowering operators are well-defined. The proof is a case by case analysis that acting as prescribed in Definition~\ref{def:shifted-lower} results in a semistandard shifted tableau.

\begin{theorem}
  For any strict partition $\gamma$, the shifted lowering operators $\{\st{f}_i\}_{1 \leq i < n}$ are well-defined maps $\st{f}_i : \SSHT_n(\gamma) \rightarrow \SSHT_n(\gamma) \cup \{0\}$. 
  \label{thm:shifted-well}
\end{theorem}

\begin{proof}
  Let $T$ be a semistandard shifted tableaux and assume $m_i(w(T)) > 0$ so that $\st{f}_i$ acts non-trivially on $T$. Let $x, y, z$ be as in Definition~\ref{def:shifted-lower}, and let their indices in the hook reading word be $p_x, p_y, p_z$, respectively. By definition, $p_x$ is the smallest number satisfying $m_i(w(T),p) = m_i(w(T))$. By Lemma~\ref{lem:free}, every cell labeled $\st{i+1}$ or $i+1$ that precedes $x$ in $w(T)$ is $i$-paired with some cell labeled $\st{i}/i$  that lies between itself and $x$. Similarly, every cell labeled $\st{i}/i$ that follows $x$ in $w(T)$ is $i$-paired with some cell labeled $\st{i+1}/i+1$ that lies between $x$ and itself. In particular, if $x$ has label $i$, then $z$ cannot also have label $i$ since it follows immediately after $x$ in $w(T)$, and if $x$ has label $\st{i}$ then $y$ cannot have label $\st{i}$ for exactly the same reason.

  \emph{Case L1(a)}: We must verify two things to show that the shifted tableaux rules are not violated: $x$ is not on a main diagonal, and the cell above $z$ is not $\st{i+1}/i+1$. We first show that $x$ can not be on the main diagonal. As $z=\st{i+1}$, it needs to be $i$-paired with a cell labeled $\st{i}/i$ after it and before $x$ in the reading word, which can only happen in the column of $x$ (Figure \ref{fig:lower-proof-1a}, left). Assume there are $t$ cells labeled $\st{i}$ in the column of $x$. The cells left adjacent to those can only be labeled $i$ or $\st{i+1}$ as columns and rows weakly increase. Furthermore, every $\st{i+1}$ in the column of $z$ needs to be paired with a $\st{i}$ between $x$ and itself, so at most $t-1$ of these can be $\st{i+1}$, and the bottom one most must be labeled $i$. Let us call this cell $x'$ (Figure \ref{fig:lower-proof-1a}, middle). As $x'=i$  it must be $i$-paired with some $i+1$ between $x$ and $x'$ in the reading word ($\st{i+1}$ never comes after $i$). By the tableaux rules, such a cell must be on the row above. If there are $t'$ such cells labeled $i$, they are $i$-paired with $t'$ cells labeled $i+1$ on the row above, which means the leftmost $i$ (say $x''$) is followed by an $\st{i+1}$. To $i$-pair this $\st{i+1}$, the cell under $x''$ must be labeled $\st{i}$, which takes us back to our starting point (see Figure \ref{fig:lower-proof-1a}, right). Since our shape finite, this case can not happen.
  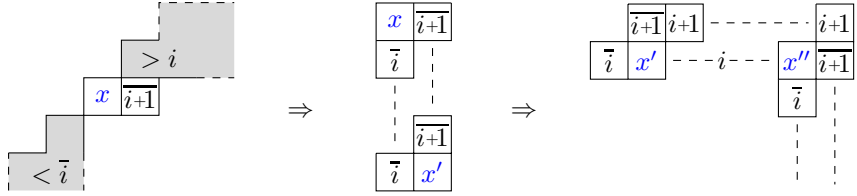
\begin{figure}[ht]
    \begin{center}
      \begin{tikzpicture}[scale=.5]
	\fill[gray!30!white] (3,3)--(3,4)--(4,4)--(4,3)--(3,3) (4,3)--(4,5)--(6,5)--(6,3)--(4,3);
	\fill[gray!30!white] (2,2)--(2,0)--(0,0)--(0,1)--(1,1)--(1,2)--(2,2);
	\draw (0,1)--(1,1)--(1,2)--(2,2)--(2,3)--(3,3)--(3,4)--(4,4) (2,2)--(2,1);
	\draw[dashed] (0,0)--(0,1) (4,4)--(4,5)--(6,5) (5,3)--(6,3) (2,1)--(2,0);
	\draw (3,3)--(5,3);
	
	\node at (1.1, .5) {$<\st{i}$};
	\node at (4,3.5) {$>i$};
	\node[blue]  at (2.5,2.5) {$x$};
	\node at (3.5, 2.5) {$\st{i\texplus1}$};
	\draw (2,2)--(4,2)--(4,3) (3,2)--(3,3);
	\end{tikzpicture}\qquad \raisebox{1cm}{$\Rightarrow$} \qquad
		\begin{tikzpicture}[scale=.5]
	\draw (2,2)--(2,3)--(4,3);
	%\draw[dashed] (0,1)--(1,1)--(1,2);

	\node[blue]  at (2.5,2.5) {$x$};
		\node at (2.5,1.5) {$\st{i}$};
		\draw (2,2)--(2,1)--(3,1)--(3,2);
	\draw[dashed](2.5,0.8)--(2.5,-.8)(3.5,1.8)--(3.5,.2);
%	\node at (2.2, 1) {$\st{i}$};
%	\node at (4.1, 1) {$\st{i \texplus 1}$};
	\node at (3.5, 2.5) {$\st{i\texplus1}$};
	\draw (2,-1)--(4,-1)--(4,-2)--(2,-2)--(2,-1) (4,-1)--(4,0)--(3,0)--(3,-2);
	\node at (3.5,-.5) {$\st{i\texplus1}$};
 \node at (2.5,-1.5) {$\st{i}$};
	\node[blue] at (3.5,-1.5) {$x'$};
	\draw (2,2)--(4,2)--(4,3) (3,2)--(3,3);
%	\draw [red] (3.8,-1.5)--(4.65,-1.15);
%	\node[red] at (5,-1) {$x'$};
	\end{tikzpicture}\qquad \raisebox{1cm}{$\Rightarrow$}\qquad
		\begin{tikzpicture}[scale=.5]
	\draw (7,0)--(8,0)--(8,-1);
	\draw (2,0)--(2,-2)--(3,-2)--(3,-1)--(4,-1)--(4,0)--(2,0) (3,0)--(3,-1)--(2,-1);
	\node at (3.5,-.5) {${i\texplus1}$};
 \node[blue] at (2.5,-1.5) {$x'$};
	\node at (2.5,-0.5) {$\st{i\texplus1}$};
\draw[white] (2,-5)--(3,-5);
	\node at (1.5,-1.5) {${\st{i}}$};
	\draw (2,-1)--(1,-1)--(1,-2)--(2,-2);
%	\draw[white] (2,2)--(4,2)--(4,3) (3,2)--(3,3);
	\draw[dashed](3.2,-1.5)--(4.3,-1.5) (4.7,-1.5)--(5.8,-1.5) (4.2,-.5)--(6.8,-.5);
	\node at (4.5, -1.5) {$i$};
		\draw (6,-1)--(8,-1)--(8,-2)--(6,-2)--(6,-1) (7,0)--(7,-2);
	\node at (7.5,-.5) {$i\texplus1$};
 \node[blue] at (6.5,-1.5) {$x''$};
  \node at (6.5,-2.5) {$\st{i}$};
  \draw (6,-2)--(6,-3)--(7,-3)--(7,-2);
	\node at (7.5,-1.5) {$\st{i\texplus1}$};
	%\node[red] at (3,-2.8) {$x'$};		
	%\draw [red] (2.5,-1.8)--(2.85,-2.65);
	%\draw [red] (6.2,-1.5)--(5.2,-2.2);
	%\node[red] at (5,-2.1) {$x''$};
	\draw[dashed](6.5,-3.2)--(6.5,-4.8)(7.5,-2.2)--(7.5,-4.8);
	        \end{tikzpicture}
	        \caption{Case L1(a), showing $x$ can not be on the main diagonal.\label{fig:lower-proof-1a}}
    \end{center}
    \end{figure}
  The second step is to show that the cell above $z$ can not be labeled $\st{i+1}/i+1$. We can eliminate the $\st{i+1}$ case easily, as $x$ is not on the main diagonal, there is a cell above $x$ and there is no valid way to fill it (Figure \ref{fig:lower-proof-1a2}, left). Assume the cell above $z$ is labeled $i+1$. As it is not on the main diagonal, any $\st{i}$ comes after it in the reading word, its $i$-pair is an entry $i$ between $x$ and itself, so there needs to be another cell labeled $i$ to the left of $x$. If $i$ is not on the main diagonal, then it has a cell above it that is labeled at least $\st{i+1}$, so the cell above $x$ also must be $i+1$, and we need another $i$ to pair it. This continues till we hit the main diagonal, and the $i+1$ on the diagonal still need an $i$-pair. As there is no more place left for an $i$, its $i$-pair must be an entry $\st{i}$ after it on the reading word, which can only happen in the column to the left. This leads to a contradiction as there is no way to fill the cell to the right of $\st{i}$ (see Figure \ref{fig:lower-proof-1a2}, right). 
  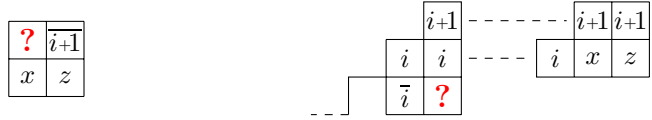
\begin{figure}[ht]\begin{center}
      \begin{tikzpicture}[scale=.5]
	\draw (0,0) grid (2,2);
        \node[red] at (.5,1.5) {\Large \textbf{?}};
        \node at (.5,.5) {$x$};
        \node at (1.5,.5) {$z$};
        \node at (1.5,1.5) {$\st{i\texplus1}$};
        \draw[white] (0,-.5)--(1,-.5) (0,2.5)--(1,2.5);
      \end{tikzpicture}
      \qquad\qquad\qquad \qquad
      \begin{tikzpicture}[scale=.5]
	\draw (0,0) grid (2,2);
	\draw (1,2)--(1,3)--(2,3)--(2,2) (0,1)--(-1,1)--(-1,0);
	\draw[dashed] (-1,0)--(-2,0) (2.2,1.5)--(3.8,1.5) (2.2,2.5)--(4.8,2.5);
	\node[red] at (1.5,.5) {\Large \textbf{?}};
        \node at (1.5,1.5) {$i$};
        \node at (.5,1.5) {$i$};
        \node at (.5,.5) {$\st{i}$};
        \node at (1.5, 2.5) {$i\texplus 1$};
        \draw (5,1) grid (7,3);
        \draw (5,1)--(4,1)--(4,2)--(5,2);
        \node at (4.5,1.5) {$i$};
        \node at (5.5,1.5) {$x$};
        \node at (6.5,1.5) {$z$};
        \node at (5.5,2.5) {$i\texplus 1$};
        \node at (6.5,2.5) {$i\texplus1$};
      \end{tikzpicture}
      \caption{Case L1(a), showing  the  cell above $z$ can not be $i+1$ or $\st{i+1}$.\label{fig:lower-proof-1a2}}\end{center}
  \end{figure}

  \emph{Case L1(b)}: In this case we assume $z \geq i+1$ (since we are not in case L1(a)) and either $y$ does not exist or $y > i+1$, so there are no possible row or column violations that can arise from changing $x$ from $i$ to $i+1$. 

  \emph{Cases L1(c) and L1(d)}: For these cases, we assume $z \geq i+1$ (since we are not in case L1(a)) and $y = \st{i+1}/i+1$ (since we are not in case L1(b)), therefore there are no possible row or column violations that can arise from changing $x$ from $i$ to $\st{i+1}$. Let $u$ denote the entry at the head of the $(i+1)$-ribbon containing $y$. If $u = i+1$ (case L1(d)), then no change occurs here, and if $u = \st{i+1}$ (case L1(c)), then since the $(i+1)$-ribbon terminates at $u$, the cell above $u$ is larger than $i+1$, so removing the marking on $u$ cannot create a row or column violation either.

  \emph{Case L2(a)}: As there can be at most one $\st{i}$ on a row, changing $x$ from $\st{i}$ to $i$ does not create a row or column violation, therefore we only need to check that no violations arise in changing $y$ from $i$ to $\st{i+1}$. Note that $y$ can not be on the main diagonal, since otherwise $x$ would immediately precede $y$, giving $m_i(w(T),p_y)>m_i(w(T),p_x)$. So the only potential is when the cell right adjacent to $y$ is labeled $\st{i+1}$. Assume it is. Then it must be $i$-paired with a cell labeled $\st{i}$ that comes before $x$ in $w(T)$. However, this can only happen in the column of $x$, so the cell below $x$ must be $\st{i}$. Assume there are $k$ cells labeled $\st{i}$ below $x$. Then we can have at most $k$ cells labeled $\st{i+1}$ in the next column, so the bottom two $\st{i}$'s need to have $i$'s left adjacent to them, which can not happen by our column rules (Figure \ref{fig:lower_proof_2}, left).
  \begin{figure}[ht]
    \centering
    \begin{tikzpicture}[scale=.5]
      \draw (0,0) grid (2,1);
      \draw (1,1)--(1,2)--(2,2)--(2,1) (1,2)--(1,3)--(2,3)--(2,2) (0,1)--(0,2)--(1,2);
      \draw[dashed] (.5,2.2)--(.5,3.8) (1.5,3.2)--(1.5,4.8);
      \node at (.5,.5) {$\st{i}$};
      \node at (.5,1.5) {$\st{i}$};
      \node[red] at (1.5,.5) {\Large \textbf{?}};
      \node at (1.5,2.5) {$\st{i\texplus1}$};
      \node at (1.5,1.5) {$i$};
      \draw (0,6) grid (2,5);
      \draw (0,5)--(0,4)--(1,4)--(1,5);
      \node at (.5,4.5) {$x$};
      \node at (.5,5.5) {$i$};
      \node at (1.5,5.5) {$\st{i\texplus1}$};
    \end{tikzpicture}\qquad \qquad \qquad
    \begin{tikzpicture}[scale=.5]
      \draw (0,0) grid (2,1);
      \draw (1,1)--(1,2)--(2,2)--(2,1);
      \draw[dashed] (.5,1.2)--(.5,3.8) (1.5,2.2)--(1.5,4.8) (2.2,.5)--(3.8,.5) (-.2,5.5)--(-1.8,5.5);
      \draw (4,0) grid (5,1) (-2,5) grid (-4,6);
      \node at (-3.5, 5.5) {$x$};
      \node at (-2.5, 5.5) {$i$};
      \node at (4.5,.5) {$i$};
      \node at (.5,.5) {$\st{i}$};
      \node at (1.5,.5) {$i$};
      \node at (1.5,1.5) {$\st{i\texplus1}$};
      \draw (0,6) grid (2,5);
      \draw (0,5)--(0,4)--(1,4)--(1,5);
      \node at (.5,4.5) {$\st{i}$};
      \node at (.5,5.5) {$i$};
      \node at (1.5,5.5) {$\st{i\texplus1}$};
    \end{tikzpicture}\qquad \qquad \qquad
    \begin{tikzpicture}[scale=.5]
      \draw (0,0) grid (2,1);
      \draw (1,1)--(1,2)--(2,2)--(2,1);
      \draw[dashed] (.5,1.2)--(.5,3.8) (1.5,2.2)--(1.5,4.8) (2.2,.5)--(3.8,.5);
      \draw (4,0) grid (5,1);
      \node at (4.5,.5) {$i$};
      \node at (.5,.5) {$\st{i}$};
      \node at (1.5,.5) {$i$};
      \node at (1.5,1.5) {$\st{i\texplus1}$};
      \draw (0,6) grid (2,5);
      \draw (0,5)--(0,4)--(1,4)--(1,5);
      \node at (.5,4.5) {$\st{i}$};
      \node at (.5,5.5) {$x$};
      \node at (1.5,5.5) {$\st{i\texplus1}$};
    \end{tikzpicture}
    \caption{Case L2: $x=\st{i}$ \label{fig:lower_proof_2}}
  \end{figure}

  \emph{Case L2(b)}: We must have $y \geq \st{i+1}$ (since we are not in case L2(a)) and either $z$ does not exist or $z > \st{i+1}$, so there are no possible row or column violations that can arise from changing $x$ from $\st{i}$ to $\st{i+1}$. 

  \emph{Case L2(c)}: Since $y \geq \st{i+1}$ (since we are not in case L2(a)) no row or column violation can arise in changing $x$ from $\st{i}$ to $i$. Therefore we only need to verify that, in the $i$-ribbon containing $x$, there is an entry $i$ that is not followed by $i$ or $\st{i+1}$ as described, and that turning it into $\st{i+1}$ creates no violations. We have two cases based on the possible values of $z$ as $i$ or $\st{i+1}$ (since we are not in case L2(b)). We first consider the case $z=i$. If the rightmost $i$ in this row is not followed by $\st{i+1}$, we are done. If it is, then it needs to be $i$-paired with a cell marked $\st{i}$ between $x$ and itself in the reading word, which can only happen in the column of the rightmost $i$. Assume this column contains $k$ cells labeled $\st{i}$. Then there can be at most $k$ cells labeled $\st{i+1}$ in the next column, so the bottom-most one needs to be followed by an $i$. If the rightmost $i$ is not followed by $\st{i+1}$, we are done. If it is, we can repeat the same argument till we find such an $i$ (Figure \ref{fig:lower_proof_2}, middle). The case $z=\st{i+1}$ is similar, as $z$ needs to be $i$-paired with a cell labeled $\st{i}$ that comes after $x$ in the reading word, and that can only happen in the column of $x$. As above, the bottom $\st{i}$ in the column must be followed by an $i$, and we can continue until we find an $i$ not followed by $\st{i+1}$ (Figure \ref{fig:lower_proof_2}, right).
\end{proof}

In order to prove that the shifted lowering operators are invertible when the image is nonzero, we offer the following explicit rule for their partial inverses.

\begin{definition}
  The \emph{shifted raising operators}, denoted by $\st{e}_i$, act on semistandard shifted tableaux by: $\st{e}_{i}(T)=0$ if $m_i(w(T),k) = m_i(w(T))$; otherwise, letting $q$ be the largest index such that $m_i(w(T),q) = m_i(w(T))$, letting $x$ denote the entry of $T$ corresponding to $w_q$, and letting $y$ be the entry south of $x$ and $z$ the entry west of $x$, we have
    \begin{itemize}
    \item[(R1)] \begin{itemize}
    \item[(a)] if $x=i+1$ and $z=\st{i+1}$, then $\st{e}_i$ changes $x$ to $\st{i+1}$ and changes $z$ to $i$;
    \item[(b)] else if $x=i+1$ and $y$ does not exist or $y < i$, then $\st{e}_i$ changes $x$ to $i$
    \item[(c)] else if $x=i+1$ then $\st{e}_i$ changes $x$ to $\st{i+1}$ and changes the first entry $\st{i+1}$ southwest along the $i+1$-ribbon containing $x$ that is not above an $i$ or $\st{i+1}$ to $i$;
    \end{itemize}
      
      \begin{figure}[ht]
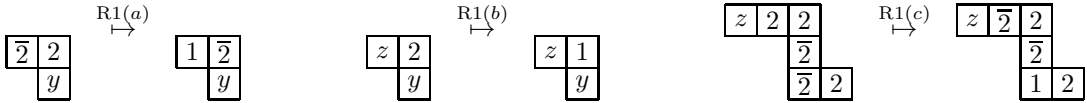

        \begin{displaymath}
          \begin{array}{ccc@{\hskip 4\cellsize}ccc@{\hskip 4\cellsize}ccc}
            \tableau{\\ \st{2}& 2 \\  & y } & \stackrel{\text{R}1(a)}{\mapsto} & \tableau{\\ 1 & \st{2} \\ & y } &
            \tableau{\\ z & 2 \\ & y }      & \stackrel{\text{R}1(b)}{\mapsto} & \tableau{\\ z &1 \\  & y } &
            \tableau{ z  & 2 & 2 \\ &  & \st{2} \\ &  & \st{2} &2 } & \stackrel{\text{R}1(c)}{\mapsto} &
            \tableau{ z & \st{2} & 2  \\  &  & \st{2}\\&  & 1 &2} 
          \end{array}
        \end{displaymath}
        \caption{\label{fig:shifted-raise-1}An illustration of the shifted raising operators with $x=i=1$.}
      \end{figure}
      
    \item[(R2)] \begin{itemize}
    \item[(a)] if $x=\st{i+1}$ and $y = i$, then $\st{e}_i$ changes $x$ to $i$ and changes $y$ to $\st{i}$;
    \item[(b)] else if $x=\st{i+1}$ and $z$ does not exist or $z < \st{i}$, then $\st{e}_i$ changes $x$ to $\st{i}$;
    \item[(c)] else if $x=\st{i+1}$, and the northeastern-most cell on the $(i)$-ribbon containing $z$ is not on the main diagonal, then $\st{e}_i$ adds a marking to that cell and changes $x$ to $\st{i}$;
    \item[(d)] else if $x=\st{i+1}$, then $\st{e}_i$ changes $x$ to $i$.
    \end{itemize}
      
      \begin{figure}[ht]
        \begin{displaymath}
          \begin{array}{ccc@{\hskip 4\cellsize}ccc@{\hskip 4\cellsize}ccc}
            \tableau{ z &  \st{2} \\ & 1 } & \stackrel{\text{R}2(a)}{\mapsto} & \tableau{ z & 1 \\  & \st{1} } &
            \tableau{ z& \st{2} \\  & y } & \stackrel{\text{R}2(b)}{\mapsto} & \tableau{ z& \st{1} \\  & y } &
            \tableau{ 1 \\ \st{1} & 1 & 1 & \st{2} \\ & & \st{1} &  y} & \stackrel{\text{R}2(c)}{\mapsto} &
            \tableau{ \st{1} \\ \st{1} & 1 & 1 & 1 \\ & &  & y }  
          \end{array}
        \end{displaymath}
        \caption{\label{fig:shifted-raise-2}An illustration of the shifted raising operators with $x=\st{i}=\st{1}$.}
      \end{figure}
      
    \end{itemize}
    \label{def:shifted-raise}
\end{definition}

The proof that the shifted raising operators are well-defined is completely analogous to that for the shifted lowering operators. Similarly, if $\st{e}_i(T) \neq 0$, then the largest index $q$ for which $m_i(w(T),q-1) = m_i(w(T))$ occurs at the leftmost $i$-free entry $\st{i+1}$ or $i+1$, parallel to Lemma~\ref{lem:free}. Details of both proofs are omitted.

\subsection{Verification of local axioms}
%%%%%%%%%%%%%%%%%%%%%%%%%%%%%%%%%%%%%%%%%%%%%%%%%%%%%%%%%%%%%%%%
\label{sec:local-B}

To prove that our operators define a normal crystal, we begin by showing that they satisfy the conditions required to be a crystal.

\begin{theorem}
  For any strict partition $\gamma$, the shifted raising and lowering operators $\st{e}_i, \st{f}_i$ for $i=1,2,\ldots,r$ together with the usual weight map define a crystal on $\SSHT_{r+1}(\gamma)$.
  \label{thm:crystal}
\end{theorem}

\begin{proof}
  From the definitions of the shifted operators, we have $\wt(\st{f}_i(T)) = \wt(T) + \alpha_i$ and $\wt(\st{e}_i(T)) = \wt(T) - \alpha_i$. Therefore there are two statements to prove based on Definition~\ref{def:base-A}: that $\st{e}_i(T)=T^{\prime}$ if and only if $\st{f}_i(T^{\prime}) = T$, and that $\downf_i(T) - \upe_i(T) = (\wt(T)_i - \wt(T)_{i+1})$. Note that the second statement follows once we show $\upe_i(T)$ is the number of $i$-free entries $\st{i+1},i+1$ in $T$ and $\downf_i(T)$ is the number of $i$-free entries $\st{i},i$ of $T$, since all other entries appear in pairs whose weights cancel. We prove both statements together in cases based on Definition~\ref{def:shifted-lower} for the shifted lowering operators. 
  
  \emph{Case L1(a)}: We assume $x=1$ and $z = \st{i+1}$ in $T$, therefore $\st{f}_i(T)$ has $x=\st{i+1}$ and $z=i+1$. For an example of this case, see Figure~\ref{fig:free-L1a}. In the hook reading words, the position of $x$ in $w(T)$ is precisely the position of $z$ in $w(\st{f}_i(T))$. However, the position of $z$ in $w(T)$ is weakly left of the position of $x$ in $w(\st{f}_i(T))$, with the offset equal to the number of marked entries strictly above $z$ or strictly below $x$ in $T$ (and in $\st{f}(T)$). However, since $\st{f}_i(T)$ is semistandard, there cannot be an entry $\st{i+1}$ above $z$, so we only need to be concerned about any entries $\st{i}$ below $x$. Note that if $u = \st{i}$ is below $x$ and $v$ is the cell immediately to its right, then $\st{i+1} < v$, and since $v$ is below $z$, we also have $v \leq \st{i+1}$. Therefore $v = \st{i+1}$, therefore in this way we pair off each $\st{i}$ below $x$ with an $\st{i+1}$ below $z$. In particular, since marked entries in the column of $z$ are read immediately before marked entries in the column of $x$, every $\st{i}$ below $x$ is $i$-paired with an $\st{i+1}$ weakly below $z$ in both $w(T)$ and in $w(\st{f}_i(T))$. Therefore the number of $i$-blocked pairs is unchanged in passing from $w(T)$ to $w(\st{f}_i(T))$. Furthermore, the result is that the rightmost $i$-free entry $i$ becomes the leftmost $i$-free entry $i+1$, landing us in case R1(a) for the shifted raising operator, which will precisely undo the action.
  
  \begin{figure}[ht]
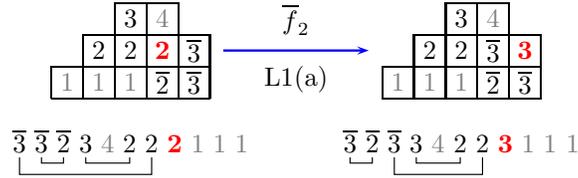

    \begin{displaymath}
      \begin{array}{c@{\hskip 3\cellsize}c}
        \rnode{A}{\tableau{& & 3 & {\gray 4} \\ & 2 & 2 & \textbf{\red 2} & \st{3} \\ {\gray 1} & {\gray 1} & {\gray 1} & \st{2} & \st{3}}} &
        \rnode{B}{\tableau{& & 3 & {\gray 4} \\ & 2 & 2 & \st{3} & \textbf{\red 3} \\ {\gray 1} & {\gray 1} & {\gray 1} & \st{2} & \st{3}}} \\ \\
        \rnode{a1}{\st{3}} \ \rnode{a2}{\st{3}} \ \rnode{a3}{\st{2}} \ \rnode{a4}{3} \ {\gray 4} \ \rnode{a5}{2} \ \rnode{a6}{2} \ \mathbf{\red 2} \ {\gray 1 \ 1 \ 1} &
        \rnode{b1}{\st{3}} \ \rnode{b2}{\st{2}} \ \rnode{b3}{\st{3}} \ \rnode{b4}{3} \ {\gray 4} \ \rnode{b5}{2} \ \rnode{b6}{2} \ \mathbf{\red 3} \ {\gray 1 \ 1 \ 1} 
      \end{array} 
      \psset{nodesep=5pt,linewidth=.2ex}
      \ncline[linecolor=blue]{->} {A}{B} \naput{\st{f}_2} \nbput{\text{L1(a)}}
      \ncdiag[linewidth=.1ex,nodesep=2pt,angleA=-90,angleB=-90,arm=1.5ex] {a1}{a6} 
      \ncdiag[linewidth=.1ex,nodesep=2pt,angleA=-90,angleB=-90,arm=0.5ex] {a2}{a3} 
      \ncdiag[linewidth=.1ex,nodesep=2pt,angleA=-90,angleB=-90,arm=0.5ex] {a4}{a5} 
      \ncdiag[linewidth=.1ex,nodesep=2pt,angleA=-90,angleB=-90,arm=0.5ex] {b1}{b2} 
      \ncdiag[linewidth=.1ex,nodesep=2pt,angleA=-90,angleB=-90,arm=1.5ex] {b3}{b6} 
      \ncdiag[linewidth=.1ex,nodesep=2pt,angleA=-90,angleB=-90,arm=0.5ex] {b4}{b5}
    \end{displaymath}
    \caption{\label{fig:free-L1a}Example of $i$-blocked/free entries when $\st{f}_i$ acts by case L1(a).}
  \end{figure}

  \emph{Case L2(a)}: We assume $x=\st{i}$ and $y=i$ in $T$, therefore $\st{f}_i(T)$ has $x=i$ and $y=\st{i+1}$. For an example of this case, see Figure~\ref{fig:free-L2a}. The position of $x=\st{i}$ in $w(T)$ is precisely the position of $y=\st{i+1}$ in $w(\st{f}_i(T))$. However, the position of $y=i$ in $w(T)$ is weakly right of the position of $x=i$ in $w(\st{f}_i(T))$, with the offset equal to the total number of unmarked entries strictly right of $y$ or strictly left of $x$ in $T$ (and in $\st{f}(T)$). We claim every cell $v$ with entry $i+1$ in this range is $i$-blocked. Indeed, in this case $v$ must be right of $y$, therefore the cell immediately below $v$, say $u$, must have entry $\st{i} < u < i+1$. However, the entry cannot be $\st{i+1}$, since this forces a downward column of entries $\st{i+1}$ equal in length to the downward column of entries $\st{i}$ from $x$, therefore $x$ would be $i$-blocked. Therefore $u$ must have entry $i$, so there are at least as many entries $i$ right of $x$ as entries $i+1$ right of $y$, proving that all of the latter are indeed $i$-blocked. Thus the number of $i$-blocked pairs is unchanged in passing from $w(T)$ to $w(\st{f}_i(T))$. Furthermore, the result is that the rightmost $i$-free entry $\st{i}$ becomes the leftmost $i$-free entry $\st{i+1}$, landing us in case R2(a) for the shifted raising operator, which will precisely undo the action.
  
  \begin{figure}[ht]
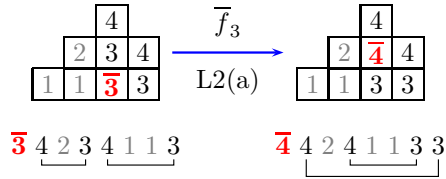

    \begin{displaymath}
      \begin{array}{c@{\hskip 3\cellsize}c}
        \rnode{E}{\tableau{ & & 4 \\ & {\gray 2} & 3 & 4 \\ {\gray 1} & {\gray 1} & \mathbf{\red\st{3}} & 3 }} &
        \rnode{F}{\tableau{ & & 4 \\ & {\gray 2} & \mathbf{\red\st{4}} & 4 \\ {\gray 1} & {\gray 1} & 3 & 3 }} \\ \\
        \mathbf{\red\st{3}} \ \rnode{e1}{4} \ {\gray 2} \ \rnode{e2}{3} \ \rnode{e3}{4} \ {\gray 1} \ {\gray 1} \ \rnode{e4}{3} &
        \mathbf{\red\st{4}} \ \rnode{f1}{4} \ {\gray 2} \ \rnode{f2}{4} \ {\gray 1} \ {\gray 1} \ \rnode{f3}{3} \ \rnode{f4}{3}
      \end{array}
      \psset{nodesep=5pt,linewidth=.2ex}
      \ncline[linecolor=blue]{->} {E}{F} \naput{\st{f}_3} \nbput{\text{L2(a)}}
      \ncdiag[linewidth=.1ex,nodesep=2pt,angleA=-90,angleB=-90,arm=0.5ex] {e1}{e2} 
      \ncdiag[linewidth=.1ex,nodesep=2pt,angleA=-90,angleB=-90,arm=0.5ex] {e3}{e4} 
      \ncdiag[linewidth=.1ex,nodesep=2pt,angleA=-90,angleB=-90,arm=1.5ex] {f1}{f4} 
      \ncdiag[linewidth=.1ex,nodesep=2pt,angleA=-90,angleB=-90,arm=0.5ex] {f2}{f3} 
    \end{displaymath}
    \caption{\label{fig:free-L2a}Example of $i$-blocked/free entries when $\st{f}_i$ acts by case L2(a).}
  \end{figure}
  
  \emph{Case L1(b)}: We have $x=i$ in $T$ become $x=i+1$ in $\st{f}_i(T)$, and by Lemma~\ref{lem:free} this happens at the rightmost $i$-free entry, so the number of $i$-blocked pairs is again unchanged and the rightmost $i$-free entry $i$ becomes the leftmost $i$-free entry $i+1$ landing in case R1(b) of the shifted raising operator, which precisely toggles back to case L1(b) under $\st{e}_i$.

  \emph{Case L2(b)}: We have $x=\st{i}$ in $T$ becomes $x=\st{i+1}$ in $\st{f}_i(T)$, and by Lemma~\ref{lem:free} this happens at the rightmost $i$-free entry, so the number of $i$-blocked pairs is unchanged and the rightmost $i$-free entry $\st{i}$ becomes the leftmost $i$-free entry $\st{i+1}$ landing in case R2(b) of the shifted raising operator, which precisely toggles back to case L2(b) under $\st{e}_i$.
  
  \emph{Case L1(c)}: We assume $x=i$, $y = \st{i+1}$ or $i+1$, and $z > \st{i+1}$ in $T$, and, letting $u$ denote the northeastern-most cell of the $(i+1)$-ribbon containing $y$, $u = \st{i+1}$. Then $\st{f}_i(T)$ has $x=\st{i+1}$ and $u = i+1$. For an example of this case, see Figure~\ref{fig:free-L1c}. The situation here is similar to case L1(a), though now both $x$ and $u$ move when passing from $w(T)$ to $w(\st{f}_i(T))$, so we consider each in turn. First, comparing the position of $x=i$ in $w(T)$ with that of $u=i+1$ in $w(\st{f}_i(T))$, $x$ moves left past any unmarked entries that lie in a row strictly between that of $u$ and $x$. We claim that any such entries $i$ are $i$-blocked. Since no two entries $\st{i+1}$ may occur in the same row, the number of entries $i+1$ along the $(i+1)$-ribbon is one fewer than the number of columns spanned by the ribbon (since the northwestern-most entry is $\st{i+1}$). Further, since $i < \st{i+1}$ with no intermediate values, any $i$ in a row between $u$ and $x$ must be immediately below the $(i+1)$-ribbon, therefore the maximum number of such entries is again the number of columns spanned minus one for the column of $x$. Since the $(i+1)$-ribbon lies above the $i$'s, all of the $i$'s will be $i$-paired with one of those $i+1$'s, thus proving the claim. Second, comparing the position of $u=\st{i+1}$ in $w(T)$ with that of $x=\st{i+1}$ in $w(\st{f}_i(T))$, $u$ moves left past any marked entries that lie weakly between $u$ and $x$ in the column reading word (bottom to top along columns, from right to left). By the same analysis of ribbons, where now we count columns instead of rows, any such entries $\st{i}$ are $i$-blocked. Therefore the number of $i$-blocked pairs is once again unchanged. Furthermore, the head of the $(i+1)$-ribbon will become the leftmost $i$-free entry $i+1$ and satisfy the conditions of case R2(c) for the shifted raising operators, whose action will precisely undo that of the shifted raising operator for this case.
  
  \begin{figure}[ht]
    \begin{displaymath}
      \begin{array}{c@{\hskip 2\cellsize}c}
        \rnode{C}{\tableau{& & \mathbf{\red 3} & \st{4} & 4 \\ & {\gray 2} & \st{3} & 3 & \st{4} \\ {\gray 1} & {\gray\st{2}} & {\gray 2} & \st{3} & \mathbf{\red 3} }} &
        \rnode{D}{\tableau{& & \mathbf{\red 3} & \mathbf{\red 4} & 4 \\ & {\gray 2} & \st{3} & 3 & \st{4} \\ {\gray 1} & {\gray\st{2}} & {\gray 2} & \st{3} & \st{4} }} \\ \\
        \rnode{c1}{\st{4}} \ \rnode{c2}{\st{3}} \ \rnode{c3}{\st{4}} \ \rnode{c4}{\st{3}} \ \mathbf{\red 3} \ \rnode{c5}{4} \ {\gray\st{2}} \ {\gray 2} \ \rnode{c6}{3} \ {\gray 1} \ {\gray 2} \ \mathbf{\red 3} &
        \rnode{d1}{\st{4}} \ \rnode{d2}{\st{4}} \ \rnode{d3}{\st{3}} \ \rnode{d4}{\st{3}} \ \mathbf{\red 3} \ \mathbf{\red 4} \ \rnode{d5}{4} \ {\gray\st{2}} \ {\gray 2} \ \rnode{d6}{3} \ {\gray 1} \ {\gray 2} 
      \end{array} 
      \psset{nodesep=5pt,linewidth=.2ex}
      \ncline[linecolor=blue]{->} {C}{D} \naput{\st{f}_3} \nbput{\text{L1(c)}}
      \ncdiag[linewidth=.1ex,nodesep=2pt,angleA=-90,angleB=-90,arm=0.5ex] {c1}{c2} 
      \ncdiag[linewidth=.1ex,nodesep=2pt,angleA=-90,angleB=-90,arm=0.5ex] {c3}{c4} 
      \ncdiag[linewidth=.1ex,nodesep=2pt,angleA=-90,angleB=-90,arm=0.5ex] {c5}{c6} 
      \ncdiag[linewidth=.1ex,nodesep=2pt,angleA=-90,angleB=-90,arm=1.5ex] {d1}{d4} 
      \ncdiag[linewidth=.1ex,nodesep=2pt,angleA=-90,angleB=-90,arm=0.5ex] {d2}{d3} 
      \ncdiag[linewidth=.1ex,nodesep=2pt,angleA=-90,angleB=-90,arm=0.5ex] {d5}{d6}
    \end{displaymath}
    \caption{\label{fig:free-L1c}Example of $i$-blocked/free entries when $\st{f}_i$ acts by case L1(c).}
  \end{figure}

  \emph{Case L2(c)}: We assume $x=\st{i}$, $y > i$, and $z = i$ or $\st{i+1}$ in $T$, and let $u=i$ denote the southwestern-most entry $i$ of the $i$-ribbon containing $x$ not followed by $i$ or $\st{i+1}$. Then $\st{f}_i(T)$ has $x=i$ and $u = \st{i+1}$. For an example of this case, see Figure~\ref{fig:free-L2c}. As with case L1(c), both $x$ and $u$ move when passing from $w(T)$ to $w(\st{f}_i(T))$. Comparing the position of $x=\st{i}$ in $w(T)$ with that of $u=\st{i+1}$ in $w(\st{f}_i(T))$, $x$ moves left past any marked entries that lie in a column strictly between that of $x$ and $u$. We claim that any such entries $\st{i}$ are already $i$-blocked in $T$. Starting from $x$, the first time the ribbon descends there must be an entry $\st{i+1}$ at the end of that row, else the last $i$ in that row would have been $u$, therefore all entries immediately right of the descending portion must be $\st{i+1}$ except for the right turn of the $i$-ribbon, which must have an $i$, and the argument repeats. Therefore each step down of the $i$-ribbon has an $\st{i}$, and in the column immediatelyto the right there are equally many $\st{i+1}$'s offset one row higher, so all entries $\st{i}$ are $i$-blocked as claimed. Comparing the position of $u=i$ in $w(T)$ with that of $x=i$ in $w(\st{f}_i(T))$, $u$ moves left past any unmarked entries that lie in a row between that of $x$ and $u$. We claim that any such entries $i+1$ are already $i$-blocked in $T$ and not by pairing with $u$. The number of $i$'s in this range is the number of columns spanned by the $i$-ribbon below the top row. Since the $(i+1)$-ribbon(s) in this range must lie immediately on top of the $i$-ribbon, the number of $i+1$'s is bounded by the number of columns spanned by the $i$-ribbon excluding the top row minus one for the entry $\st{i+1}$ that must end the top row, therefore there are strictly more $i$'s including $u$. Therefore all entries $i+1$ remain $i$-blocked in $\st{f}_i(T)$. 
  
  \begin{figure}[ht]
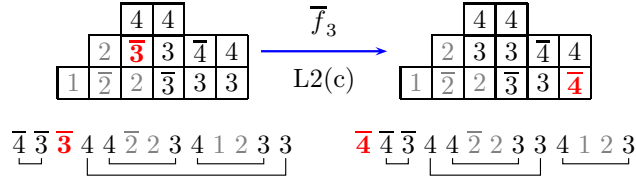

    \begin{displaymath}
      \begin{array}{c@{\hskip 2\cellsize}c}
        \rnode{G}{\tableau{ & & 4 & 4 \\ & {\gray 2} & \mathbf{\red\st{3}} & 3 & \st{4} & 4 \\ {\gray 1} & {\gray\st{2}} & {\gray 2} & \st{3} & 3 & 3 }} &
        \rnode{H}{\tableau{ & & 4 & 4 \\ & {\gray 2} & 3 & 3 & \st{4} & 4 \\ {\gray 1} & {\gray\st{2}} & {\gray 2} & \st{3} & 3 & \mathbf{\red\st{4}} }} \\ \\
        \rnode{g1}{\st{4}} \ \rnode{g2}{\st{3}} \ \mathbf{\red\st{3}} \ \rnode{g3}{4} \ \rnode{g4}{4} \ {\gray\st{2}} \ {\gray 2} \ \rnode{g5}{3} \ \rnode{g6}{4} \ {\gray 1} \ {\gray 2} \ \rnode{g7}{3} \ \rnode{g8}{3} &
        \mathbf{\red\st{4}} \ \rnode{h1}{\st{4}} \ \rnode{h2}{\st{3}} \ \rnode{h3}{4} \ \rnode{h4}{4} \ {\gray\st{2}} \ {\gray 2} \ \rnode{h5}{3} \ \rnode{h6}{3} \ \rnode{h7}{4} \ {\gray 1} \ {\gray 2} \ \rnode{h8}{3} 
      \end{array}
      \psset{nodesep=5pt,linewidth=.2ex}
      \ncline[linecolor=blue]{->} {G}{H} \naput{\st{f}_3} \nbput{\text{L2(c)}}
      \ncdiag[linewidth=.1ex,nodesep=2pt,angleA=-90,angleB=-90,arm=0.5ex] {g1}{g2} 
      \ncdiag[linewidth=.1ex,nodesep=2pt,angleA=-90,angleB=-90,arm=1.5ex] {g3}{g8} 
      \ncdiag[linewidth=.1ex,nodesep=2pt,angleA=-90,angleB=-90,arm=0.5ex] {g4}{g5} 
      \ncdiag[linewidth=.1ex,nodesep=2pt,angleA=-90,angleB=-90,arm=0.5ex] {g6}{g7} 
      \ncdiag[linewidth=.1ex,nodesep=2pt,angleA=-90,angleB=-90,arm=0.5ex] {h1}{h2} 
      \ncdiag[linewidth=.1ex,nodesep=2pt,angleA=-90,angleB=-90,arm=1.5ex] {h3}{h6} 
      \ncdiag[linewidth=.1ex,nodesep=2pt,angleA=-90,angleB=-90,arm=0.5ex] {h4}{h5} 
      \ncdiag[linewidth=.1ex,nodesep=2pt,angleA=-90,angleB=-90,arm=0.5ex] {h7}{h8} 
    \end{displaymath}
    \caption{\label{fig:free-L2c}Example of $i$-blocked/free entries when $\st{f}_i$ acts by case L2(c).}
  \end{figure}
  \emph{Case L1(d)}: We assume $x=i$, $y = \st{i+1}$ or $i+1$, and $z > \st{i+1}$ in $T$, and, letting $u$ denote the northeastern-most cell of the $(i+1)$-ribbon containing $y$. Note that $u = i+1$ lies on the main diagonal, since each column between $u$ and $x$ contains at least one $i+1$, and as all are $i$-paired with entries to the left of $x$ in the reading word, the column strictly to the left of $u$ must contain an $i$. Then $\st{f}_i(T)$ has $x=\st{i+1}$. Since $u$ lies on the main diagonal, it occurs in $w(T)$ and in $w(\st{f}_i(T))$ after all marked entries $\st{i}, \st{i+1}$ and before all unmarked entries $i,i+1$ on the two relevant ribbons, therefore the analysis of case L1(c) resolves this case as well, with inverse given by case R1(d).
\end{proof}

From the proof of Theorem~\ref{thm:crystal}, we have the following characterization of the lengths of the heads and tails of $i$-strings in terms of the number of $i$-free entries, and this is equivalent to applying \eqref{e:lmax} and \eqref{e:max} to the hook reading word.

\begin{corollary}
  For $T$ a semistandard shifted tableau and $i \geq 1$ an index, $\upe_i(T)$ is the number of $i$-free entries $\st{i+1},i+1$ in $T$, and $\downf_i(T)$ is the number of $i$-free entries $\st{i},i$ of $T$, and we have
  \begin{eqnarray}
    \upe_i(T) & = & \max\{ r \mid \wt(w_{r} w_{r+1} \cdots w_{n})_{i+1} - \wt(w_{r} w_{r+1} \cdots w_{n})_{i} \},
    \label{e:shifted-head} \\
    \downf_i(T) & = & \max\{ r \mid \wt(w_{1} w_{2} \cdots w_{r})_{i} - \wt(w_{1} w_{2} \cdots w_{r})_{i+1} \}.
    \label{e:shifted-tail}
  \end{eqnarray}
  \label{cor:string-lengths}
\end{corollary}

We use Stembridge's characterization of regular graphs \cite{Ste03} to establish directly that the shifted lowering operators determine a normal crystal on shifted tableaux. To begin with, we show that the notation used in Definition~\ref{def:regular} is well-defined for the graph on semistandard shifted tableaux with edges given by the shifted crystal operators.

\begin{lemma}
  The graph on $\SSHT_n(\gamma)$ with edges $x {\blue \ial} y$ whenever $\st{e}_i(x) = y$ and $x {\blue \iar} z$ whenever $\st{f}_i(x) = z$ is a directed, colored graph satisfying Stembridge axioms A1 and A2.
  \label{lem:ax12}
\end{lemma}

\begin{proof}
  For axiom A1, all monochromatic directed paths have finite length since $\st{f}_i$ changes the weight of a semistandard shifted tableau by $\wt(\st{f}_i(T)) = \wt(T) + \alpha_i$, where $\alpha_i$ is the simple root $\e_i - \e_{i+1}$, and the weight of a semistandard shifted tableau has non-negative parts. Axiom A2, stating for every vertex $x$, there is at most one edge $x {\blue \ial} y$ and at most one edge $x {\blue \iar} z$, follows from $\st{f}_i$ being well-defined, proved in Theorem~\ref{thm:shifted-well}, and from $\st{f}_i$ and $\st{e}_i$ being inverses to one another, proved in Theorem~\ref{thm:crystal}. 
\end{proof}

Given the local nature of the shifted operators, meaning that $\st{f}_i$ looks only at the positions of entries $\st{i}, i, \st{i+1}, i+1$, axioms A3--A6 for $|i-j| \geq 2$ are easy to establish. Corollary~\ref{cor:string-lengths} helps to establish axioms A3 and A4 in the case $j=i$ (and axioms A5 and A6 are vacuous in this case). For the remaining cases, $j = i \pm 1$, axioms A3--A6 are resolved with the help of the following lemma.

\begin{lemma}
  Let $T$ be a semistandard shifted tableau. Then for $i>1$, $\st{f}_i(T)$ either has one more $(i-1)$-free entry $\st{i-1},i-1$ or one fewer $(i-1)$-free entry $\st{i},i$ and not both, and for $i\geq 1$, $\st{f}_i(T)$ either has one more $(i+1)$-free entry $\st{i+1},i+1$ or one fewer $(i+1)$-free entry $\st{i+2},i+2$ and not both.
  \label{lem:pm1}
\end{lemma}

\begin{proof}
  The net effect of $\st{f}_i$ on the weight is to remove an entry $\st{i}$ or $i$ and create an entry $\st{i+1}$ or $i+1$. If the removed entry $\st{i}, i$ was $i-1$-blocked, then some entry  $\st{i-1}, i-1$ that was $i-1$-blocked becomes $i-1$-free; otherwise the removed entry $\st{i}, i$ was $i-1$-free. Similarly, If the created entry $\st{i+1}, i+1$ pairs with an otherwise $i+1$-free entry $\st{i+2},i+2$, then both entries become $i+1$-blocked; otherwise, the created entry $\st{i+1}, i+1$ is $i+1$-free. Therefore we need only show that the traveling entries in the cases of L1(a),(c),(d) and L2(a),(c) of Definition~\ref{def:shifted-lower} do not otherwise change the number of $i\pm 1$-free entries. This is a case by case analysis similar to that in the proof of Theorem~\ref{thm:crystal}.
\end{proof}

\begin{theorem}
  The shifted raising and lowering operators $\st{e}_i, \st{f}_i$ for $i=1,2,\ldots,r$ define a normal crystal on $\SSHT_{r+1}(\gamma)$.
  \label{thm:shifted-crystal}
\end{theorem}

\begin{proof}
  We show directly that the graph on $\SSHT_n(\gamma)$ with edges given by the shifted crystal operators is regular (Definition~\ref{def:regular}). Lemma~\ref{lem:ax12} proves axioms A1 and A2.

  For axioms A3 and A4, we have three cases based on $j$. For $j=i$, axiom A4 is vacuous, and the characterization of $\upe_i$ and $\downf_i$ in Corollary~\ref{cor:string-lengths} combined with Lemma~\ref{lem:free} shows that $\upe_i(\st{e}_i(T)) = \upe_i(T)-1$ and $\downf_i(\st{e}_i(T)) = \downf_i(T)+1$, which implies
  \[ \Delta_i \upe_i(T) = \upe_i(T) - \upe_i(\st{e}_i(T)) = 1 = \downf_i(\st{e}_i(T)) - \downf_i(T) = \Delta_i \downf_i(T), \]
  thereby proving axiom A3 for $j=i$. For $|i-j| \geq 2$, $\upe_j(\st{e}_i(T)) = \upe_j(T)$ and $\downf_j(\st{e}_i(T)) = \downf_j(T)$, therefore $\Delta_i \upe_j(T) = 0 = \Delta_i \downf_j(T)$, establishing axioms A3 and A4. For $|i-j| = 1$, Lemma~\ref{lem:pm1} ensures that either $\upe_i(\st{e}_i(T)) = \upe_i(T) + 1$ or $\downf_i(\st{e}_i(T)) = \downf_i(T) - 1$ but not both, again establishing axioms A3 and A4.

  For axioms A5 and A6, we have the same three cases based on $j$. For $i=j$, both axioms are vacuous. For $|i-j| \geq 2$, axiom A6 is vacuous. For axiom A5, note that $|i-j| \geq 2$ implies $\{i,i+1\} \cap \{j,j+1\} = \varnothing$, therefore $\st{e}_i,\st{f}_i$ do not alter the relative positions of $j,j+1$, and neither do $\st{e}_j,\st{f}_j$ alter the relative positions of $i,i+1$. Therefore $i$-operators and $j$-operators commute as required. Finally, consider the case $|i-j| = 1$.

  For axiom A5, note that by Lemma~\ref{lem:pm1}, we have
  \[ \nabla_i \downf_j(T) = 0 \ \Rightarrow \ \downf_j(\st{f}_i(T)) = \downf_j(T) \ \Rightarrow \ \left\{
  \begin{array}{rl}
    \st{f}_i \mbox{ removes a $j$-free entry $i$} & \text{if} \ j=i-1 , \\
    \st{f}_i \mbox{ creates a $j$-blocked entry $i+1$} & \text{if} \ j=i+1 .
  \end{array}\right. \]
  So $\st{f}_i$ removes an $(i-1)$-free entry $i$ therefore the rightmost $(i-1)$-free entry $i-1$ is the same in $T$ and in $\st{f}_i(T)$. Similarly, $\st{f}_{i-1}$ will create an $i$-blocked entry $i$, ensuring that the leftmost $i$-free entry $i$ is the same for $T$ and $\st{f}_{i-1}(T)$. Combining these, we see that $\st{f}_{i-1} \st{f}_{i} (T) = \st{f}_{i} (T) \st{f}_{i-1} (T)$ as desired. The case $j=i+1$ is identical, and the analysis for $\Delta_i \upe_j(T) = 0$ is analogous.
  
  Finally, for axiom A6, by Lemma~\ref{lem:pm1} we have
  \[ \nabla_i \downf_j(T) = -1 \ \Rightarrow \ \downf_j(\st{f}_i(T)) = \downf_j(T)+1 \ \Rightarrow \ \left\{
  \begin{array}{rl}
    \st{f}_i \mbox{ removes a $j$-blocked entry $i$} & \text{if} \ j=i-1 , \\
    \st{f}_i \mbox{ creates a $j$-free entry $i+1$} & \text{if} \ j=i+1 .
  \end{array}\right. \]
  Therefore $\st{f}_i$ removes an $(i-1)$-blocked entry $i$, allowing another $(i-1)$-free entry $i-1$ to manifest. Applying $\st{f}_{i-1}^2$ changes the newly created $(i-1)$-free entry $i-1$ of $\st{f}_i(T)$ and the original rightmost $(i-1)$-free entry $i-1$ of $T$ to $i$ so that a final application of $\st{f}_i$ changes the latter to $i+1$. In the other direction, $\st{f}_{i-1}$ creates an $i$-free entry $i$. Applying $\st{f}_{i}^2$ changes the newly created $i$-free entry $i$ of $\st{f}_{i-1}(T)$ and the original rightmost $i$-free entry $i$ of $T$ to $i+1$ so that a final application of $\st{f}_{i-1}$ yields the same result, and we have $\st{f}_{i} \st{f}_{i-1}^2 \st{f}_{i} (T) = \st{f}_{i-1} \st{f}_{i}^2 (T) \st{f}_{i-1} (T)$ as desired. The case $j=i+1$ is identical, and the analysis for $\Delta_i \upe_j(T) = -1$ is analogous.
\end{proof}

Using Theorem~\ref{thm:shifted-crystal}, we give a new proof of Theorem~\ref{thm:P_Schur}. Note that this characterization of the Schur coefficients of a Schur $P$-polynomial is more explicit than Sagan's and Worley's shifted insertion rule \cite{Sag87,Wor84} or Assaf's dual equivalence characterization \cite{Ass18}.

\begin{definition}
  For $\gamma$ a strict partition, the set of \emph{Yamanouchi shifted tableaux of shape $\gamma$}, denoted by $\Yam(\gamma)$, is the set of semistandard shifted tableaux $T$ of shape $\gamma$ such that $\upe_i(T)=0$ for all $i$.
  \label{def:yamanouchi}
\end{definition}

For example, the Yamanouchi shifted tableaux of shape $(4,3,1)$ are shown in Figure~\ref{fig:highest-weights}. The Yamanouchi shifted tableaux are precisely the highest weights of the normal crystal $(\SSHT_n(\gamma),\{\st{e}_i,\st{f}_i\}_{1 \leq i < n},\wt)$. By Corollary~\ref{cor:highest-weights}, this gives an explicit characterization of the Schur expansion of a Schur $P$-polynomial.

\begin{figure}[ht]
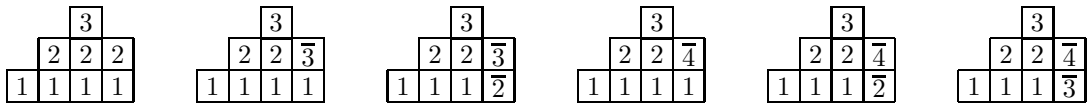

  \begin{displaymath}
    \begin{array}{c@{\hskip 2\cellsize}c@{\hskip 2\cellsize}c@{\hskip 2\cellsize}c@{\hskip 2\cellsize}c@{\hskip 2\cellsize}c}
      \tableau{ & & 3\\ & 2 & 2 & 2 \\ 1& 1 & 1 & 1} & 
      \tableau{ & & 3\\ & 2 & 2 & \st{3} \\ 1& 1 & 1 & 1} & 
      \tableau{ & & 3\\ & 2 & 2 & \st{3} \\ 1& 1 & 1 & \st{2}} &      
      \tableau{ & & 3\\ & 2 & 2 & \st{4} \\ 1& 1 & 1 & 1} & 
      \tableau{ & & 3\\ & 2 & 2 & \st{4} \\ 1& 1 & 1 & \st{2}} & 
      \tableau{ & & 3\\ & 2 & 2 & \st{4} \\ 1& 1 & 1 & \st{3}}       
    \end{array}
  \end{displaymath}
  \caption{\label{fig:highest-weights}The Yamanouchi shifted tableaux of strict shape $(4,3,1)$.}
\end{figure}

\begin{corollary}
  For $\gamma$ a strict partition, we have
  \begin{equation}
    P_{\gamma}(x_1,\ldots,x_{n}) = \sum_{T \in \Yam(\gamma)} s_{\wt(T)}(x_1,\ldots,x_{n}).
  \end{equation}
  In particular, the Schur $P$-polynomial is Schur positive with coefficients $g_{\gamma,\lambda} = \#\{T \in \Yam(\gamma) \mid \wt(T)=\lambda \}$.
  \label{cor:P2schur}
\end{corollary}

For example, from Figure~\ref{fig:highest-weights} we compute
\begin{eqnarray*}
 P_{(4,3,1)}(x_1,x_2,x_3,x_4) = s_{(4,3,1)}(x_1,x_2,x_3,x_4) +s_{(4,2,2)}(x_1,x_2,x_3,x_4) +s_{(3,3,2)}(x_1,x_2,x_3,x_4)\\+s_{(4,2,1,1)}(x_1,x_2,x_3,x_4) +s_{(3,3,1,1)}(x_1,x_2,x_3,x_4) +s_{(3,2,2,1)}(x_1,x_2,x_3,x_4).
\end{eqnarray*}

In order to ease computations of the Schur expansion of a Schur $P$-polynomial using Corollary~\ref{cor:P2schur}, we have the following necessary condition for a semistandard shifted tableau to be Yamanouchi.

\begin{lemma}\label{lem:Yamanouchi}
  If $T$ is a Yamanouchi shifted tableaux, then any unmarked entry on row $i$ is equal to $i$. In particular the leftmost box in row $i$ is labeled $i$.
\end{lemma}

\begin{proof}
  This is true for the first row, as the unmarked cells on the first row come last in the reading word and any $i+1>1$ can not be $i$-paired, violating the Yamanouchi condition. Now assume row $k$ has a box labeled $k+i$, $i>0$ and no row below has an unmarked entry larger than the row index. As there is no cell labeled $k+i-1$ on a row strictly below, this $k+i$ needs to be $(k+i-1)$-paired with  $\st{k+i-1}$, which can only happen if the leftmost box of the row below is labeled $k-1+i$, violating our assumption.
\end{proof}

To demonstrate the utility of our formula and the description in Lemma~\ref{lem:Yamanouchi}, for $k>1$ and integer, consider $\delta_k = (k-1,k-2,\ldots,1)$, the \emph{staircase partition}, which is, in particular, strict. Then we have the following result for the coincidence of Schur polynomials and Schur $P$-polynomials for staircase shapes.

\begin{corollary}
  For $k>1$, we have $P_{\delta_k}(x_1,\ldots,x_n) = s_{\delta_k}(x_1,\ldots,x_n)$. Moreover, if $\gamma$ is a strict partition such that $\gamma \neq \delta_k$ for any $k$, then $P_{\delta_k}(x_1,\ldots,x_n)$ has more than one term in its Schur expansion.
  \label{cor:staircase}
\end{corollary}

\begin{proof}
 Let $T \in \Yam(\delta_n)$. Note that by Lemma \ref{lem:Yamanouchi} the highest row of $T$ contains $n$, so all the entries on $T$ are bounded by $n$. Also, as the leftmost entry is equal to the row index, any marked number $k$ is on a row of index less than or equal to $k$. So if $T$ contains a marked entry, $T$ has a row $i$ such that row $i$ contains an entry greater than $i$, and row $i+1$ only contains $i+1$.
By column rules, this row has $i$s except for the rightmost cell which contains $\st{i+1}$, as seen in Figure \ref{fig:staircaselemmaproof}, left. It needs to be $i$-paired with some $\st{i}$, which can only be below the rightmost $i$ on row $i$. Consider the bottom-most $\st{i}$ on that column. The cell right adjacent to it can not be equal to $i$ by Lemma  \ref{lem:Yamanouchi}. It can not be $\st{i+1}$ either, as there is no way to $i$-pair that $\st{i+1}$. There are no other options by row and column rules, so such a tableaux does not exist.

Now assume $\gamma$ is not a staircase. Then there exists some $i$ such that $\gamma_i \geq \gamma_{i+1}+2$. The shifted tableaux that contains only $k$s on each row $k$ except for the rightmost cell of row $i$ which is labeled $\st{i+1}$ is a Yamanouchi shifted tableaux(Figure \ref{fig:staircaselemmaproof}, right), so there are at least two elements of $\Yam(\gamma)$.
\end{proof}

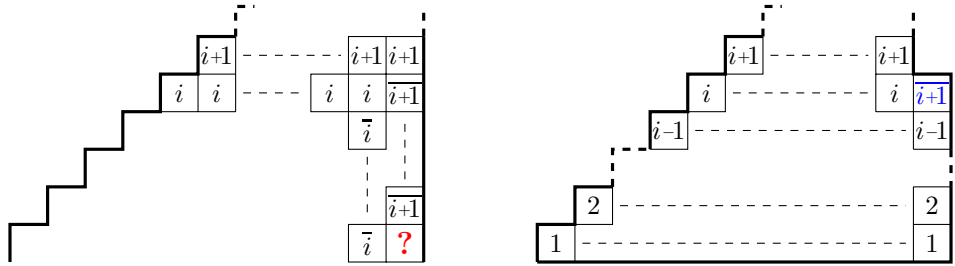
\begin{figure}[ht]
    \centering
    
	\begin{tikzpicture}[scale=.5]
	\draw (0,1) grid (2,2);
    \draw[line width=1.2 pt] (-4,-3)--(-4,-2)--(-3,-2)--(-3,-1)--(-2,-1)--(-2,0)--(-1,0)--(-1,1)--(0,1)--(0,2)--(1,2)--(1,3)--(2,3)(7,3)--(7,-3);
    %\draw[line width=1.2 pt] (5,4.8)--(6,4.8)--(6,5.6)--(7,5.6)--(7,-3);
\draw[line width=1.2 pt, dashed] (7,3)--(7,3.8) (2,3)--(2,3.8)--(2.5,3.8);
    \draw (2,3)--(2,2);
    \draw[dashed] (5.5,-.2)--(5.5,-1.8) (6.5,-.8)--(6.5,.8) (2.2,1.5)--(3.8,1.5) (2.2,2.5)--(4.8,2.5);
	\node[red] at (6.5,-2.5) {\Large \textbf{?}};
    \node at (1.5,1.5) {$i$};
    \node at (.5,1.5) {$i$};
    \node at (1.5, 2.5) {$i\texplus 1$};
    \draw (5,1) grid (7,3);
    \draw (5,1)--(4,1)--(4,2)--(5,2);
    \node at (4.5,1.5) {$i$};
    \node at (5.5,1.5) {$i$};
    \node at (6.5,1.5) {$\st{i\texplus1}$};
    \draw (5,1)--(5,0)--(6,0)--(6,1);
    \draw (5,-2)--(5,-3)--(6,-3)--(6,-2)--(5,-2) (6,-3)--(7,-3)--(7,-1)--(6,-1)--(6,-2) (7,-2)--(6,-2);
    \node at (5.5,0.5) {$\st{i}$};
    \node at (5.5,-2.5) {$\st{i}$};
    \node at (6.5,-1.5) {$\st{i\texplus1}$};
    \node at (5.5,2.5) {$i\texplus 1$};
    \node at (6.5,2.5) {$i\texplus1$};
	\end{tikzpicture}\qquad \qquad
		\begin{tikzpicture}[scale=.5]
    \draw[line width=1.2 pt] (-4,-3)--(-4,-2)--(-3,-2)--(-3,-1)--(-2,-1) (-1,0)--(-1,1)--(0,1)--(0,2)--(1,2)--(1,3)--(2,3) (6,3)--(6,2)--(7,2)--(7,0) (7,-1)--(7,-3)--(-4,-3);
\draw[line width=1.2 pt, dashed] (6,3)--(6,3.8) (2,3)--(2,3.8)--(2.5,3.8) (7,0)--(7,-1) (-2,-1)--(-2,0)--(-1,0);
    \draw  (7,1)--(7,0)--(6,0);
    \draw[dashed](1.2,1.5)--(4.8,1.5) (2.2,2.5)--(4.8,2.5) (0.2,.5)--(5.8,.5);
    \node at (.5,1.5) {$i$};
    \node at (1.5, 2.5) {$i\texplus 1$};
    \draw (5,1) grid (7,2) (5,2)--(5,3)--(6,3)--(6,2);
    %\draw (5,1)--(4,1)--(4,2)--(5,2);
    \node at (5.5,1.5) {$i$};
    \node[blue] at (6.5,1.5) {$\st{i\texplus1}$};
    \draw (6,0)--(6,1);
    \node at (6.5,0.5) {$i\texminus1$};
    %\node at (5.5,0.5) {$i\texminus1$};
    \node at (5.5,2.5) {$i\texplus 1$};
    \node at (-3.5,-2.5) {1};
    \node at (-2.5,-1.5) {2};
     \node at (-0.5,.5) {$i\texminus 1$};
     \draw (-1,0)--(0,0)--(0,1)--(1,1)--(1,2)--(2,2)--(2,3);
    \draw (-4,-3)--(-3,-3)--(-3,-2)--(-2,-2)--(-2,-1);
     \draw[dashed](-2.8,-2.5)--(5.8,-2.5) (-1.8,-1.5)--(5.8,-1.5);
     \draw (6,-3) grid (7,-1);
   \node at (6.5,-2.5) {1};
    \node at (6.5,-1.5) {2};

	\end{tikzpicture}
    \caption{There is no shifted Yamanouchi tableaux of staircase shape that contains a marked entry(left), but such a tableaux can be found for any other shape(right). \label{fig:staircaselemmaproof}}
\end{figure}

Finally, we conclude this section with the following proof that our shifted crystal operators, while seemingly different, in fact coincide with the crystal operators defined recently by Hawkes, Paramonov, and Schilling \cite{HPS17}. Hawkes, Paramonov, and Schilling \cite{HPS17} define their operators, which we hereafter term the HPS operators, directly in terms of \emph{shifted insertion}. Sagan defined the shifted insertion algorithm \cite{Sag87}, also developed independently by Worley \cite{Wor84}, to generalize the Robinson-Schensted insertion algorithm developed by Schensted \cite{Sch61} based on work of Robinson \cite{Rob38} and later generalized by Knuth \cite{Knu70}. Haiman \cite{Hai89} generalized shifted insertion to \emph{mixed insertion} to prove a conjecture of Shor that rectification commutes with shifted insertion, thus resolving a question of Sagan \cite{Sag87}. Hawkes, Paramonov, and Schilling use Haiman's mixed insertion to define crystal operators on shifted tableaux by letting the raising and lowering operators act on the recording tableau under the mixed insertion correspondence. Rather than recall details of these algorithms, we refer the interested reader to the papers \cite{Sag87, Hai89} for details on shifted and mixed insertion, and to \cite{HPS17} for the explicit definition of the HPS operators.

\begin{proposition}
  The shifted crystal operators agree with the HPS operators.
  \label{prop:HPS}
\end{proposition}

\begin{proof}
  Hawkes, Paramonov, and Schilling \cite{HPS17} define their operators, which we hereafter term the HPS operators, directly in terms of shifted insertion. In the discussion to follow, we alter the presentation in \cite{HPS17} only in switching their notation from English to French to coincide with ours.

  The HPS reading word \cite{HPS17}(p.13) is different from our hook reading word (Definition~\ref{def:hook-word}). The HPS reading word of a semistandard shifted tableau first reads all marked entries up columns from right to left, then reads all unmarked entries right to left along rows from top to bottom. They then use a bracketing rule equivalent to Definition~\ref{def:block} and select the rightmost $i$-free entry $\st{i}$ or $i$ on which to act. While our reading words differ, the choice of entry of the tableau on which to act is the same since any primed entry $\st{i}$ or $\st{i+1}$ must occur in the hook reading word before any unmarked entry $i$ or $i+1$.

  The HPS crystal operators act by first transposing the shape and promoting entries one step along the total order $\st{1} < 1 < \st{2} < 2 < \cdots$ if the selected entry on which to act is primed, then transposing and demoting after the action. With this caveat in mind, there are four cases, numbered 1, 2(a), 2(b), 2(c), for the HPS operators. We provide a dictionary between their cases and ours in Definition~\ref{def:shifted-lower} and provide details only in the one nontrivial case. The correspondence is: 
  \[ \begin{array}{lcl}
    \text{HPS} \ 1 & \leftrightarrow & L1(a)/L2(a) \\
    \text{HPS} \ 2(a) & \leftrightarrow & L1(b)/L2(b) \\
    \text{HPS} \ 2(b) & \leftrightarrow & L1(d) \\
    \text{HPS} \ 2(c) & \leftrightarrow & L1(c)/L2(c)
  \end{array}\]
  where we match with L1 when an unmarked entry is selected and with L2 when a primed entry is selected (note that HPS 2(b) is vacuous in this case). The cases are direct translations of one another with the exception of case HPS 2(c) when $\st{i}$ is selected. For this case, HPS operators first transpose and promote entries, then follow the $\st{i+1},i+1$-ribbon. Back in the original tableau, this does not correspond to the $\st{i},i$-ribbon that we following in case L2(c) of Definition~\ref{def:shifted-lower}. However, the terminal points for those ribbons, different in the two cases, in fact correspond. Therefore the operators agree in their resulting actions.
\end{proof}

%%%%%%%%%%%%%%%%%%%%%%%%%%%%%%%%%%%%%%%%%%%%%%%%%%%%%%%%%%%%%%%%
%
\section{Crystals for the quantum queer Lie superalgebra}
%
%%%%%%%%%%%%%%%%%%%%%%%%%%%%%%%%%%%%%%%%%%%%%%%%%%%%%%%%%%%%%%%%
\label{sec:graphs-Q}

Recently, Grantcharov, Jung, Kang, Kashiwara, and Kim \cite{GJKKK14} developed crystal bases for the quantum queer superalgebra. In this section, we review the queer crystal theory arising from $U_q(\mathfrak{q}(n))$ from the combinatorial viewpoint. In \S\ref{sec:crystal-Q}, we review queer crystal bases and define normal queer crystals as those arising from tensor products of the standard queer crystal. In \S\ref{sec:tableaux-Q}, we augment our crystal operators on semistandard shifted tableaux with an additional operator that results in a connected, normal queer crystal. In \S\ref{sec:local-Q}, we formulate an alternative local axioms for normal queer crystals, analogous to that of Stembridge \cite{Ste03}, toward a means to prove that a given queer crystal structure is normal.

%%%%%%%%%%%%%%%%%%%%%%%%%%%%%%%%%%%%%%%%%%%%%%%%%%%%%%%%%%%%%%%%
\subsection{Queer crystals}
%%%%%%%%%%%%%%%%%%%%%%%%%%%%%%%%%%%%%%%%%%%%%%%%%%%%%%%%%%%%%%%%
\label{sec:crystal-Q}

Using notation and terminology from \S~\ref{sec:crystal-A}, the \emph{dominant weights} $\Gamma^{+} \subset \Lambda$ are those $\lambda \in \Lambda$ such that $\lambda_1 \geq \lambda_2 \geq \cdots \geq \lambda_{r+1} \geq 0$ and $\lambda_i = \lambda_{i+1}$ implies $\lambda_i = \cdots = \lambda_{r+1} = 0$. In other words, $\Lambda^{+}$ is to partitions as $\Gamma^{+}$ is to \emph{strict} partitions. We have the following combinatorial definition for queer crystals, augmenting Definition~\ref{def:base-A}.

\begin{definition}
  A \emph{queer crystal} of dimension $r+1$ is crystal of dimension $r+1$ together with additional \emph{queer crystal operators} $e_0, f_0  :  \B \rightarrow \B \cup \{0\}$ satisfying the conditions
  \begin{enumerate}
  \item for $b,b^{\prime}\in\B$, $e_0(b)=b^{\prime}$ if and only if $f_0(b^{\prime}) = b$, and in this case we have $\wt(b^{\prime}) = \wt(b) + \alpha_1$;
  \item for $i=3,4,\ldots,r$, the operators $e_0$ and $f_0$ commute with $e_i$ and $f_i$, and if $e_0(b)\neq 0$ then $\downf_i(e_0 (b)) = \downf_i(b)$ and $\upe_i(e_0(b)) = \upe_i(b)$.    
  \end{enumerate}
  \label{def:base-Q}
\end{definition}

For example, the \emph{standard queer crystal} $\Q(n)$, for $n \in \mathbb{Z}_{>0}$, is the standard crystal $\B(n)$ together with queer crystal operator $f_0$ that acts on $\raisebox{-0.3\cellsize}{$\tableau{i}$}$ by incrementing the entry if $i=1$ or $0$ otherwise. The standard queer crystal is represented diagrammatically in Figure~\ref{fig:queer} by its queer crystal graph.

\begin{figure}[ht]
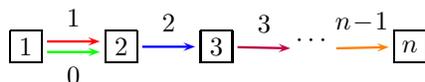

  \begin{displaymath}
    \begin{array}{c@{\hskip 2\cellsize}c@{\hskip 2\cellsize}c@{\hskip 2\cellsize}c@{\hskip 2\cellsize}c}
      \rnode{a1}{\tableau{1}} & \rnode{a2}{\tableau{2}} & \rnode{a3}{\tableau{3}} & \rnode{c}{\raisebox{0.5\cellsize}{$\cdots$}} & \rnode{a4}{\tableau{n}} 
    \end{array}
    \psset{nodesep=2pt,linewidth=.2ex}
    \ncline[offset=-2pt,linecolor=green]{->}  {a1}{a2} \nbput{0}
    \ncline[offset=2pt,linecolor=red]{->} {a1}{a2} \naput{1}
    \ncline[linecolor=blue]{->}  {a2}{a3} \naput{2}
    \ncline[linecolor=purple]{->}  {a3}{c} \naput{3}
    \ncline[linecolor=orange]{->}  {c}{a4} \naput{n\!-\!1}
  \end{displaymath}
  \caption{\label{fig:queer}The standard queer crystal.}
\end{figure}

The notion of highest weight elements is still vital to classifying queer crystals, though now the concept is not as straightforward. Given a queer crystal $\Q$ of dimension $r+1$, we define automorphisms $S_i$, for $i = 1,2,\ldots,r$ by
\[ S_i = \left\{ \begin{array}{rl}
  f_i^{\wt(b)_i - \wt(b)_{i+1}} (b) & \text{if } \wt(b)_{i} \geq \wt(b)_{i+1}, \\
  e_i^{\wt(b)_{i+1} - \wt(b)_i} (b) & \text{if } \wt(b)_{i+1} \geq \wt(b)_{i} . 
\end{array} \right. \]
Kashiwara \cite{Kas91} showed these operators satisfy the braid relations for the symmetric group, therefore to any permutation $w$ we may define $S_w$ by $S_{i_1} S_{i_2} \cdots S_{i_k}$ whenever $w = s_{i_1} s_{i_2} \cdots s_{i_k}$ is a reduced expression for $w$. For $i=1,2,\ldots,r$, define the \emph{odd crystal operators} $e_{\st{i}}, f_{\st{i}}$ by
\begin{equation} e_{\st{1}}=e_0 \hspace{3em} f_{\st{1}}=f_0 \hspace{3em}
  e_{\st{i}} = S_{w_i^{-1}} e_{0} S_{w_i} \hspace{3em} f_{\st{i}} = S_{w_i^{-1}} f_{0} S_{w_i} \text{ for } i>1
\end{equation}
where $w_i = s_2 \cdots s_i s_1 \cdots s_{i-1}$ is the shortest (in coxeter length) permutation such that $w_i \cdot \alpha_i = \alpha_1$.

\begin{definition}
  An element $b \in \Q$ of a queer crystal is a \emph{highest weight element} if $e_i(b) = 0 = e_{\st{i}}(b)$ for all $i=1,2,\ldots,r$.
  \label{def:hw-Q}
\end{definition}

Again, one of the motivating goals of crystal theory is to use these combinatorial objects to study tensor representations of the queer superalgebra, therefore again we must restrict our attention to \emph{normal queer crystals}. Connected normal queer crystals are in one-to-one correspondence with dominant weights $\Gamma^{+}$, which in turn index irreducible representations for the queer superalgebra. Given a dominant weight $\gamma \in \Gamma^{+}$, let $\Q(\gamma)$ denote the connected normal crystal with highest weight $\gamma$. Then $\mathrm{ch}(\Q(\gamma))$ is precisely the character of the irreducible representation indexed by $\gamma$, which corresponds to the Schur-$P$ polynomial $P_{\gamma}(x_1,\ldots,x_{r+1})$. As in the classical case, we have the remarkable fact that the following combinatorial procedure on queer crystals corresponds to the tensor product of the corresponding representations.

\begin{definition}
  Given two queer crystals $\Q_1$ and $\Q_2$, the \emph{tensor product} $\Q_1 \otimes \Q_2$ is the set $\Q_1 \otimes \Q_2$ together with crystal operators $e_i, f_i$ are defined on the tensor product $\Q_1 \otimes \Q_2$ by Definition~\ref{def:tensor-A} for $i=1,2,\ldots,r$, and for $i=0$ we have the additional rule
\begin{equation}
  f_0(b_1 \otimes b_2) = \left\{ \begin{array}{rl}
    f_0(b_1) \otimes b_2 & \mbox{if } \wt(b_2)_1 = \wt(b_2)_2 = 0, \\
    b_1 \otimes f_0(b_2) & \mbox{otherwise}.
  \end{array} \right.
\end{equation}
\label{def:tensor-Q}
\end{definition}

For example, Figure~\ref{fig:tensor-Q} computes the tensor product of two copies of the standard queer crystal $\Q(3)$.

\begin{figure}[ht]
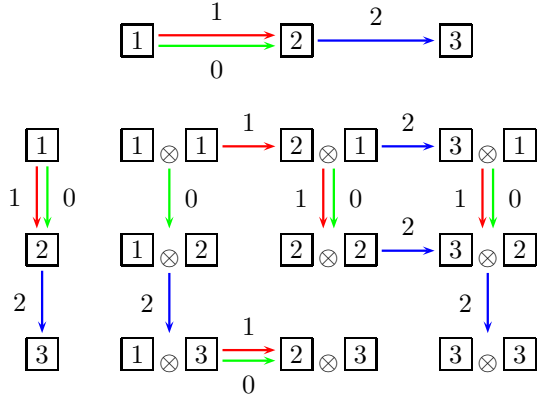

  \begin{displaymath}
    \begin{array}{c@{\hskip 2\cellsize}l@{\hskip 2\cellsize}l@{\hskip 2\cellsize}l}
                              & \rnode{a1}{\tableau{1}}               & \rnode{a2}{\tableau{2}}               & \rnode{a3}{\tableau{3}} \\[2\cellsize]
      \rnode{b1}{\tableau{1}} & \rnode{c11}{\tableau{1}\otimes\tableau{1}} & \rnode{c21}{\tableau{2}\otimes\tableau{1}} & \rnode{c31}{\tableau{3}\otimes\tableau{1}} \\[2\cellsize]
      \rnode{b2}{\tableau{2}} & \rnode{c12}{\tableau{1}\otimes\tableau{2}} & \rnode{c22}{\tableau{2}\otimes\tableau{2}} & \rnode{c32}{\tableau{3}\otimes\tableau{2}} \\[2\cellsize]
      \rnode{b3}{\tableau{3}} & \rnode{c13}{\tableau{1}\otimes\tableau{3}} & \rnode{c23}{\tableau{2}\otimes\tableau{3}} & \rnode{c33}{\tableau{3}\otimes\tableau{3}}
    \end{array}
    \psset{nodesep=2pt,linewidth=.2ex}
    \ncline[offset=-2pt,linewidth=.2ex,linecolor=green]{->}  {a1}{a2} \nbput{0}
    \ncline[offset=2pt,linewidth=.2ex,linecolor=red]{->} {a1}{a2} \naput{1}
    \ncline[linecolor=blue]{->}  {a2}{a3} \naput{2}
    \ncline[linewidth=.2ex,linecolor=red]{->} {c11}{c21} \naput{1}
    \ncline[linecolor=blue]{->}  {c21}{c31} \naput{2}
    \ncline[offset=2pt,linecolor=green]{->}  {b1}{b2} \naput{0}
    \ncline[offset=-2pt,linewidth=.2ex,linecolor=red]{->} {b1}{b2} \nbput{1}
    \ncline[linewidth=.2ex,linecolor=green]{->} {c11}{c12} \naput{0}
    \ncline[offset=2pt,linewidth=.2ex,linecolor=green]{->} {c21}{c22} \naput{0}
    \ncline[offset=-2pt,linewidth=.2ex,linecolor=red]{->} {c21}{c22} \nbput{1}
    \ncline[offset=2pt,linewidth=.2ex,linecolor=green]{->} {c31}{c32} \naput{0}
    \ncline[offset=-2pt,linewidth=.2ex,linecolor=red]{->} {c31}{c32} \nbput{1}
    \ncline[linecolor=blue]{->}  {c22}{c32} \naput{2}
    \ncline[linecolor=blue]{->}  {b2}{b3} \nbput{2}
    \ncline[linecolor=blue]{->}  {c12}{c13} \nbput{2}
    \ncline[linecolor=blue]{->}  {c32}{c33} \nbput{2}
    \ncline[offset=-2pt,linewidth=.2ex,linecolor=green]{->} {c13}{c23} \nbput{0}
    \ncline[offset=2pt,linewidth=.2ex,linecolor=red]{->} {c13}{c23} \naput{1}
  \end{displaymath}
  \caption{\label{fig:tensor-Q}The tensor product of two standard queer crystals for $U_{q}(\mathfrak{q}(3))$.}
\end{figure}

As the crystal operators $e_{\st{i}}$ and $f_{\st{i}}$ for $i>1$ are defined as a combination of other operators, they can alter both sides of the tensor. For example, we have
\[ f_{\st{2}}\left(\left(\raisebox{-0.3\cellsize}{$\tableau{1}$} \otimes \raisebox{-0.3\cellsize}{$\tableau{3}$}\right)\otimes \raisebox{-0.3\cellsize}{$\tableau{2}$}\right) = \left(\raisebox{-0.3\cellsize}{$\tableau{3}$} \otimes \raisebox{-0.3\cellsize}{$\tableau{3}$}\right)\otimes \raisebox{-0.3\cellsize}{$\tableau{1}$} . \]
Nevertheless, we have the following analog of Proposition~\ref{prop:normal-A}

\begin{proposition}
  Given two normal queer crystals $\Q_1$ and $\Q_2$, the tensor product $\Q_1 \otimes \Q_2$ is a normal crystal. Further, every connected normal queer crystal of dimension $r+1$ and degree $k$ arises as a connected component in $\Q(r+1)^{\otimes k}$, the $k$-fold tensor product of the standard crystal $\Q(r+1)$.
  \label{prop:queer-tensor}
\end{proposition}

For example, Figure~\ref{fig:tensor-Q} constructs the unique normal queer crystal of dimension $3$ and degree $2$, namely $\Q((2,0,0))$, from the tensor product of two copies of the standard queer crystal $\Q(3) = \Q((1,0,0))$.

%%%%%%%%%%%%%%%%%%%%%%%%%%%%%%%%%%%%%%%%%%%%%%%%%%%%%%%%%%%%%%%%
\subsection{Queer crystals for shifted tableaux}
%%%%%%%%%%%%%%%%%%%%%%%%%%%%%%%%%%%%%%%%%%%%%%%%%%%%%%%%%%%%%%%%
\label{sec:tableaux-Q}

Sergeev \cite{Ser84} established that the characters of irreducible tensor representations for the queer superalgebra are Schur $P$-functions. Grantcharov, Jung, Kang, Kashiwara, and Kim \cite{GJKKK14} developed crystal bases for the quantum queer superalgebra and gave an explicit construction of the queer crystal on semistandard decomposition tableaux, the latter being another combinatorial model for Schur $P$-polynomials introduced by Serrano \cite{Ser10}. Grantcharov, Jung, Kang, Kashiwara, and Kim raised the question of whether an explicit queer crystal could be defined directly on semistandard shifted tableaux. To answer this affirmatively we have the following construction building on the crystal graph defined in \S\ref{sec:crystal-B}.

Note that in a semistandard shifted tableau, both $1$ and $\st{2}$ can only be used in the first row, and any cell on the second row is at least $2$. This observation ensures the shifted queer lowering operator below is well-defined and acts only in the first row.

\begin{figure}[ht]
  \begin{center}
    \begin{displaymath}
      \begin{array}{ccc}  
        & \rnode{a11}{\tableau{ & 2 \\ 1 & 1}} & \\[7ex]
        \rnode{b10}{\tableau{ & 2 \\ 1 & \st{2}}} & & \rnode{b12}{\tableau{ & 3 \\ 1  & 1}} \\[7ex]
        \rnode{c11l}{\tableau{ & 3 \\ 1 & \st{2}}} & & \rnode{c11r}{\tableau{ & 3 \\ 1 & 2}} \\[7ex]
        \rnode{d9}{\tableau{ & 3 \\ 1 & \st{3}}} & & \rnode{d13}{ \tableau{ & 3 \\ 2 & 2}} \\[7ex]
        & \rnode{e11}{\tableau{ & 3 \\ 2 & \st{3}}} & 
      \end{array} \hspace{\cellsize}
      \begin{array}{ccccccc}
        & & & \rnode{D1}{\tableau{1 & 1 & 1}} & & & \\[7ex]
        & & \rnode{C2}{\tableau{1 & 1 & 2}} & & \rnode{E2}{\tableau{1 & 1 & \st{2}}} & & \\[7ex]
        & \rnode{B3}{\tableau{1 & 2 & 2}} & \rnode{C3}{\tableau{1 & 1 & 3}} & \rnode{D3}{\tableau{1 & \st{2} & 2}} & & \rnode{F3}{\tableau{1 & 1 & \st{3}}} & \\[7ex]
        \rnode{A4}{\tableau{2 & 2 & 2}} & & \rnode{C4}{\tableau{1 & 2 & 3}} & \rnode{D4}{\tableau{1 & \st{2} & 3}} & & \rnode{F4}{\tableau{1 & 2 & \st{3}}} & \rnode{G4}{\tableau{1 & \st{2} & \st{3}}} \\[7ex]
        & \rnode{B5}{\tableau{2 & 2 & 3}} & \rnode{C5}{\tableau{1 & 3 & 3}} & \rnode{D5}{\tableau{1 & \st{3} & 3}} & & \rnode{F5}{\tableau{2 & 2 & \st{3}}} & \\[7ex]
        & & \rnode{C6}{\tableau{2 & 3 & 3}} & & \rnode{E6}{\tableau{2 & \st{3} & 3}} & & \\[7ex]
        & & & \rnode{D7}{\tableau{3 & 3 & 3}} & & & 
      \end{array}
      \psset{linewidth=.2ex,nodesep=3pt}      
      % (210)
      \ncline[linecolor=blue]{<-} {b12}{a11}  \nbput{2}
      \ncline[offset=-3pt,linecolor=red,linewidth=.2ex]{->} {a11}{b10}  \nbput{1}
      \ncline[linecolor=red,linewidth=.2ex]{->} {b12}{c11r} \naput{1}
      \ncline[linecolor=blue]{->} {b10}{c11l} \nbput{2}
      \ncline[linecolor=blue]{->} {c11l}{d9}  \nbput{2}
      \ncline[offset=2pt,linecolor=red,linewidth=.2ex]{->} {c11r}{d13} \naput{1}
      \ncline[offset=-3pt,linecolor=red,linewidth=.2ex]{->} {d9}{e11}  \nbput{1}
      \ncline[linecolor=blue]{->} {d13}{e11}  \naput{2}
      % zero edges
      \ncline[offset=-1pt, linecolor=green,linewidth=.2ex]{<-} {e11}{d9}  \nbput{0}
      \ncline[offset=-1pt,linecolor=green,linewidth=.2ex]{<-} {b10}{a11}  \nbput{0}
      \ncline[offset=2pt,linecolor=green,linewidth=.2ex]{<-} {d13}{c11r}  \naput{0}
      \ncline[linecolor=green,linewidth=.2ex]{<-} {c11l}{b12}  \naput{0}      
      % (300)
      \ncline[linecolor=red]{->}  {D1}{C2}  \nbput{1}
      \ncline[linecolor=green]{->}{D1}{E2}  \naput{0}
      \ncline[linecolor=red]{->}  {C2}{B3}  \nbput{1}
      \ncline[linecolor=blue]{->} {C2}{C3}  \naput{2}
      \ncline[linecolor=green]{->}{C2}{D3}  \naput{0}
      \ncline[linecolor=red]{->}  {E2}{D3}  \nbput{1}
      \ncline[linecolor=blue]{->} {E2}{F3}  \naput{2}
      \ncline[offset=2pt,linecolor=red]{<-}  {A4}{B3}  \naput{1}
      \ncline[offset=2pt,linecolor=green]{->}{B3}{A4}  \naput{0}
      \ncline[linecolor=blue]{->} {B3}{C4}  \naput{2}
      \ncline[linecolor=red]{->}  {C3}{C4}  \naput{1}
      \ncline[linecolor=green]{->}{C3}{D4}  \naput{0}
      \ncline[linecolor=blue]{->} {D3}{D4}  \naput{2}
      \ncline[linecolor=red]{->}  {F3}{F4}  \naput{1}
      \ncline[linecolor=green]{->}{F3}{G4}  \naput{0}
      \ncline[linecolor=blue]{->} {A4}{B5}  \nbput{2}
      \ncline[offset=2pt,linecolor=red]{<-}  {B5}{C4}  \naput{1}
      \ncline[offset=2pt,linecolor=green]{->}{C4}{B5}  \naput{0}
      \ncline[linecolor=blue]{->} {C4}{C5}  \naput{2}
      \ncline[linecolor=blue]{->} {D4}{D5}  \naput{2}
      \ncline[offset=2pt,linecolor=red]{<-}  {F5}{F4}  \naput{1}
      \ncline[offset=2pt,linecolor=green]{->}{F4}{F5}  \naput{0}
      \ncline[linecolor=blue]{->} {B5}{C6}  \nbput{2}
      \ncline[offset=2pt,linecolor=red]{<-}  {C6}{C5}  \naput{1}
      \ncline[offset=2pt,linecolor=green]{->}{C5}{C6}  \naput{0}
      \ncline[linecolor=red]{->}  {D5}{E6}  \naput{1}
      \ncline[linecolor=blue]{->} {F5}{E6}  \naput{2}
      \ncline[linecolor=blue]{->} {C6}{D7}  \nbput{2}
    \end{displaymath}   
    \caption{\label{fig:P21-B} The queer crystal structures on semistandard shifted tableaux of shape $(2,1)$ and $(3)$ with entries $\{\st{1},1,\st{2},2,\st{3},3\}$ and no marks on the main diagonal.}
  \end{center}
\end{figure}

\begin{definition}
  The \emph{queer lowering operator}, denoted by $\st{f}_0$, acts on semistandard shifted tableaux by: if $T$ has no cell labeled $1$ or if $T$ has a cell labeled $\st{2}$, then $\st{f}_0(T)=0$; otherwise $\st{f}_0(T)$ changes the rightmost $1$ in the first row of $T$ to $2$ if it is on the main diagonal and to $\st{2}$ otherwise. 
  \label{def:queer-lower}
\end{definition}

For examples of the queer lowering operator on semistandard shifted tableaux, see Figures~\ref{fig:P21-B} and \ref{fig:P31-B}.

The $\st{f}_0$ edges may be identified from the normal crystal by connecting $T_1\rightarrow T_2$ if $T_1$ and $T_2$ differ in only one, which is labeled $1$ in $T_1$ and $\st{2}$ in $T_2$ if the cell is not on the main diagonal, $2$ if it is.

\begin{figure}[ht]
  \begin{center}
    \begin{displaymath}
      \begin{array}{c@{\hskip 2\cellsize}c@{\hskip 2\cellsize}c@{\hskip 2\cellsize}c@{\hskip 2\cellsize}c@{\hskip 2\cellsize}c}
        & & \rnode{c1}{\tableau{ & 2 \\ 1 & 1 & 1}} & & & \\[7ex]
        & \rnode{b2}{\tableau{ & 2 \\ 1 & 1 & 2}} & \rnode{c2}{\tableau{ & 3 \\ 1 & 1 & 1}} & \rnode{d2}{\tableau{ & 2 \\ 1 & 1 & \st{2}}} & & \\[7ex]
        \rnode{a3}{\tableau{ & 2 \\ 1 & \st{2} & 2}} & \rnode{b3}{\tableau{ & 2 \\ 1 & 1 & 3}} & \rnode{c3}{\tableau{ & 3 \\ 1 & 1 & 2}} & \rnode{d3}{\tableau{ & 3 \\ 1 & 1 & \st{2}}} & & \rnode{f3}{\tableau{ & 2 \\ 1 & 1 & \st{3}}}\\[7ex]
        \rnode{a4}{\tableau{ & 2 \\ 1 & \st{2} & 3}} & \rnode{b4}{\tableau{ & 3 \\ 1 & 1 & 3}} & \rnode{c4}{\tableau{ & 3 \\ 1 & 2 & 2}} & \rnode{d4}{\tableau{ & 3 \\ 1 & \st{2} & 2}} & \rnode{e4}{\tableau{ & 3 \\ 1 & 1 & \st{3}}} & \rnode{f4}{\tableau{ & 2 \\ 1 & \st{2} & \st{3}}}\\[7ex]
        \rnode{a5}{\tableau{ & 3 \\ 1 & \st{2} & 3}} & \rnode{b5}{\tableau{ & 3 \\ 1 & 2 & 3}} & \rnode{c5}{\tableau{ & 3 \\ 2 & 2 & 2}} & \rnode{d5}{\tableau{ & 3 \\ 1 & 2 & \st{3}}} & & \rnode{f5}{\tableau{ & 3 \\ 1 & \st{2} & \st{3}}}\\[7ex]
        & \rnode{b6}{\tableau{ & 3 \\ 1 & \st{3} & 3}} & \rnode{c6}{\tableau{ & 3 \\ 2 & 2 & 3}} & \rnode{d6}{\tableau{ & 3 \\ 2 & 2 & \st{3}}} & & \\[7ex]
        & & \rnode{c7}{\tableau{ & 3 \\ 2 & \st{3} & 3}} & & & 
      \end{array}
      \psset{nodesep=2pt,linewidth=.2ex}
      % RANK 1 TO 2
      \ncline[linewidth=.2ex,linecolor=red]{->} {c1}{b2} \nbput{1}
      \ncline[linecolor=blue]{->}  {c1}{c2} \naput{2}
      \ncline[linewidth=.3ex,linecolor=green]{->}   {c1}{d2} \naput{0}
      % RANK 2 TO 3
      \ncline[offset=2pt,linewidth=.3ex,linecolor=green]{->}   {b2}{a3} \nbput{1}
      \ncline[offset=2pt,linewidth=.2ex,linecolor=red]{<-} {a3}{b2} \nbput{0}
      \ncline[linecolor=blue]{->}  {b2}{b3} \naput{2}
      \ncline[linewidth=.2ex,linecolor=red]{->} {c2}{c3} \naput{1}
      \ncline[linewidth=.3ex,linecolor=green]{->}   {c2}{d3} \naput{0}
      \ncline[linecolor=blue]{->}  {d2}{d3} \naput{2}
      % RANK 3 TO 4
      \ncline[linecolor=blue]{->}  {a3}{a4} \nbput{2}
      \ncline[offset=2pt,linewidth=.3ex,linecolor=green]{->}   {b3}{a4} \nbput{1}
      \ncline[offset=2pt,linewidth=.2ex,linecolor=red]{<-} {a4}{b3} \nbput{0} 
      \ncline[linecolor=blue]{->}  {b3}{b4} \naput{2}
      \ncline[linewidth=.2ex,linecolor=red]{->} {c3}{c4} \naput{1}
      \ncline[linewidth=.3ex,linecolor=green]{->}   {c3}{d4} \naput{0}
      \ncline[linewidth=.2ex,linecolor=red]{->} {d3}{d4} \naput{1}
      \ncline[linecolor=blue]{->}  {d3}{e4} \naput{2}
      \ncline[offset=2pt,linewidth=.3ex,linecolor=green]{->}   {f3}{f4} \nbput{1}
      \ncline[offset=2pt,linewidth=.2ex,linecolor=red]{<-} {f4}{f3} \nbput{0} 
      % RANK 4 TO 5
      \ncline[linecolor=blue]{->}  {a4}{a5} \nbput{2}
      \ncline[linewidth=.3ex,linecolor=green]{->}   {b4}{a5} \naput{0}
      \ncline[linewidth=.2ex,linecolor=red]{->} {b4}{b5} \naput{1}
      \ncline[linecolor=blue]{->}  {c4}{b5} \nbput{2}
      \ncline[offset=2pt,linewidth=.3ex,linewidth=.3ex,linecolor=green]{->} {c4}{c5} \nbput{1}
      \ncline[offset=2pt,linewidth=.2ex,linecolor=red]{<-} {c5}{c4} \nbput{0}
      \ncline[linecolor=blue]{->}  {d4}{d5} \naput{2}
      \ncline[linewidth=.2ex,linecolor=red]{->} {e4}{d5} \naput{1}
      \ncline[linewidth=.3ex,linecolor=green]{->}   {e4}{f5} \naput{0}
      \ncline[linecolor=blue]{->}  {f4}{f5} \naput{2}
      % RANK 5 TO 6
      \ncline[linecolor=blue]{->}  {a5}{b6} \nbput{2}
      \ncline[offset=2pt,linewidth=.2ex,linecolor=red]{->}{b5}{c6} \nbput{0}
      \ncline[offset=2pt,linewidth=.3ex,linecolor=green]{<-}  {c6}{b5} \nbput{1} 
      \ncline[linecolor=blue]{->}  {c5}{c6} \naput{2}
      \ncline[offset=2pt,linewidth=.3ex,linecolor=green]{->} {d5}{d6} \nbput{1}
      \ncline[offset=2pt,linewidth=.2ex,linecolor=red]{<-} {d6}{d5} \nbput{0}
      % RANK 6 TO 7
      \ncline[offset=2pt,linewidth=.2ex,linecolor=red]{->}{b6}{c7} \nbput{0}
      \ncline[offset=2pt,linewidth=.3ex,linecolor=green]{<-}  {c7}{b6} \nbput{1} 
      \ncline[linecolor=blue]{->}  {c6}{c7} \naput{2}
    \end{displaymath}
    \caption{\label{fig:P31-B}The queer crystal structure on semistandard shifted tableaux of shape $(3,1)$ with entries $\{\st{1},1,\st{2},2,\st{3},3\}$ and no marks on the main diagonal.}
  \end{center}
\end{figure}

\begin{definition}
  The \emph{queer raising operator}, denoted by $\st{e}_0$, acts on semistandard shifted tableaux by: if $T$ has no cell labeled $\st{2}$ and the leftmost entry of the first row is not $2$, then $\st{e}_0(T)=0$; otherwise $\st{e}_0(T)$ changes the leftmost $\st{2}$ in the first row of $T$, if it exists, or the leftmost entry in the first row, otherwise, to $1$.
  \label{def:queer-raise}
\end{definition}

As required for a queer crystal, the queer raising and lowering operators are inverse to one another.

\begin{proposition}
  The queer raising and lowering operators satisfy $\st{e}_0(T) = T^{\prime}$ if and only if $\st{f}_0(T^{\prime}) = T$.
  \label{prop:queer-inverse}
\end{proposition}

\begin{proof}
  Suppose $\st{e}_0(T) = T^{\prime}$. Since $T$ is a semistandard shifted tableau, it has at most one $\st{2}$ in its first row. There are two disjoint cases: either (i) $T$ has a $\st{2}$ in its first row, or, since $\st{e}_0(T) \neq 0$, (ii) the leftmost entry in the first row is $2$. For case (i), $\st{e}_0$ changes the $\st{2}$ in $T$ to become the rightmost $1$ in $T^{\prime}$, therefore $\st{f}_0$ will act non-trivially on $T^{\prime}$ by changing this entry back to $\st{2}$. For case (ii), $\st{e}_0$ changes the leftmost entry in the first row to a $1$ in $T^{\prime}$, therefore $\st{f}_0$ will act non-trivially on $T^{\prime}$ by changing this entry back to $2$. Thus $\st{f}_0(T^{\prime}) = T$.

  Suppose $\st{f}_0(T^{\prime}) = T$. Since $\st{f}_0(T^{\prime}) \neq 0$, $T^{\prime}$ must have a $1$ and no $\st{2}$ in the first row. Again, we have two disjoint cases: either (i) $T$ has a unique $1$ in its first row on the main diagonal, or (ii) the rightmost $1$ in the first row of $T$ is not on the main diagonal. For case (i), $\st{f}_0$ changes the unique $1$ to a $2$ in $T$, therefore $\st{e}_0$ will act non-trivially on $T$ by changing this entry back to $1$. For case (ii), $\st{f}_0$ changes the rightmost $1$ in $T^{\prime}$ to a $\st{2}$ in $T$, therefore $\st{e}_0$ will act non-trivially on $T$ by changing this entry back to $1$. Thus $\st{e}_0(T) = T^{\prime}$.
\end{proof}

As we saw in the example for $\gamma=(3,1)$ in Figure~\ref{fig:P31}, the normal crystal $(\SSHT_n(\gamma),\{\st{e}_i,\st{f}_i\}_{1 \leq i < n},\wt)$ is not always connected. However, the queer crystal obtained by augmenting this with the queer raising and lowering operators is connected. For example, see Figure~\ref{fig:P31-B}. 

\begin{lemma}
  For a strict partition $\gamma$, let $T \in \Yam(\gamma)$ have the smallest primed entry given by $\st{k+1}$ for some $k\geq 2$. Then $S_{w_k}(T)$ has a coordinate labeled $\st{2}$.
  \label{lem:connected}
\end{lemma}

\begin{proof}
  Let $(a_1,a_2,\ldots,a_n)$ denote the weight of $T$. Consider boxes on $T$ with labels $\leq k$. As $T \in \Yam(\gamma)$, we have $m_i(T)=0$ for all $i$, therefore these boxes form a sub-diagram $\theta$ of size $(a_1,a_2,\ldots, a_k)$ with all the boxes on row $i$ are labeled $i$, as illustrated in the left tableau of Figure~\ref{fig:connected}. Note that we have $a_1> a_2> \ldots> a_k \geq a_{k+1}$. 

  Let $T'$ denote  $S_{k-1} S_{k-1}\cdots S_1 (T)$. Then, $T'$ has weight $(a_k,a_1,a_2,\ldots a_{k-1},a_{k+1},a_{k+2},\ldots a_n)$ and outside $\theta$ it matches $T$ exactly. On $\theta$, the $i$th row has $a_{k}$ cells labeled $i$, one cell labeled $\st{i+1}$ and $a_i-a_{k}-1$ cells labeled ${i+1}$ for all $i<k$, and the $k$th row contains $a_k$ cells labeled $k$. In particular, the $k+1$st southwest to northeast diagonal is formed by the primed entries $\st{2},\st{3},\ldots,\st{k}$. The operation $S_{i}S_{i+1}\ldots S_{k}(T')$ preserves the cell labeled $\st{i}$ on the said diagonal, as well as all entries with labels less than $i$. This implies that $S_{w_k}(T)=S_{2}S_{3}\ldots S_{k}(T')$ contains an entry marked $\st{2}$.
\end{proof}

\begin{figure}[ht]
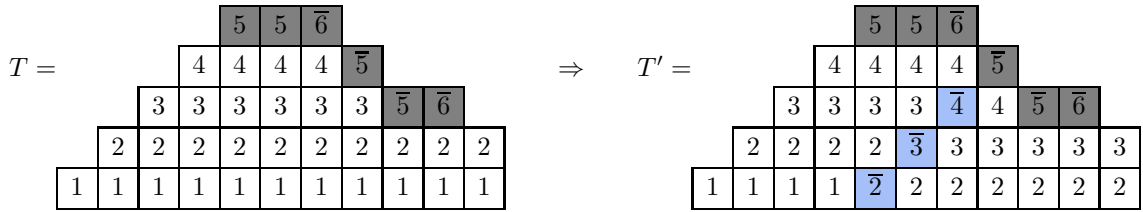

  \begin{displaymath}
    \raisebox{-.4cm}{$T=$}\begin{ytableau}
      \none &\none&\none&\none& *(gray)5 &*(gray)5 &*(gray)\st{6}\\
      \none& \none &\none & 4 & 4 & 4 & 4 &*(gray) \st{5}\\
      \none &\none& 3 & 3 & 3 & 3 & 3 & 3 & *(gray)\st{5}& *(gray)\st{6}\\
      \none & 2 & 2 & 2 & 2 & 2 & 2 & 2 & 2 & 2 &2\\
      1 & 1 & 1 & 1 & 1 & 1 & 1 &  1 & 1 & 1 & 1
    \end{ytableau}\qquad \raisebox{-.4cm}{$\Rightarrow$}\qquad \raisebox{-.4cm}{$T'=$}
    \begin{ytableau}
      \none &\none&\none&\none& *(gray)5 &*(gray)5 &*(gray)\st{6}\\
      \none& \none &\none & 4 & 4 & 4 &4 & *(gray) \st{5}\\
      \none &\none & 3 & 3 & 3 & 3 & *(lblue) \st{4} &  4 & *(gray)\st{5}& *(gray)\st{6}\\
      \none & 2 & 2 & 2 & 2 & *(lblue)\st{3} & 3 & 3 & 3 & 3 &3\\
      1 & 1 & 1 & 1 & *(lblue)\st{2} & 2 & 2 &  2 & 2 & 2 & 2
    \end{ytableau}
  \end{displaymath}
  \caption{\label{fig:connected}An example for Lemma~\ref{lem:connected}, with the outside of $\theta$ shown in gray.} 
\end{figure}

\begin{theorem}
  For $\gamma$ a strict partition, the shifted and queer raising and lowering operators $\st{e}_i, \st{f}_i$ for $i=1,2,\ldots,r$ and $\st{e}_0, \st{f}_0$ define a connected, normal queer crystal on $\SSHT_{r+1}(\gamma)$.
  \label{thm:Q-crystal}  
\end{theorem}

\begin{proof}
  Note that without the shifted raising and lowering operators, we have a collection of crystals connected to highest weights $T \in A(\gamma)$. We will show that with the added shifted operators, the shifted tableaux $T$ with weight $\gamma$ is the unique highest weight element, implying the crystal is connected. Let $T \in A(\gamma)$ be a Yamanouchi shifted tableau with $\wt(T)\neq \gamma$. By Lemma~\ref{lem:Yamanouchi}, any unmarked entry in row $i$ must equal $i$, so $T$ must contain a marked entry. If $T$ contains a cell labeled $\st{2}$, $\st{e}_1(T)\neq 0$ is an element of higher weight. Otherwise let $k>2$ denote the smallest primed entry in $T$. By Lemma \ref{lem:connected}, $S_{w_k}(T)$ has a coordinate labeled $\st{2}$, so $\st{e}_0S_{w_k}(T)\neq 0$ and consequentially $\st{e}_k(T)\neq 0$. Therefore the graph is connected.

  Definition~\ref{def:base-Q} holds for the shifted raising and lowering operators by Theorem~\ref{thm:crystal} and for the queer raising and lowering operators by Proposition~\ref{prop:queer-inverse}. To see that the queer crystal is normal, by Proposition~\ref{prop:HPS}, we know that applying Sagan's shifted insertion to the word $b_n \ldots b_2 b_1$ results in a shifted tableau that can be identified with the type A crystal basis $b_1 \otimes b_2 \otimes \cdots \otimes b_n$. We show by induction that our queer operator defined directly on shifted tableaux agrees with Definition~\ref{def:tensor-Q}, noting that the base case follows immediately from the standard queer crystal. Consider, then, a shifted tableau $T$ of degree $n-1 \geq 1$. By Definition~\ref{def:tensor-Q}, we have
  \begin{equation}
    f_0 \left( \tableau{i} \otimes T \right) = \left\{ \begin{array}{ll}
      \tableau{2} \otimes T & \mbox{if $i=1$ and $\wt(T)_1 = \wt(T)_2 = 0$,} \\
      0 & \mbox{if $i \geq 2$ and $\wt(T)_1 = \wt(T)_2 = 0$,}\\
      0 & \mbox{if $f_0(T)=0$ and either $\wt(T)_1 >0$ or $\wt(T)_2 > 0$,}\\
      \tableau{i} \otimes f_0(T) & \mbox{if $f_0(T)\neq 0$.}
    \end{array} \right.
    \label{e:ten}
  \end{equation}
  Consider the shifted insertion of $i$ into $T$, denoted by $T \leftarrow i$. If $\wt(T)_1 = \wt(T)_2 = 0$, then inserting $1$ into $T$ has the same bumping path as inserting $2$ into $T$, since all letters of $T$ are larger than both $1$ and $2$. Therefore $T \leftarrow 2$ is precisely $T \leftarrow 1$ with the $1$ changed to a $2$. In particular,
  \begin{displaymath}
    f_0(T \leftarrow 1) = T \leftarrow 2 \ \mbox{whenever} \ \wt(T)_1 = \wt(T)_2 = 0.
  \end{displaymath}
  Similarly, if $\wt(T)_1 = \wt(T)_2 = 0$, then $T \leftarrow i$ has no entry equal to $1$ for $i\geq 2$, and so 
  \begin{displaymath}
    f_0(T \leftarrow i) = 0 \ \mbox{whenever} \ \wt(T)_1 = \wt(T)_2 = 0 \mbox{ and } i \geq 2.
  \end{displaymath}
  If $f_0(T) = 0$ and either $\wt(T)_1 >0$ or $\wt(T)_2 > 0$, then either $T$ has an entry $\st{2}$ in the first row, or $T$ has only entries weakly greater than $2$ in the first row. In these cases, $T \leftarrow i$ will have the same property for $i \geq 2$. For $T \leftarrow 1$, the $1$ will either bump the $\st{2}$, if it exists, or will bump a $2$ in the first row, if it doesn't, with the result that $T \leftarrow 1$ will have a $\st{2}$ in the first row. Therefore, in all cases, 
  \begin{displaymath}
    f_0(T \leftarrow i) = 0 \ \mbox{whenever} \ f_0(T) = 0 \mbox{ and either } \wt(T)_1 > 0 \mbox{ or } \wt(T)_2 > 0.
  \end{displaymath}
  Finally, if $f_0(T) \neq 0$ and either $\wt(T)_1 >0$ or $\wt(T)_2 > 0$, then we must in fact have a $1$ in the first row and no $\st{2}$ in $T$. In this case, $T \leftarrow i$ will not affect any entries $1$, $\st{2}$, or $2$ in $T$ for $i \geq 3$, and for $i\geq 2$ will at most insert an additional $2$. In these cases, the rightmost $1$ of $T$ changing to $2$ does not alter the insertion path, so $f_0(T \leftarrow i) = f_0(T) \leftarrow i$. For the case $T \leftarrow 1$, the inserted $1$ might bump a $2$, but insodoing pushes it to a higher row since the $2$ cannot be on the diagonal in the first row (since $T$ has a $1$). Similarly, the insertion $f_0(T) \leftarrow 1$ will have the $1$ bump the newly created $\st{2}$ which will then follow the bumping path of $T \leftarrow 1$. Thus again we have $f_0(T \leftarrow 1) = f_0(T) \leftarrow 1$, and so 
  \begin{displaymath}
    f_0(T \leftarrow i) = f_0(T) \leftarrow i \ \mbox{whenever} \ f_0(T) \neq 0 .
  \end{displaymath}
  Therefore we have shown the following,
  \begin{equation}
    f_0 \left( T \leftarrow i\right) = \left\{ \begin{array}{ll}
      T \leftarrow 2 & \mbox{if $i=1$ and $\wt(T)_1 = \wt(T)_2 = 0$,} \\
      0 & \mbox{if $i \geq 2$ and $\wt(T)_1 = \wt(T)_2 = 0$,}\\
      0 & \mbox{if $f_0(T)=0$ and $\wt(T)_1 >0$ or $\wt(T)_2 > 0$,}\\
      f_0(T) \leftarrow i & \mbox{if $f_0(T) \neq 0$.}
    \end{array} \right.
    \label{e:ins}
  \end{equation}
  The result follows by comparison of cases between \eqref{e:ten} and \eqref{e:ins} and induction on $n$.
\end{proof}

Using the odd crystal operators $\st{e}_{\st{i}}$, we may characterize our normal queer crystals on semistandard shifted tableaux by their highest weights as in Definition~\ref{def:hw-Q}. For example, removing the $\st{f}_0$ edges and inserting edges $\st{f}_{\st{1}} = \st{f}_{0}$ and $\st{f}_{\st{1}} = S_{1} S_{2} \st{f}_{0} S_{2} S_{1}$ for the queer crystal for $\SSHT_3(3,1)$ results in the crystal shown in Figure~\ref{fig:P31-Q}, which clearly has a unique highest weight.

\begin{figure}[ht]
  \begin{center}
    \begin{displaymath}
      \begin{array}{c@{\hskip 2\cellsize}c@{\hskip 2\cellsize}c@{\hskip 2\cellsize}c@{\hskip 2\cellsize}c@{\hskip 2\cellsize}c}
        & & \rnode{c1}{\tableau{ & 2 \\ 1 & 1 & 1}} & & & \\[7ex]
        & \rnode{b2}{\tableau{ & 2 \\ 1 & 1 & 2}} & \rnode{c2}{\tableau{ & 3 \\ 1 & 1 & 1}} & \rnode{d2}{\tableau{ & 2 \\ 1 & 1 & \st{2}}} & & \\[7ex]
        \rnode{a3}{\tableau{ & 2 \\ 1 & \st{2} & 2}} & \rnode{b3}{\tableau{ & 2 \\ 1 & 1 & 3}} & \rnode{c3}{\tableau{ & 3 \\ 1 & 1 & 2}} & \rnode{d3}{\tableau{ & 3 \\ 1 & 1 & \st{2}}} & & \rnode{f3}{\tableau{ & 2 \\ 1 & 1 & \st{3}}}\\[7ex]
        \rnode{a4}{\tableau{ & 2 \\ 1 & \st{2} & 3}} & \rnode{b4}{\tableau{ & 3 \\ 1 & 1 & 3}} & \rnode{c4}{\tableau{ & 3 \\ 1 & 2 & 2}} & \rnode{d4}{\tableau{ & 3 \\ 1 & \st{2} & 2}} & \rnode{e4}{\tableau{ & 3 \\ 1 & 1 & \st{3}}} & \rnode{f4}{\tableau{ & 2 \\ 1 & \st{2} & \st{3}}}\\[7ex]
        \rnode{a5}{\tableau{ & 3 \\ 1 & \st{2} & 3}} & \rnode{b5}{\tableau{ & 3 \\ 1 & 2 & 3}} & \rnode{c5}{\tableau{ & 3 \\ 2 & 2 & 2}} & \rnode{d5}{\tableau{ & 3 \\ 1 & 2 & \st{3}}} & & \rnode{f5}{\tableau{ & 3 \\ 1 & \st{2} & \st{3}}}\\[7ex]
        & \rnode{b6}{\tableau{ & 3 \\ 1 & \st{3} & 3}} & \rnode{c6}{\tableau{ & 3 \\ 2 & 2 & 3}} & \rnode{d6}{\tableau{ & 3 \\ 2 & 2 & \st{3}}} & & \\[7ex]
        & & \rnode{c7}{\tableau{ & 3 \\ 2 & \st{3} & 3}} & & & 
      \end{array}
      \psset{nodesep=2pt,linewidth=.2ex}
      % RANK 1 TO 2
      \ncline[linewidth=.2ex,linecolor=red]{->} {c1}{b2} \nbput{1}
      \ncline[offset=2pt,linecolor=cyan]{->}  {c1}{c2} \nbput{2}
      \ncline[offset=2pt,linecolor=blue]{<-}  {c2}{c1} \nbput{\st{2}}
      \ncline[linewidth=.3ex,linecolor=magenta]{->}   {c1}{d2} \naput{\st{1}}
      % RANK 2 TO 3
      \ncline[offset=2pt,linewidth=.3ex,linecolor=magenta]{->}   {b2}{a3} \nbput{1}
      \ncline[offset=2pt,linewidth=.2ex,linecolor=red]{<-} {a3}{b2} \nbput{\st{1}}
      \ncline[offset=2pt,linecolor=cyan]{->}  {b2}{b3} \nbput{2}
      \ncline[offset=2pt,linecolor=blue]{<-}  {b3}{b2} \nbput{\st{2}}
      \ncline[linewidth=.2ex,linecolor=red]{->} {c2}{c3} \naput{1}
      \ncline[linewidth=.3ex,linecolor=magenta]{->}   {c2}{d3} \naput{\st{1}}
      \ncline[linecolor=blue]{->}  {d2}{d3} \naput{2}
      \ncline[linecolor=cyan]{->}  {d2}{f3} \naput{\st{2}}
      % RANK 3 TO 4
      \ncline[linecolor=blue]{->}  {a3}{a4} \nbput{2}
      \ncline[linecolor=cyan]{->}  {a3}{d4} \naput{\st{2}}
      \ncline[offset=2pt,linewidth=.3ex,linecolor=magenta]{->}   {b3}{a4} \nbput{1}
      \ncline[offset=2pt,linewidth=.2ex,linecolor=red]{<-} {a4}{b3} \nbput{\st{1}} 
      \ncline[linecolor=blue]{->}  {b3}{b4} \naput{2}
      \ncline[linecolor=cyan]{->}  {c3}{b4} \naput{\st{2}}
      \ncline[linewidth=.2ex,linecolor=red]{->} {c3}{c4} \naput{1}
      \ncline[linewidth=.3ex,linecolor=magenta]{->}   {c3}{d4} \naput{\st{1}}
      \ncline[linewidth=.2ex,linecolor=red]{->} {d3}{d4} \naput{1}
      \ncline[offset=2pt,linecolor=cyan]{->}  {d3}{e4} \nbput{2}
      \ncline[offset=2pt,linecolor=blue]{<-}  {e4}{d3} \nbput{\st{2}}
      \ncline[offset=2pt,linewidth=.3ex,linecolor=magenta]{->}   {f3}{f4} \nbput{1}
      \ncline[offset=2pt,linewidth=.2ex,linecolor=red]{<-} {f4}{f3} \nbput{\st{1}} 
      % RANK 4 TO 5
      \ncline[offset=2pt,linecolor=cyan]{->}  {a4}{a5} \nbput{2}
      \ncline[offset=2pt,linecolor=blue]{<-}  {a5}{a4} \nbput{\st{2}}
      \ncline[linewidth=.3ex,linecolor=magenta]{->}   {b4}{a5} \naput{\st{1}}
      \ncline[linewidth=.2ex,linecolor=red]{->} {b4}{b5} \naput{1}
      \ncline[linecolor=blue]{->}  {c4}{b5} \nbput{2}
      \ncline[linecolor=cyan]{->}  {c4}{d5} \naput{\st{2}}
      \ncline[offset=2pt,linewidth=.3ex,linewidth=.3ex,linecolor=magenta]{->} {c4}{c5} \nbput{1}
      \ncline[offset=2pt,linewidth=.2ex,linecolor=red]{<-} {c5}{c4} \nbput{\st{1}}
      \ncline[linecolor=blue]{->}  {d4}{d5} \naput{2}
      \ncline[linewidth=.2ex,linecolor=red]{->} {e4}{d5} \naput{1}
      \ncline[linewidth=.3ex,linecolor=magenta]{->}   {e4}{f5} \naput{\st{1}}
      \ncline[offset=2pt,linecolor=cyan]{->}  {f4}{f5} \nbput{2}
      \ncline[offset=2pt,linecolor=blue]{<-}  {f5}{f4} \nbput{\st{2}}
      % RANK 5 TO 6
      \ncline[linecolor=blue]{->}  {a5}{b6} \nbput{2}
      \ncline[linecolor=cyan]{->}  {b5}{b6} \naput{\st{2}}
      \ncline[offset=2pt,linewidth=.2ex,linecolor=red]{->}{b5}{c6} \nbput{\st{1}}
      \ncline[offset=2pt,linewidth=.3ex,linecolor=magenta]{<-}  {c6}{b5} \nbput{1} 
      \ncline[linecolor=blue]{->}  {c5}{c6} \naput{2}
      \ncline[linecolor=cyan]{->}  {c5}{d6} \naput{\st{2}}
      \ncline[offset=2pt,linewidth=.3ex,linecolor=magenta]{->} {d5}{d6} \nbput{1}
      \ncline[offset=2pt,linewidth=.2ex,linecolor=red]{<-} {d6}{d5} \nbput{\st{1}}
      % RANK 6 TO 7
      \ncline[offset=2pt,linewidth=.2ex,linecolor=red]{->}{b6}{c7} \nbput{\st{1}}
      \ncline[offset=2pt,linewidth=.3ex,linecolor=magenta]{<-}  {c7}{b6} \nbput{1} 
      \ncline[offset=2pt,linecolor=cyan]{->}  {c6}{c7} \nbput{2}
      \ncline[offset=2pt,linecolor=blue]{<-}  {c7}{c6} \nbput{\st{2}}
    \end{displaymath}
    \caption{\label{fig:P31-Q}The queer crystal $(\SSHT_3(3,1),\{\st{f}_1,\st{f}_{\st{1}},\st{f}_2,\st{f}_{\st{2}}\},\wt)$.}
  \end{center}
\end{figure}

By Proposition~\ref{prop:queer-tensor}, the tensor product of two normal queer crystals is again a normal queer crystal. This gives an explicit formula for the Schur $P$-expansion of a product of Schur $P$-polynomials.

\begin{corollary}
  For $\gamma,\delta$ strict partitions, we have
  \begin{equation}
    P_{\gamma} (x_1,\ldots,x_{n}) P_{\delta} (x_1,\ldots,x_{n}) = \sum_{\substack{(S,T) \in \SSHT(\gamma)\times\SSHT(\delta) \\ \st{e}_{i}(S\otimes T) = 0 = \st{e}_{\st{i}}(S\otimes T) \forall i}} P_{\wt(S)+\wt(T)}(x_1,\ldots,x_{n}) ,
  \end{equation}
  where the sum of weights is coordinate-wise. In particular, the product of Schur $P$-polynomials is Schur $P$-positive with coefficients given by
  \begin{equation}
    f_{\gamma,\delta}^{\epsilon} = \#\{(S,T) \in \SSHT(\gamma)\times\SSHT(\delta)  \mid \wt(S)+\wt(T)=\epsilon, \ \st{e}_{i}(S \otimes T) = 0 = \st{e}_{\st{i}}(S \otimes T)  \forall i \} .
  \end{equation}
  \label{cor:PP2P}
\end{corollary}

%%%%%%%%%%%%%%%%%%%%%%%%%%%%%%%%%%%%%%%%%%%%%%%%%%%%%%%%%%%%%%%%
\subsection{Local characterization for queer crystals}
%%%%%%%%%%%%%%%%%%%%%%%%%%%%%%%%%%%%%%%%%%%%%%%%%%%%%%%%%%%%%%%%
\label{sec:local-Q}

Following Stembridge \cite{Ste03}, we desire a local characterization of normal queer crystals to aide in proving that a given queer crystal is, in fact, normal.

To this end, a \emph{queer graph of dimension $r+1$} will mean a directed, colored graph $\mathcal{Y}$ with directed edges $e_i(x) {\blue \iar} x {\blue \iar} f_i(x)$ for $i = 0,1,2,\ldots,r$, and we adopt notation from \S\ref{sec:local-A}. Every queer crystal gives us a queer graph.

\begin{definition}
  A queer crystals graph $\mathcal{Y}$ is a {\em queer regular graph} if the following hold:

  \begin{itemize}
  \item[(B0)] The subgraph $\mathcal{Y}_+$ generated by edges with non-zero labels is a regular graph.
  
  \item[(B1)] all $0$ paths have length $1$, and $\upe_0(x)+\downf_0(x)=1$ if and only if $wt_1(x)+wt_2(x)>0$;

  \item[(B2)] for every vertex $x$, there is at most one edge $x {\color{ForestGreen} {\ensuremath\stackrel{0}{\longleftarrow}}} y$ and at most one edge $x {\color{ForestGreen} {\ensuremath\stackrel{0}{\longrightarrow}}} z$;

  \item[(B3)] assuming $e_0 x$ is defined, $\Delta_0 \upe_i(x) + \Delta_0 \downf_i(x) = \left\{
    \begin{array}{rl}
        2 & \;\mbox{if}\;\; i \leq 1 \\
        -1 & \;\mbox{if}\;\; i=2 \\
        0 & \;\mbox{if}\;\; i\geq 3
    \end{array}
    \right.$;
   
  \item[(B4)] assuming $e_0 x$ is defined,   $ \begin{array}{rl}
        \Delta_0 \upe_i(x)\geq0, \Delta_0 \downf_i(x) > 0 & \;\mbox{if}\;\; i=1 \\
                \Delta_0 \upe_i(x)\leq 0 , \Delta_0 \downf_i(x) \leq 0 & \;\mbox{if}\;\; i = 2 \\
                        \Delta_0 \upe_i(x)=0, \Delta_0 \downf_i(x) = 0 & \;\mbox{if}\;\; i \geq 3 \\
    \end{array} $

  \item[(B5)]For $i\geq 2$ $e_i x=e_0 y=z \Rightarrow f_i f_0 z=f_0 f_i z$; \\
             For $i=1$ or $i \geq 3$ $f_i x=f_0 y=z , x\neq y  \Rightarrow e_i e_0 z=e_0 e_i z$; 

  \item[(B6)] assuming $e_0 x$ is defined,$\begin{array}{ll}
\Delta_0 \upe_1(x)=1 \Rightarrow \downf_1(x)=0 \text{ and } e_1x=e_0 x\\
  \Delta_0 \downf_2(x)=0 \Leftrightarrow \downf_2(x)=0
    \end{array}$. 
  \end{itemize}
  \label{def:queer-regular}
\end{definition}

Axiom B0 relies on Stembridge's characterization of regular graphs (Definition~\ref{def:regular}). The other six axioms in Definition~\ref{def:queer-regular} give the analog of the corresponding axioms in Definition~\ref{def:regular} for the queer raising and lowering operators. To begin to justify our definition, we have the following result.

\begin{theorem}
  Every normal queer crystal is a regular queer graph.
\label{thm:structure-crystal-Q}
\end{theorem}

\begin{proof}
  Axioms B0 and B2 follow directly from the definition of queer crystals. Also, for any $i\geq3$, the operators $\st{e}_i$ and $\st{f}_i$ only affect cells labeled $i$, $\st{i}$, $i+1$ or $\st{i+1}$ so they are fully independent of the queer operators $\st{e}_0$ and $\st{f}_0$.  This is enough to establish the statements for $i\geq 3$, therefore we need only look at how $0$ moves interact with $1$ and $2$ paths.

  For axiom B1, assume $f_0(x)=y$. Then either $y$ contains a $\st{2}$ or the leftmost box on its first row is labeled $2$. In either case we have $f_0y=0$, so all $0$ strings have length $1$, and $\upe_0(x)+\downf_0(x)\leq 1$. For the second part, assume $\upe_0(x)+\downf_0(x)=0$. Then $e_0(x)=0$, so $x$ contains  no $\st{2}$, and $f_0(x)=0$ implies $x$ contains no $1$. Also the leftmost box of the first row of $x$ can not be $2$ as $e_0(x)=0$, so $x$ contains no $2$ either..

  For axioms B3, B4, and B6, recall $\downf_i(T)=m_i(T)$ and $\upe_i(T)$ is equal to the difference between the number of $i+1$s and the number of $i$s to the right of $w_q$, where $q$ is the largest index such that $m_i(w(T),q) = m_i(w(T))> 0$  ( if $m_i(w(T))\leq 0$, $\upe_i(T)$ is the difference between the total number of $i+1$s and the total number of $i$s). First consider how $e_0$ affects the $1$-string. Assume $e_0(x)=y$. Then $y$ contains no $\st{2}$, and at least one $1$. If it contains a single $1$, it is on the main diagonal, and $f_0$ changes to a $2$, so that we have $e_1(x)=e_0(x)=y$, and (B1) implies $\Delta_0 \upe_1(x)= \Delta_0 \downf_1(x) = 1$. Note that in this case $m_1(w(x))=\downf_1(x)=0$. If $y$ contains $k>1$ cells  labeled $1$, $e_0$ acts by changing the rightmost $1$ to a $2'$. The length of the $\upe_1$ string is given by the number of $2$s on the first row and remains unchanged. The $m_1$ value is decreased by $2$ as the rightmost $1$ is deleted, and replaced with a $2$ that comes to the left of all other $1$s in the reading word. In this case we have $\Delta_0 \upe_1(x)=2$, $\Delta_0 \downf_1(x) = 0$. 
    
  Now let us look at the possibilities for $\Delta_0\downf_2(x)$ and $\Delta_0 \upe_2(x)$. Assume $e_0(x)=y$. Note that the difference between the reading words of $x$ and $y$ is that we have one less $2$ and an extra $1$. If $f_2(x)=0$, then  $f_2(y)=0$ as well, and the difference between the total number of $3$s and $2$s is increased by one, which means $\Delta_0 \downf_2(x)=0$ and $\Delta_0 \upe_2(x)=-1$. This deals with the case $\downf_2(x)=0$. Let us now assume $\downf_2(x)=k>0$. Assume $q$ is the largest index where $m_2$ is achieved. If $x$ has a $2'$, that comes before all the other $2$ in the reading word. Otherwise, $x$ has no cells labeled $1$, so the second row contains no $2$, and the leftmost $2$ on the first row comes before all other $2$ on the reading word. In both cases, the $2$ that turns in  to $1$ with the $e_0$ move has an index $\leq q$, so the $e_0$ move increases $m_2$ by 1, and does not change $\upe_2$. We have $\Delta_0 \downf_2(x)=-1$ and $\Delta_0 \upe_2(x)=0$.

  Finally, for axiom B5, let us first assume that we have $e_2(x)=e_0(y)=z$. The action $f_0$ on $z$ creates a new $2$ or $\st{2}$ that comes before all the other $2$ on the reading word, so $f_2(y)$ is defined, and the algorithm selects the same cell labeled $2$ as in $f_2(z)$. Furthermore as the moves (L1) are independent of the changes happening strictly to the left and weakly below the selected cell, it commutes with the action of $(f_0)$: $f_0(f_2(z))=f_2(f_0(z))$.
   Now let us assume $f_1(x)=f_0(y)=z$ with $x\neq y$. As there can be no cells labeled $\st{2}$ on the second row, and no cells labeled $\st{1}$ anywhere on a shifted tableaux, $f_1$ acts on $x$ by (L1) a,b or d of Definition \ref{def:shifted-lower}. We can further eliminate (L1) b and (L1) d, as $x\neq y$. So the first row of $z$ contains a $\st{2}$ adjacent to a $2$. $ e_1 (e_0 (z))=e_0 (e_1 (z))$ gives the tableaux where both these entries are replaced with 1.
\end{proof}

Similar to Figure~\ref{fig:P5P6} giving a graphical illustration of the local connected component of a regular crystal when considering only two string colors, the following two lemmas give graphical illustrations of the local structure of a regular queer crystal for two color components involving $0$ edges.

\begin{lemma}
  Connected components of the subgraph generated by $f_0$ and $f_1$ of a regular queer graph are of the form shown in Figure~\ref{fig:queer0-1}.
  \label{lem:queer0-1}
\end{lemma}

\begin{figure}[ht]
  \begin{displaymath}
    \begin{array}{c@{\hskip2\cellsize}c@{\hskip2\cellsize}c@{\hskip3\cellsize}c@{\hskip2\cellsize}c@{\hskip2\cellsize}c@{\hskip3\cellsize} c@{\hskip2\cellsize}}
      \rnode{a1}{\bullet} & \rnode{b1}{\bullet} &  &  &  &  &  \\[2\cellsize]
      %\rnode{a2}{\bullet} & \rnode{b2}{\bullet} &  &  &  &  & \\[2\cellsize]
      \rnode{a3}{\bullet} & \rnode{b3}{\bullet} &  &  &  &  &  \\[2\cellsize]
      \rnode{a4}{\bullet} & \rnode{b4}{\bullet} &  & \rnode{c4}{\bullet} & \rnode{d4}{\bullet} &  &  \\[2\cellsize]
      \rnode{a5}{\bullet} &  &  & \rnode{c5}{\bullet} &  &  & \rnode{e5}{\bullet} \\[2\cellsize] 
      \rnode{a6}{\bullet} &  &  & \rnode{c6}{\bullet} &  &  & \rnode{e6}{\bullet}\\[1\cellsize] 
      k > 2 &  &  & k = 2  &  &  & k = 1 
    \end{array} 
    \psset{linewidth=.2ex,nodesep=3pt}
    \everypsbox{\scriptstyle}
    \ncline[linecolor=red,linestyle=dotted,linewidth=.3ex] {a1}{a3}
    \ncline[linecolor=red,linestyle=dotted,linewidth=.3ex] {b1}{b3}
    %\ncline[linecolor=red]{->}              {a2}{a3} \nbput{1}
    \ncline[linecolor=red]{->}              {a3}{a4} \nbput{1}
    \ncline[linecolor=red]{->}              {a4}{a5} \nbput{1}
    %\ncline[linecolor=red]{->}              {b2}{b3} \nbput{1}
    \ncline[linecolor=red]{->}              {b3}{b4} \nbput{1}
    \ncline[linecolor=green]{->}              {a1}{b1} \nbput{0}
    %\ncline[linecolor=green]{->}              {a2}{b2} \nbput{0}
    \ncline[linecolor=green]{->}              {a3}{b3} \nbput{0}
    \ncline[linecolor=green]{->}              {a4}{b4} \nbput{0}
    \ncline[offset=2pt,linecolor=red]{->}{a5}{a6} \nbput{0}
    \ncline[offset=2pt,linecolor=green]{<-}  {a6}{a5} \nbput{1}      
    \ncline[linecolor=green]{->}              {c4}{d4} \nbput{0}
    \ncline[linecolor=red]{->}              {c4}{c5} \nbput{1}
    \ncline[offset=2pt,linecolor=red]{->}{c5}{c6} \nbput{0}
    \ncline[offset=2pt,linecolor=green]{<-}  {c6}{c5} \nbput{1}
    \ncline[offset=2pt,linecolor=red]{->}{e5}{e6} \nbput{0}
    \ncline[offset=2pt,linecolor=green]{<-}  {e6}{e5} \nbput{1}       
  \end{displaymath}
  \caption{\label{fig:queer0-1}Possible connected components for $f_0$ and $\st{f}_1$ in a regular queer graph.}
\end{figure}

\begin{proof}
  By (B3), every connected component will have at least one $1$ edge. Consider a maximal $1$ string  of length $k\geq 1$ on a connected component. Let $x$ be the on this string with $\downf_1(x)=0$. 

  If $f_0 x = y$ for some $y$, then as $\downf_1 (y)$ can not be less than $0$, $\Delta_0 \downf_1(x)=0$ by (B4). This can not happen as $\Delta_0 \upe_1(x)=2$ would imply $y$ is on a longer $1$ string. So $f_0 x=0$, and by (B1) there exists $y$ such that $e_0 x =y$. If $e_1 x \neq y$, by (B5) there exists $z$ that $z=e_0 e_1 x=e_1 (e_0(x))$. This can not happen either, as in this case we would have $\downf_1(e_0 (x))\neq 0$,$\Delta_0 \upe_1(e_0 (x))=2$ implying $\downf_1(x)\geq 1$.

  So, we must have $e_0(x)=e_1(x)$. If $k=1$, we are done. If $k>1$, consider $z=e_1 (e_1 (x))$. By (B1), either $e_0(z)$ or $f_0(z)$ exists. If $e_0(z)$ existed, by (B6) we would have $\Delta_0  \upe_1(e_0(x))=2$, $\Delta_0  \downf_1(e_0(x))=0$, contradicting the maximality of the $1$ string. Then $f_0(z)$ exists and is not equal to $f_1(x)$ as $0$ strings have length $1$. $f_0(z)$ satisfies $\downf_1(f_0(x))=k-2$ and $\upe_1(f_0(z))=0$, so it is on a $1$ string of size $k-2$, and the strings are connected as shown in the diagram by (B5).
\end{proof}

\begin{lemma}
  Connected components of the subgraph generated by $f_0$ and $f_2$ of a regular queer graph are of the form shown in Figure~\ref{fig:queer0-2}, with optional $f_0$ edge presented by dashed lines.
  \label{lem:queer0-2}
\end{lemma}

\begin{figure}[ht]
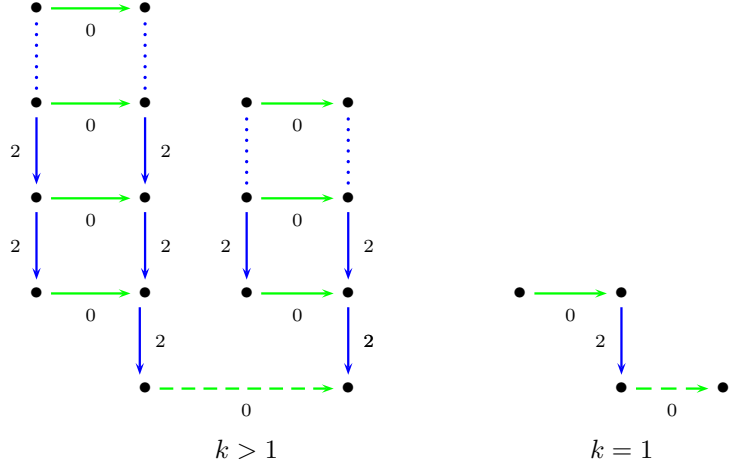

  \begin{displaymath}
    \begin{array}{c@{\hskip3\cellsize}c@{\hskip2\cellsize}c@{\hskip2\cellsize}c@{\hskip2\cellsize}c@{\hskip3\cellsize}c@{\hskip2\cellsize} c@{\hskip2\cellsize}c@{\hskip2\cellsize}}
      \rnode{a1}{\bullet} & \rnode{b1}{\bullet} &  &  &  &  & & \\[2\cellsize]
      \rnode{a2}{\bullet} & \rnode{b2}{\bullet} & \rnode{c2}{\bullet}  & \rnode{d2}{\bullet} &  & & & \\[2\cellsize]
      %\rnode{a3}{\bullet} & \rnode{b3}{\bullet} & \rnode{c3}{\bullet}  & \rnode{d3}{\bullet} &  & & &  \\[2\cellsize]
      \rnode{a4}{\bullet} & \rnode{b4}{\bullet} & \rnode{c4}{\bullet}  & \rnode{d4}{\bullet} &  & & &  \\[2\cellsize]
      \rnode{a5}{\bullet} & \rnode{b5}{\bullet}  & \rnode{c5}{\bullet}  & \rnode{d5}{\bullet} &  &\rnode{e5}{\bullet} &\rnode{f5}{\bullet} &  \\[2\cellsize] 
      & \rnode{b6}{\bullet}  &  & \rnode{d6}{\bullet} &  & &\rnode{f6}{\bullet} &\rnode{g6}{\bullet} \\[1\cellsize] 
      &  & k>1  &  &  &  & k = 1 
    \end{array} 
    \psset{linewidth=.2ex,nodesep=3pt}
    \everypsbox{\scriptstyle}
    \ncline[linecolor=blue,linestyle=dotted,linewidth=.3ex] {a1}{a2}
    \ncline[linecolor=blue,linestyle=dotted,linewidth=.3ex] {b1}{b2}
    \ncline[linecolor=blue,linestyle=dotted,linewidth=.3ex] {c2}{c4}
    \ncline[linecolor=blue,linestyle=dotted,linewidth=.3ex] {d2}{d4}
    \ncline[linecolor=blue]{->}              {a2}{a4} \nbput{2}
    %\ncline[linecolor=blue]{->}              {a3}{a4} \nbput{2}
    \ncline[linecolor=blue]{->}              {a4}{a5} \nbput{2}
    %\ncline[linecolor=blue]{->}              {c3}{c4} \nbput{2}
    \ncline[linecolor=blue]{->}              {c4}{c5} \nbput{2}
    \ncline[linecolor=blue]{->}              {d5}{d6} \naput{2}
    \ncline[linecolor=blue]{->}              {d3}{d4} \naput{2}
    \ncline[linecolor=blue]{->}              {d4}{d5} \naput{2}
    \ncline[linecolor=blue]{->}              {b2}{b4} \naput{2}
    %\ncline[linecolor=blue]{->}              {b3}{b4} \naput{2}
    \ncline[linecolor=blue]{->}              {b4}{b5} \naput{2}
    \ncline[linecolor=green]{->}              {a5}{b5} \nbput{0}
    \ncline[linecolor=green]{->}              {a1}{b1} \nbput{0}
    \ncline[linecolor=green]{->}              {a2}{b2} \nbput{0}
    %\ncline[linecolor=green]{->}              {a3}{b3} \nbput{0}
    \ncline[linecolor=green]{->}              {a4}{b4} \nbput{0}
    \ncline[linecolor=green]{->}              {c5}{d5} \nbput{0}
    \ncline[linecolor=green]{->}              {c2}{d2} \nbput{0}
    %\ncline[linecolor=green]{->}              {c3}{d3} \nbput{0}
    \ncline[linecolor=green]{->}              {c4}{d4} \nbput{0}
    \ncline[linecolor=green,linestyle=dashed,linewidth=.2ex]{->}              {b6}{d6} \nbput{0}
    \ncline[linecolor=green,linestyle=dashed,linewidth=.2ex]{->}              {f6}{g6} \nbput{0}
    \ncline[linecolor=green]{->}              {e5}{f5} \nbput{0}
    \ncline[linecolor=blue]{->}              {f5}{f6} \nbput{2}
    \ncline[offset=2pt,linecolor=blue]{<-}  {b6}{b5} \nbput{2}      
  \end{displaymath}
  \caption{\label{fig:queer0-2}Possible connected components for $f_0$ and $f_2$ in a regular queer graph.}
\end{figure}

\begin{proof} For a connected component, let $x$ be on a maximal $2$ string with $\upe_2(x)=0$, $\downf_2(x)=k> 0$. By (B1), either $f_0(x)$ or $e_0(x)$ is non-zero. The first is not possible, as $z=f_0(x)$ and  $\downf_2(z)\leq k$ imply $\Delta_0\upe_2(z)=\upe_2(z)=-1$, which is impossible. So we must have some $z=e_0(x)$. By (B6), $\Delta_0\downf_2(x)=-1$, so that $\downf_2(y)=k-1$. (B4) implies that all the edges on the $2$ string of $y$ will commute with $0$ edges.

As any $0$ string is of length $1$, if $w:=f_2^k(x)$ is not connected to any $0$ edges, we are done. There can not be an edge $e_0(w)$ by maximality of $k$. If there is a vertex $f_0(w)$, then by maximality of $k$, $\upe_2(f_0(w))=k-1$. As $f_0(f_0(w))=0$, if $k>1$, the $2$ string of $f_0(w)$ needs to be connected to a $k-2$ string by $0$ edges that commute with $2$ edges, completing the connected component.
\end{proof}

Combining Lemmas~\ref{lem:queer0-1} and \ref{lem:queer0-2}, we have the following simple characterization of regular queer graphs.

\begin{corollary}
  A regular crystal with $0$ strings of length $1$ is a regular queer graph if and only if connected $0-1$ components are characterized as in Figure \ref{fig:queer0-1}, connected $0-2$ components are characterized as in Figure \ref{fig:queer0-2} and $0$ edges commute with $i$ strings for $i>2$.
\end{corollary}

Finally, we believe that the converse of Theorem~\ref{thm:structure-crystal-Q} holds as well, as evidenced by the following.

\begin{proposition}
  Every regular queer graph of degree $3$ is a normal queer crystal.
  \label{prop:dim3}
\end{proposition}

\begin{proof}
  Let $v$ be a highest weight vertex (a priori not necessarily unique) of a connected, regular queer graph of degree $3$. By axiom (B5), all edges $f_i$ for $i\geq 3$ commute with $f_0$, so it is enough to consider dimension $3$ as well. In this case, $v$ has two possible weights, $(2,1,0)$ or $(3,0,0)$.

  Consider first the case of $\wt(v) = (2,1,0)$. As every regular queer graph is a regular graph when $f_0$ is ignored, we must have $f_1$ and $f_2$ as shown on the left side of Figure \ref{fig:conj-21}. Note that $e_0(v)=0$ as it is a highest weight, and $\downf_1(v)=1$, so we must be in case $k=1$ of Figure~\ref{fig:queer0-1}. Therefore $f_0(v) = f_1(v)$. Similarly, as an $e_0$ move from a vertex of weight $(1,0,2)$ is not possible, that vertex also has an $f_0$ edge that commutes with the $f_1$ edge, as shown in the center of Figure~\ref{fig:conj-21}. By axiom (B5) of Definition~\ref{def:queer-regular}, the vertex with weight $(2,0,1)$ has an $f_0$ edge satisfying $f_0(f_2(v))=f_2(f_0(v))$. As this $f_0$ edge decreases the $1$-head by $2$, we must be in case $k=2$ of Figure~\ref{fig:queer0-1} which gives us the final $f_0$ edge as shown on the right side of Figure~\ref{fig:conj-21}. Note that as all $0$-strings are of length $1$, no more edges are possible. Therefore the graph is exactly the normal queer crystal with highest weight $(2,1,0)$ seen in Figure~\ref{fig:P21-B}.

 \begin{figure}[ht]
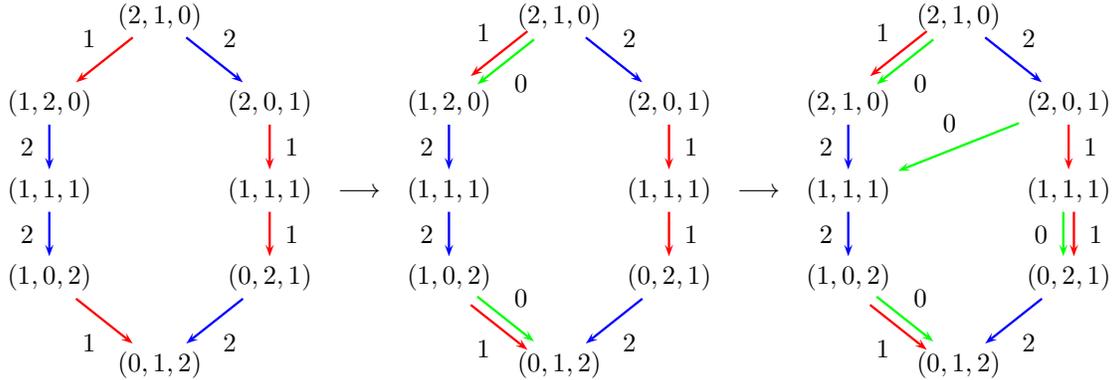

   \begin{center}
     \begin{displaymath}
       \begin{array}{ccccccccccc}  
         & \rnode{a2}{(2,1,0)} & & & &\rnode{a6}{(2,1,0)} & & & & \rnode{a11}{(2,1,0)} & \\[1.75\cellsize]
         \rnode{b1}{(1,2,0)} & & \rnode{b3}{(2,0,1)} & & \rnode{b5}{(1,2,0)} & & \rnode{b7}{(2,0,1)} & & \rnode{b10}{(2,1,0)} & & \rnode{b12}{(2,0,1)} \\[1.75\cellsize]
          \rnode{c1}{(1,1,1)} & & \rnode{c3}{(1,1,1)} &\rnode{c4}{{\longrightarrow}} & \rnode{c5}{(1,1,1)} & & \rnode{c7}{(1,1,1)} &\rnode{c8}{{\longrightarrow}} & \rnode{c11l}{(1,1,1)} & & \rnode{c11r}{(1,1,1)} \\[1.75\cellsize] 
         \rnode{d1}{(1,0,2)} & & \rnode{d3}{(0,2,1)} & \rnode{d4}{} & \rnode{d5}{ {(1,0,2)}} & & \rnode{d7}{ {(0,2,1)}} & \rnode{d8}{ } & \rnode{d9}{ {(1,0,2)}} & & \rnode{d13}{ {(0,2,1)}} \\[1.75\cellsize]
         & \rnode{e2}{(0,1,2)} & & & & \rnode{e6}{(0,1,2)} & & & & \rnode{e11}{(0,1,2)} & 
       \end{array}
       \psset{linewidth=.2ex,nodesep=3pt}
       % first
       \ncline[linecolor=blue]{->} {a2}{b3}  \naput{2}
       \ncline[linecolor=red,linewidth=.2ex]{->} {a2}{b1}  \nbput{1}
       \ncline[linecolor=red,linewidth=.2ex]{->} {b3}{c3} \naput{1}
       \ncline[linecolor=blue]{->} {b1}{c1} \nbput{2}
       \ncline[linecolor=blue]{->} {c1}{d1}  \nbput{2}
       \ncline[linecolor=red,linewidth=.2ex]{->} {c3}{d3} \naput{1}
       \ncline[linecolor=red,linewidth=.2ex]{->} {d1}{e2}  \nbput{1}
       \ncline[linecolor=blue]{->} {d3}{e2}  \naput{2}
       % second
       \ncline[linecolor=blue]{->} {a6}{b7}  \naput{2}
       \ncline[offset=-3pt,linecolor=red,linewidth=.2ex]{->} {a6}{b5}  \nbput{1}
       \ncline[linecolor=red,linewidth=.2ex]{->} {b7}{c7} \naput{1}
       \ncline[linecolor=blue]{->} {b5}{c5} \nbput{2}
       \ncline[linecolor=blue]{->} {c5}{d5}  \nbput{2}
       \ncline[linecolor=red,linewidth=.2ex]{->} {c7}{d7}  \naput{1}
       \ncline[offset=-3pt,linecolor=red,linewidth=.2ex]{->} {d5}{e6}  \nbput{1}
       \ncline[linecolor=blue]{->} {d7}{e6}  \naput{2}
       % third
       \ncline[linecolor=blue]{<-} {b12}{a11}  \nbput{2}
       \ncline[offset=-3pt,linecolor=red,linewidth=.2ex]{->} {a11}{b10}  \nbput{1}
       \ncline[linecolor=red,linewidth=.2ex]{->} {b12}{c11r} \naput{1}
       \ncline[linecolor=blue]{->} {b10}{c11l} \nbput{2}
       \ncline[linecolor=blue]{->} {c11l}{d9}  \nbput{2}
       \ncline[offset=2pt,linecolor=red,linewidth=.2ex]{->} {c11r}{d13} \naput{1}
       \ncline[offset=-3pt,linecolor=red,linewidth=.2ex]{->} {d9}{e11}   \nbput{1}
       \ncline[linecolor=blue]{->} {d13}{e11}  \naput{2}
       % zero edges
       \ncline[offset=-1pt, linecolor=green,linewidth=.2ex]{<-} {b5}{a6}  \nbput{0}
       \ncline[offset=-1pt, linecolor=green,linewidth=.2ex]{<-} {e6}{d5}  \nbput{0}
             \ncline[offset=-1pt, linecolor=green,linewidth=.2ex]{<-} {e11}{d9}  \nbput{0}
       \ncline[offset=-1pt,linecolor=green,linewidth=.2ex]{<-} {b10}{a11}  \nbput{0}
       \ncline[offset=2pt,linecolor=green,linewidth=.2ex]{<-} {d13}{c11r}  \naput{0}
       %\ncline[linecolor=green,linewidth=.2ex]{<-} {c5}{b7}  \naput{0}
       \ncline[linecolor=green,linewidth=.2ex]{<-} {c11l}{b12}  \naput{0}      
     \end{displaymath}
     \caption{\label{fig:conj-21} Constructing the unique regular queer graph with highest weight $(2,1,0)$, where the weights of the vertices are indicated.}
   \end{center}
 \end{figure}
 
 Consider last the case of $\wt(v) = (3,0,0)$. Again, since every regular queer graph is a regular graph when $f_0$ is ignored, we must have $f_1$ and $f_2$ as shown on the left side of Figure~\ref{fig:conj-3}. Since $e_0(v)=0$ and $\downf_1(v)=3$, we must be in case $k=3$ of Figure~\ref{fig:queer0-1}. Therefore we must have another vertex, say $w$, not on this component, of weight $(2,1,0)$ such that $f_0(v) = w$ and $e_1(w) = 0$. Since $w$ is on a regular crystal, it must be a highest weight, and so we have the vertices depicted on the right of Figure~\ref{fig:conj-3}. Applying Figure~\ref{fig:queer0-1}, we have $f_0(f_1(v)) = f_1(f_0(v)) = f_1(w)$ and $f_0 f_1^3 (v) = f_1 f_0 f_1^2(v)$. Applying Figure~\ref{fig:queer0-2} to force $f_0 f_2 = f_2 f_0$ whenever both are defined at a vertex results in the situation depicted on the right side of Figure~\ref{fig:conj-3}. 
 
 \begin{figure}[ht]
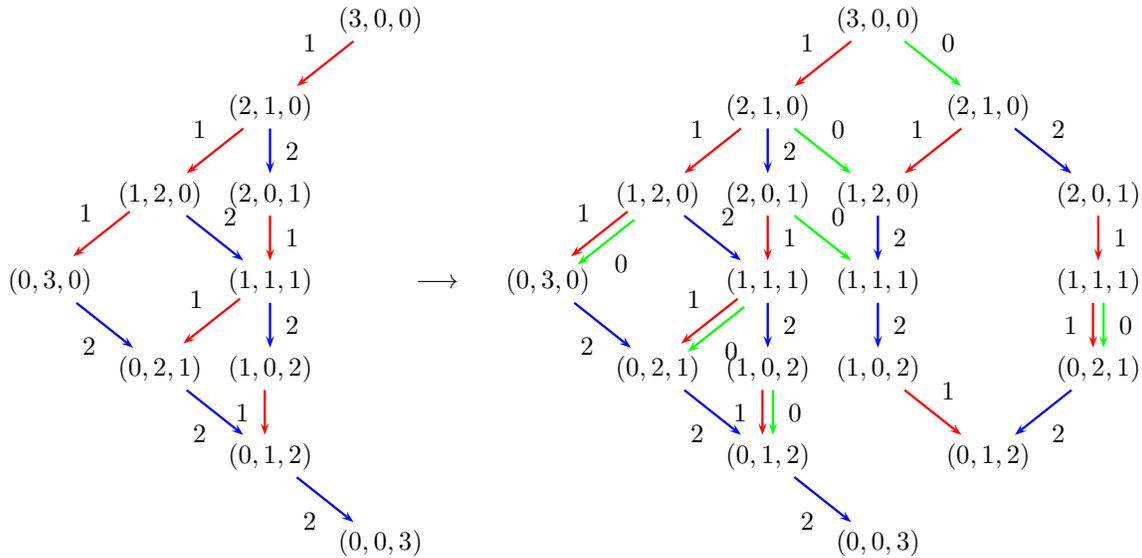

   \begin{center}
     \begin{displaymath}
       \begin{array}{cccc}
         & & & \rnode{d1}{(3,0,0)} \\[1.75\cellsize]
         & & \rnode{c2}{(2,1,0)} & \\[1.75\cellsize]
         & \rnode{b3}{(1,2,0)} & \rnode{c3}{(2,0,1)} & \\[1.75\cellsize]
         \rnode{a4}{(0,3,0)} & & \rnode{c4}{(1,1,1)} & \\[1.75\cellsize]
         & \rnode{b5}{(0,2,1)} & \rnode{c5}{(1,0,2)} & \\[1.75\cellsize]
         & & \rnode{c6}{(0,1,2)} & \\[1.75\cellsize]
         & & & \rnode{d7}{(0,0,3)} 
       \end{array}
       \hspace{-1em} \longrightarrow \hspace{1em}
       \begin{array}{cccccc}
         & & & \rnode{dD1}{(3,0,0)} & & \\[1.75\cellsize]
         & & \rnode{cC2}{(2,1,0)} & & \rnode{eE2}{(2,1,0)} & \\[1.75\cellsize]
         & \rnode{bB3}{(1,2,0)} & \rnode{cC3}{(2,0,1)} & \rnode{dD3}{(1,2,0)} & & \rnode{fF3}{(2,0,1)} \\[1.75\cellsize]
         \rnode{aA4}{(0,3,0)} & & \rnode{cC4}{(1,1,1)} & \rnode{dD4}{(1,1,1)} & & \rnode{fF4}{(1,1,1)} \\[1.75\cellsize]
         & \rnode{bB5}{(0,2,1)} & \rnode{cC5}{(1,0,2)} & \rnode{dD5}{(1,0,2)} & & \rnode{fF5}{(0,2,1)} \\[1.75\cellsize]
         & & \rnode{cC6}{(0,1,2)} & & \rnode{eE6}{(0,1,2)} & \\[1.75\cellsize]
         & & & \rnode{dD7}{(0,0,3)} & & 
       \end{array}
       \psset{linewidth=.2ex,nodesep=3pt}      
       \ncline[linecolor=red]{->}  {d1}{c2}  \nbput{1}
       \ncline[linecolor=red]{->}  {c2}{b3}  \nbput{1}
       \ncline[linecolor=blue]{->} {c2}{c3}  \naput{2}
       \ncline[offset=2pt,linecolor=red]{<-}  {a4}{b3}  \naput{1}
       \ncline[linecolor=blue]{->} {b3}{c4}  \naput{2}
       \ncline[linecolor=red]{->}  {c3}{c4}  \naput{1}
       \ncline[linecolor=blue]{->} {a4}{b5}  \nbput{2}
       \ncline[offset=2pt,linecolor=red]{<-}  {b5}{c4}  \naput{1}
       \ncline[linecolor=blue]{->} {c4}{c5}  \naput{2}
       \ncline[linecolor=blue]{->} {b5}{c6}  \nbput{2}
       \ncline[offset=2pt,linecolor=red]{<-}  {c6}{c5}  \naput{1}
       \ncline[linecolor=blue]{->} {c6}{d7}  \nbput{2}
       \ncline[linecolor=red]{->}  {dD1}{cC2}  \nbput{1}
       \ncline[linecolor=green]{->}{dD1}{eE2}  \naput{0}
       \ncline[linecolor=red]{->}  {cC2}{bB3}  \nbput{1}
       \ncline[linecolor=blue]{->} {cC2}{cC3}  \naput{2}
       \ncline[linecolor=green]{->}{cC2}{dD3}  \naput{0}
       \ncline[linecolor=red]{->}  {eE2}{dD3}  \nbput{1}
       \ncline[linecolor=blue]{->} {eE2}{fF3}  \naput{2}
       \ncline[offset=2pt,linecolor=red]{<-}  {aA4}{bB3}  \naput{1}
       \ncline[offset=2pt,linecolor=green]{->}{bB3}{aA4}  \naput{0}
       \ncline[linecolor=blue]{->} {bB3}{cC4}  \naput{2}
       \ncline[linecolor=red]{->}  {cC3}{cC4}  \naput{1}
       \ncline[linecolor=green]{->}{cC3}{dD4}  \naput{0}
       \ncline[linecolor=blue]{->} {dD3}{dD4}  \naput{2}
       \ncline[linecolor=red]{->}  {fF3}{fF4}  \naput{1}
       \ncline[linecolor=blue]{->} {aA4}{bB5}  \nbput{2}
       \ncline[offset=2pt,linecolor=red]{<-}  {bB5}{cC4}  \naput{1}
       \ncline[offset=2pt,linecolor=green]{->}{cC4}{bB5}  \naput{0}
       \ncline[linecolor=blue]{->} {cC4}{cC5}  \naput{2}
       \ncline[linecolor=blue]{->} {dD4}{dD5}  \naput{2}
       \ncline[offset=2pt,linecolor=red]{<-}  {fF5}{fF4}  \naput{1}
       \ncline[offset=2pt,linecolor=green]{->}{fF4}{fF5}  \naput{0}
       \ncline[linecolor=blue]{->} {bB5}{cC6}  \nbput{2}
       \ncline[offset=2pt,linecolor=red]{<-}  {cC6}{cC5}  \naput{1}
       \ncline[offset=2pt,linecolor=green]{->}{cC5}{cC6}  \naput{0}
       \ncline[linecolor=red]{->}  {dD5}{eE6}  \naput{1}
       \ncline[linecolor=blue]{->} {fF5}{eE6}  \naput{2}
       \ncline[linecolor=blue]{->} {cC6}{dD7}  \nbput{2}
       %
  %     \ncline[linecolor=red]{->}  {D1}{C2}  \nbput{1}
  %     \ncline[linecolor=green]{->}{D1}{E2}  \naput{0}
  %     \ncline[linecolor=red]{->}  {C2}{B3}  \nbput{1}
  %     \ncline[linecolor=blue]{->} {C2}{C3}  \naput{2}
  %     \ncline[linecolor=green]{->}{C2}{D3}  \naput{0}
  %     \ncline[linecolor=red]{->}  {E2}{D3}  \nbput{1}
  %     \ncline[linecolor=blue]{->} {E2}{F3}  \naput{2}
  %     \ncline[offset=2pt,linecolor=red]{<-}  {A4}{B3}  \naput{1}
  %     \ncline[offset=2pt,linecolor=green]{->}{B3}{A4}  \naput{0}
  %     \ncline[linecolor=blue]{->} {B3}{C4}  \naput{2}
  %     \ncline[linecolor=red]{->}  {C3}{C4}  \naput{1}
  %     \ncline[linecolor=green]{->}{C3}{D4}  \naput{0}
  %     \ncline[linecolor=blue]{->} {D3}{D4}  \naput{2}
  %     \ncline[linecolor=red]{->}  {F3}{F4}  \naput{1}
  %     \ncline[linecolor=green]{->}{F3}{G4}  \naput{0}
  %     \ncline[linecolor=blue]{->} {A4}{B5}  \naput{2}
  %     \ncline[offset=2pt,linecolor=red]{<-}  {B5}{C4}  \naput{1}
  %     \ncline[offset=2pt,linecolor=green]{->}{C4}{B5}  \naput{0}
  %     \ncline[linecolor=blue]{->} {C4}{C5}  \naput{2}
  %     \ncline[linecolor=blue]{->} {D4}{D5}  \naput{2}
  %     \ncline[offset=2pt,linecolor=red]{<-}  {F5}{F4}  \naput{1}
  %     \ncline[linecolor=blue]{->} {B5}{C6}  \naput{2}
  %     \ncline[offset=2pt,linecolor=red]{<-}  {C6}{C5}  \naput{1}
  %     \ncline[offset=2pt,linecolor=green]{->}{C5}{C6}  \naput{0}
  %     \ncline[linecolor=red]{->}  {D5}{E6}  \naput{1}
  %     \ncline[linecolor=blue]{->} {F5}{E6}  \naput{2}
  %     \ncline[linecolor=blue]{->} {C6}{D7}  \naput{2}
     \end{displaymath}
     \caption{\label{fig:conj-3} Constructing the unique regular queer graph with highest weight $(3,0,0)$, where the weights of the vertices are indicated.}
   \end{center}
 \end{figure}

 By axiom (B5), we must have $e_0(f_2(w)) = 0$. Since $\downf_1(f_2(w))=2$, we must be in case $k=2$ of Figure~\ref{fig:queer0-1}, and so we must have another vertex, say $x$, not yet in the picture, of weight $(1,1,1)$ such that $f_0(f_2(w)) = x$ and $e_1(x) = 0$. By Figure~\ref{fig:queer0-2}, we must also have $e_2(x)=0$. Therefore, since $x$ is on a regular crystal of dimension $3$, it is a highest weight, and so $f_1(x) = f_2(x) = 0$. This completes the picture, and we have a graph isomorphic to the normal queer crystal with highest weight $(3,0,0)$ seen in Figure~\ref{fig:P21-B}. 
\end{proof}

As noted in the introduction, our original announcement of our constructions and results \cite{AO18} conjectured that every queer regular graph is a normal queer crystal. However, this is not the case. The primary difference between our characterization and that of Stembridge \cite{Ste03}, is that we do not give explicit conditions for when the potential $0$-edge (dashed in Figure~\ref{fig:queer0-2}) is present or not for the $f_0, f_2$ components of the queer crystal. We believe that making this precise will lead to a full characterization of normal queer crystals. As demonstrated in our use of Stembridge's axioms to prove our shifted crystal operators form a crystal, such a characterization will provide a powerful tool in the study of Schur $P$-positive polynomials. 

%%%%%%%%%%%%%%%%%%%%%%%%%%%%%%%%%%%%%%%%%%%%%%%%%%%%%%%%%%%%
%
%  Bibliography
%
%%%%%%%%%%%%%%%%%%%%%%%%%%%%%%%%%%%%%%%%%%%%%%%%%%%%%%%%%%%%

\bibliographystyle{amsalpha} 
\bibliography{bohnert}

\end{document}